\documentclass[reqno,12pt]{amsart}
\usepackage[utf8]{inputenc}
\usepackage[left=1.15in,right=1.15in,top=1in,bottom=1in]{geometry}
\usepackage{amsmath,amssymb,amsthm,mathrsfs,mathtools}
\usepackage{xcolor}
\usepackage{cite}
\usepackage[colorlinks=true,linkcolor=blue,citecolor=red,urlcolor=blue]{hyperref}
\usepackage{enumitem}
\numberwithin{equation}{section}
\allowdisplaybreaks[2]

\theoremstyle{plain}
\newtheorem{theorem}{Theorem}[section]
\newtheorem{lemma}[theorem]{Lemma}
\newtheorem{proposition}[theorem]{Proposition}
\theoremstyle{definition}
\newtheorem{definition}[theorem]{Definition}

\newcommand{\Rn}{\mathbb R^n}
\newcommand{\Sn}{\mathbb S^n}
\newcommand{\tr}{\operatorname{tr}}

\newcommand{\epsN}{\epsilon_N}
\newcommand{\muN}{\mu_N}
\newcommand{\rhoN}{\rho_N}

\newcommand{\F}{\mathcal F}

\begin{document}

\title[Blow-up for the $\sigma_2$-Yamabe equation]{Blow-up phenomena for the $\sigma_2$-Yamabe equation}
\author{Bin Deng}
\address{Department of Mathematics and Statistics, Wuhan University,
Wuhan 430072, Hubei, P.R. China}
\email{dbmath@whu.edu.cn}
\author{Han Lu}
\address{School of Mathematics and Statistics, Henan University,
Kaifeng 475004, China}
\email{hlu@henu.edu.cn}
\author{Jun-cheng Wei}
\address{Department of Mathematics, Chinese University of Hong Kong,
Shatin, NT, Hong Kong}
\email{wei@math.cuhk.edu.hk}
\date{July 12, 2026}

\begin{abstract}
For every integer $n\ge27$, we construct a smooth metric on $\mathbb S^n$
that is invariant under the antipodal map and is not locally conformally
flat.  For this fixed background metric, the normalized $\sigma_2$-Yamabe
equation admits a noncompact family of positive $\Gamma_2^+$-admissible
solutions.  The main difficulty is the possible loss of ellipticity of the
linearized operator.  This difficulty does not occur for the scalar Yamabe
equation, whose linearization has a fixed Laplace-type principal part.  In
the $\sigma_2$ problem, the positive Newton tensor of a standard bubble
decays in the far field, while the terms produced by the background metric
need not decay at the same rate.  Our construction provides the relative
decay needed to keep the conformal metrics inside the ellipticity cone.
A quartic profile in the finite-dimensional reduction yields the endpoint
dimension $n=27$.
\end{abstract}


\maketitle

\tableofcontents

\section{Introduction}

Let $(M^n,g)$ be a compact Riemannian manifold, $n\ge5$, and let
\[
 A_g=\frac1{n-2}\left(\operatorname{Ric}_g
 -\frac{R_g}{2(n-1)}g\right)
\]
be its Schouten tensor.  If $\mu(A_g)=(\mu_1,\ldots,\mu_n)$ denotes the
eigenvalues of $A_g$ with respect to $g$, then
\[
 \sigma_k(g^{-1}A_g)=\sum_{i_1<\cdots<i_k}
 \mu_{i_1}\cdots\mu_{i_k}.
\]
The $\sigma_k$-Yamabe problem asks for a conformal metric of constant
$\sigma_k$-curvature whose Schouten eigenvalues remain in the admissible
cone $\Gamma_k^+$.  For $k=2$ this is a genuinely fully nonlinear equation:
the coefficients of its linearization are given by the first Newton tensor
of the unknown conformal metric, and ellipticity is available only as long
as the solution stays in $\Gamma_2^+$.

The existence and regularity theory for fully nonlinear Yamabe-type
equations was developed in several settings by Chang--Gursky--Yang,
Guan--Wang, Gursky--Viaclovsky, Li--Li, Ge--Wang, and
Sheng--Trudinger--Wang; see
\cite{CGY,GuanWang,GurskyViaclovsky,LiLi1,LiLi2,GeWang,STW}. Compactness results for fully nonlinear conformal $\sigma_k$-problems
have also been obtained for $k\geq n/2$; see \cite{GurskyViaclovskyVolume,LiNguyenCompactness,LiNguyenWangNirenberg} and the references therein.
The case $k=2$ has an additional feature essential here: the total
$\sigma_2$-curvature has a conformal variational structure without a local
conformal flatness assumption \cite{Viaclovsky}.  This makes it possible to
solve the final finite-dimensional equations by locating a critical point
of a reduced energy. 

The corresponding compactness question is classical in conformal
geometry.  For the Yamabe equation, compactness away from the round sphere
in dimensions at most $24$ follows from the work of Druet, Li--Zhang,
Khuri--Marques--Schoen, Li--Zhu, and Marques
\cite{DruetMulti,DruetCompact,LiZhang1,LiZhang2,LiZhang3,KMS,LiZhu,Marques}.
Brendle and Brendle--Marques proved noncompactness in high dimensions by
constructing localized Weyl-quadratic perturbations for which the reduced
energy has a stable critical point \cite{Brendle,BrendleMarques}.  Related
compactness and noncompactness phenomena occur for the $Q$-curvature,
fractional Yamabe, spinorial Yamabe, and CR Yamabe problems; see, for
example, \cite{WeiZhao,LiXiong,KMW,AmmannHumbertMorel,JerisonLee}.

The fully nonlinear problem has a difficulty absent from these fixed
leading-operator models.  To see it, work in Euclidean coordinates and put a
standard bubble
\[
 U_{\lambda,0}(x)
 =\left(\frac{2\lambda}{\lambda^2+r^2}\right)^{\frac{n-4}{4}},
 \qquad r=|x|,
\]
on a background metric $e^h\delta$.  Set
\[
 p=\frac{r^2}{\lambda^2+r^2},
 \qquad
 q=p(1-p)=\frac{\lambda^2r^2}{(\lambda^2+r^2)^2}.
\]
For the flat background, the Euclidean-coordinate matrix of the first Newton
tensor of $U_{\lambda,0}^{8/(n-4)}\delta$ is
\[
 \frac{2(n-1)q}{r^2}I,
 \qquad\text{where}\qquad
 \frac q{r^2}=\frac{\lambda^2}{(\lambda^2+r^2)^2}.
\]
It is positive, but its size is only $O(\lambda^2r^{-4})$ in the far
field.  On the perturbed background, the change in the Schouten endomorphism,
and hence in the coefficients of the linearized operator, is measured
relative to the bubble Newton tensor by
\[
 \Xi_h(r):=\frac{r^2}{q}
 \left(|D^2h|+|Dh|^2+\frac pr|Dh|+\frac p{r^2}|h|\right).
\]
Smallness of the ordinary $C^2$ norm of $h$ does not imply
$\Xi_h=o(1)$, because $q$ tends to zero in the far field.  If $\Xi_h$
is not small, the background
terms can be as large as the positive bubble Newton tensor, and ellipticity
of the linearized operator is no longer guaranteed.  Thus multiplying a
standard bubble by a fixed background metric is not enough for the present
construction.  We instead use a symmetric two-ended profile and prove decay
estimates for $h$ relative to the local Newton scale throughout the
construction.

For a positive function $u$ on a compact manifold $(M^n,g)$, set
$g_u=u^{8/(n-4)}g$ and
\[
 \mathcal N_g(u)=u^{\frac{4n}{n-4}}
 \left(\sigma_2(g_u^{-1}A_{g_u})-\frac14\binom n2\right).
\]
We also use the scale-invariant $\sigma_2$-Yamabe quotient
\begin{equation}\label{eq:quotient-definition-step24}
 \mathcal Q_g(u)
 :=\frac{\displaystyle\int_M
       \sigma_2(g_u^{-1}A_{g_u})\,dv_{g_u}}
      {\displaystyle\operatorname{Vol}(g_u)^{(n-4)/n}},
 \qquad
 \mathcal Y_2(\mathbb S^n)
 :=\frac14\binom n2|\mathbb S^n|^{4/n}.
\end{equation}
Let $J:\mathbb S^n\to\mathbb S^n$, $J(p)=-p$, denote the antipodal map.
The main result is the following.

\begin{theorem}[Two-bubble noncompactness]\label{thm:target}
Assume that $n\ge 27$.  Then there exist a smooth Riemannian metric $g_0$
on $\mathbb S^n$ and a sequence of positive functions
$v_N\in C^\infty(\mathbb S^n)$, $N\in\mathbb N$, with the following
properties:
\begin{enumerate}[label=\textup{(\roman*)}]
 \item the metric $g_0$ is invariant under the antipodal map $J$ and is not
 locally conformally flat;
 \item each $v_N$ is $J$-invariant, the conformal metric
 $g_{v_N}=v_N^{8/(n-4)}g_0$ is $\Gamma_2^+$-admissible, and
 \[
  \mathcal N_{g_0}(v_N)=0 \qquad\text{on }\mathbb S^n;
 \]
 \item for every $N$,
 \[
  \mathcal Q_{g_0}(v_N)<2^{4/n}\mathcal Y_2(\mathbb S^n),
  \qquad
  \lim_{N\to\infty}\mathcal Q_{g_0}(v_N)
  =2^{4/n}\mathcal Y_2(\mathbb S^n);
 \]
 \item
 \[
  \lim_{N\to\infty}\sup_{\mathbb S^n}v_N=\infty.
 \]
\end{enumerate}
\end{theorem}

The form of the blow-up is an important part of the result.  The solutions develop two peaks of equal height, interchanged by $J$, and
their combined energy approaches the two-bubble level in
Theorem~\ref{thm:target}(iii).

We now describe the geometry of the construction.  Choose antipodal points
$p_-,p_+\in\mathbb S^n$.  Around a sequence of points approaching $p_-$ we
place mutually disjoint Weyl-quadratic perturbations of the round metric.
Their antipodal images form a second sequence approaching $p_+$.  The resulting background perturbation $h$ is fixed once
and for all and is invariant under the antipodal map $J$.  For the $N$-th
solution, we select the $N$-th pair of metric perturbations and center an
exact round bubble near each member of this pair.  We call this pair the
\emph{principal pair}; all pairs with index $M\ne N$ are the
\emph{remaining perturbations}.  In Euclidean stereographic coordinates,
the left bubble is matched to
a translated Schwarzschild factor; its antipodal Kelvin image gives the
right bubble, and the two ends are assembled through a common centered
Schwarzschild region.  From the spherical point of view, the picture is
therefore a centrally symmetric perturbation of the round background,
together with two equal bubbles at symmetric locations connected by a
smooth Schwarzschild bridge.

The analysis relies on two features.  The first is the
\emph{symmetric bubble--Schwarzschild gluing}.  The bubble is kept exact in
each core, while the transition is measured relative to the Schwarzschild
Newton tensor and the middle connector is exactly Kelvin invariant.  This
provides a positive Newton scale throughout the bridge and prevents loss of
ellipticity.  The second is the \emph{sharp decay of
the metric perturbation}.  For the quartic profile, in coordinates centered
at the principal point $x_N$, one has
\[
 |D^j h_N(x_N+y)|\le C\mu_N(\epsilon_N+|y|)^{10-j},
 \qquad 0\le j\le3.
\]
These precise bounds make the metric perturbation small relative to the bubble Newton
tensor and place the transition, tail, and interaction terms below the leading
reduced-energy scale.  They are used both to preserve admissibility and to
control the energy remainder.

The background metric is modeled on the localized Weyl-quadratic
perturbations introduced in the Yamabe noncompactness constructions of
Brendle and Brendle--Marques \cite{Brendle,BrendleMarques}.  Locally, both
constructions start from an algebraic Weyl tensor and a radial polynomial
profile, and the profile determines the leading reduced energy.  Our global
construction differs in two respects.  Every perturbation is
paired with its exact antipodal image, and the summable collection of all
pairs defines a single smooth $J$-invariant background metric before any
bubble is selected.  In addition, we use scale-dependent estimates through
third order: quadratic vanishing at the centers of the principal
perturbations and the corresponding bounds at the opposite end obtained by
Kelvin transformation.  These estimates
are needed because the relevant comparison is with the degenerating Newton
tensor of the approximate solution, rather than with a fixed uniformly
elliptic operator.  Thus the local Brendle-type reduced-energy calculation
remains valid, while the paired construction gives the symmetry and decay
required for two simultaneous bubbles.

The Schwarzschild bridge also has a different role from the connected-sum
construction of Catino and Mazzieri \cite{CatinoMazzieri}.  Their theorem
starts from two given compact \emph{nondegenerate} positive solutions of the
$\sigma_k$-Yamabe problem and constructs constant-$\sigma_k$ metrics on the
connected sum $M_1\#M_2$. The standard round metric on \(\mathbb S^n\) does not satisfy this nondegeneracy condition. Here we do not begin with two background solutions,
and the topology is not changed: the manifold $\mathbb S^n$, the smooth metric
$g_0$, and the conformal class $[g_0]$ are fixed once and for all, while the
two ends are round bubbles concentrating simultaneously inside a single
solution and related by the fixed antipodal involution.  Correspondingly,
we impose no nondegeneracy hypothesis on a pre-existing compact
$\sigma_2$-Yamabe metric.  The explicit dilation and translation Jacobi
fields of the round bubble are instead handled by localized projections and
the finite-dimensional reduction.  Thus the Schwarzschild factor serves as an internal analytic connector
that preserves the Newton scale between the two peaks, rather than as a neck
that changes the underlying manifold.  These are two different notions: we make no
nondegeneracy assumption on a given background solution, but we prove that
the reduced functional has a stable nondegenerate critical point.

More precisely, fix a selected index $N$, put
$\lambda=\epsilon_N\lambda'$ and $c_N(\xi)=x_N^-+\xi$, and let
$\widetilde U^-_{N,\lambda,\xi}$ denote the one-ended matched bubble
profile that equals the exact bubble $U_{\lambda,c_N(\xi)}$ in its core
and is joined to
\[
 S_{\lambda,c_N(\xi)}(x)
 =(2\lambda)^a\bigl(1+|x-c_N(\xi)|^{-a_\Sigma}\bigr).
\]
Its right-end partner is the exact Kelvin image
$\widetilde U^+_{N,\lambda,\xi}=\mathcal K_A
\widetilde U^-_{N,\lambda,\xi}$.  With paired cutoffs
$\chi_+=\chi_-\circ\iota_A$, $\chi_0=1-\chi_--\chi_+$, and
$S_\lambda(x)=(2\lambda)^a(1+|x|^{-a_\Sigma})$, the global approximate
factor is
\[
 \log \bar u_{N,\lambda,\xi}
 =\chi_-\log\widetilde U^-_{N,\lambda,\xi}
  +\chi_+\log\widetilde U^+_{N,\lambda,\xi}
  +\chi_0\log S_\lambda,
 \qquad
 \bar v_{N,\lambda,\xi}
 =U_{1,0}^{-1}\bar u_{N,\lambda,\xi}.
\]
It satisfies $\mathcal K_A\bar u=\bar u$ and
$\bar v\circ J=\bar v$ exactly.

The localized reduced energy at either end is governed by the same
one-bubble functional $\mathcal F$, and hence by the same coefficient
functions $P_f,Q_f,R_f$.  Antipodal symmetry makes the two principal
contributions identical, while all cross-interactions are lower order:
\[
 \mathscr J_N=2\mathscr E_2(\mathbb S^n)
 +2s_N\mathcal F+o_{C^1}(s_N).
\]
Thus the unnormalized coefficients are twice the one-bubble coefficients,
and, after division by $2s_N$, the finite-dimensional problem coincides
with the one-bubble problem.  Appendix~A gives the full formulas for
$P_f,Q_f,R_f$ and the beta-integral rule used for the quartic profile.
An explicit algebraic calculation gives a nondegenerate minimum with
negative value in every dimension $n\ge27$.

The paper is organized as follows.  Section~2 constructs the paired
background metric and proves its sharp decay.  Section~3 builds the matched bubble--Schwarzschild--bubble
profile.  Section~4 proves the global residual and admissibility estimates,
and Sections~5 and~6 establish the projected linear and nonlinear theories.
Section~7 localizes and expands the reduced energy.  Section~8 solves the
finite-dimensional problem and completes the construction.

\section{Background geometry}

Let $p_-,p_+\in\Sn$ be antipodal and identify
$\Sn\setminus\{p_+\}$ with $\Rn$ by stereographic projection, with $p_-=0$.
Then
\begin{equation}\label{eq:round-metric}
 g_{\rm rd}=\Omega(x)^2\,dx^2,
 \qquad \Omega(x)=\frac{2}{1+|x|^2}.
\end{equation}
Let $J:\Sn\to\Sn$ denote the antipodal map.  In the chosen chart it is represented by
\begin{equation}\label{eq:iotaA}
 \iota_A(x)=-\frac{x}{|x|^2}.
\end{equation}
Put
\begin{equation}\label{eq:R}
 R(x)=I-2\frac{x}{|x|}\otimes\frac{x}{|x|},
 \qquad D\iota_A(x)=-|x|^{-2}R(x).
\end{equation}
For a symmetric matrix field $b$, define
\begin{equation}\label{eq:K2A}
 (K_{2,A}b)(x)=R(x)b(\iota_A(x))R(x).
\end{equation}

Let $W_{ijkl}$ be a nonzero algebraic Weyl tensor on $\Rn$ and define
\begin{equation}\label{eq:Hseed}
 H_{ij}(y)=\sum_{p,q=1}^n W_{ipjq}y_py_q.
\end{equation}
Fix a polynomial $f$ of degree $m_f$, and assume
\begin{equation}\label{eq:dimension-beta}
 n>4m_f+8,
 \qquad \beta_f=\frac{n-4m_f-8}{4}>0.
\end{equation}
Put
\begin{equation}\label{eq:basic-exponents}
 a=\frac{n-4}{4},\qquad a_\Sigma=\frac{n-4}{2}=2a,
 \qquad d_f=2m_f+2.
\end{equation}
For $N\ge1$, set
\begin{equation}\label{eq:sequences}
 \epsN=2^{-N},\qquad \muN=\epsN^{\beta_f},\qquad
 \rhoN=N^{-3},\qquad x_N^-=N^{-1}e_1.
\end{equation}
Choose $\eta\in C_c^\infty([0,\infty))$ satisfying
$0\le\eta\le1$, $\eta=1$ on $[0,1]$, and $\eta=0$ on $[2,\infty)$, and set
\begin{equation}\label{eq:hminus}
 h^-_{N,ij}(x)
 =\eta\!\left(\rhoN^{-1}|x-x_N^-|\right)
 \muN\epsN^{2m_f}
 f\!\left(\epsN^{-2}|x-x_N^-|^2\right)
 H_{ij}(x-x_N^-).
\end{equation}
Define
\begin{equation}\label{eq:hplus-hsum}
 h_N^+=K_{2,A}h_N^-,
 \qquad
 h=\sum_{N=N_0+1}^\infty(h_N^-+h_N^+).
\end{equation}

It is easy to verify that the metric
\begin{equation}\label{eq:paired-metric}
 g(X,Y)=g_{\rm rd}(e^hX,Y)
\end{equation}
is smooth and satisfies $J^*g=g$.

For $N>N_0$, set
\begin{equation}\label{eq:principal-distances}
 r_N^-(x)=|x-x_N^-|,\qquad
 x_N^+=\iota_A(x_N^-)=-Ne_1,
\end{equation}
and, on the reflected support,
\begin{equation}\label{eq:right-scales}
 \epsilon_N^+=N^2\epsilon_N,\qquad
 \rho_N^+=N^2\rho_N=N^{-1},\qquad
 r_N^+(x)=N^2|\iota_A(x)-x_N^-|.
\end{equation}

\begin{proposition}[Sharp derivative estimates for the paired metric columns]\label{lem:step2}
After increasing $N_0$, the following estimates hold, with constants independent
of $N$.

\begin{enumerate}[label=\textup{(\roman*)}]
\item \emph{Left column.} For $0\le j\le3$ and
$x\in\operatorname{supp}h_N^-$,
\begin{equation}\label{eq:left-core-tail-unified}
 |D^jh_N^-(x)|
 \le C_j\mu_N\bigl(\epsilon_N+r_N^-(x)\bigr)^{d_f-j}.
\end{equation}
On the inner core $r_N^-\le\epsilon_N$,
\begin{align}
 |h_N^-|+r_N^-|Dh_N^-|+(r_N^-)^2|D^2h_N^-|
 &\le C\mu_N\epsilon_N^{2m_f}(r_N^-)^2,
 \label{eq:left-quadratic-core}\\
 |D^3h_N^-|
 &\le C\mu_N\epsilon_N^{2m_f-2}r_N^-
 \quad (m_f\ge1),
 \label{eq:left-third-core}
\end{align}
while $D^3h_N^-=0$ on the uncut core when $m_f=0$.  In the polynomial
tail $\epsilon_N\le r_N^-\le2\rho_N$,
\begin{align}
 |D^jh_N^-|&\le C_j\mu_N(r_N^-)^{d_f-j},
 \qquad 0\le j\le3,
 \label{eq:left-tail}\\
 |D^2h_N^-|+(r_N^-)^{-1}|Dh_N^-|+(r_N^-)^{-2}|h_N^-|
 &\le C\mu_N(r_N^-)^{d_f-2}.
 \label{eq:left-newton-combination}
\end{align}
For every fixed integer $\ell\ge0$,
\begin{equation}\label{eq:left-global}
 \|D^\ell h_N^-\|_{L^\infty(\mathbb R^n)}
 \le C_\ell\mu_N\rho_N^{d_f-\ell}
 =C_\ell\mu_NN^{-3(d_f-\ell)}.
\end{equation}

\item \emph{Right column.} On $\operatorname{supp}h_N^+$,
\begin{equation}\label{eq:right-geometry}
 |x|\asymp N,\qquad r_N^+(x)\asymp|x-x_N^+|,
 \qquad 0\le r_N^+(x)\le C\rho_N^+.
\end{equation}
For $0\le j\le3$,
\begin{equation}\label{eq:right-core-tail-unified}
 |D^jh_N^+(x)|
 \le C_j\mu_NN^{-2d_f}
       \bigl(\epsilon_N^++r_N^+(x)\bigr)^{d_f-j}.
\end{equation}
On $r_N^+\le\epsilon_N^+$,
\begin{align}
 |h_N^+|+r_N^+|Dh_N^+|+(r_N^+)^2|D^2h_N^+|
 &\le C\mu_NN^{-2d_f}(\epsilon_N^+)^{2m_f}(r_N^+)^2,
 \label{eq:right-quadratic-core}\\
 |D^3h_N^+|
 &\le C\mu_NN^{-2d_f}
 \left[N^{-1}(\epsilon_N^+)^{2m_f}
 +\mathbf 1_{\{m_f\ge1\}}(\epsilon_N^+)^{2m_f-2}r_N^+\right].
 \label{eq:right-third-core}
\end{align}
In the reflected polynomial tail $\epsilon_N^+\le r_N^+\le C\rho_N^+$,
\begin{align}
 |D^jh_N^+|&\le C_j\mu_NN^{-2d_f}(r_N^+)^{d_f-j},
 \qquad 0\le j\le3,
 \label{eq:right-tail}\\
 |D^2h_N^+|+(r_N^+)^{-1}|Dh_N^+|+(r_N^+)^{-2}|h_N^+|
 &\le C\mu_NN^{-2d_f}(r_N^+)^{d_f-2}.
 \label{eq:right-newton-combination}
\end{align}
For every fixed integer $\ell\ge0$,
\begin{equation}\label{eq:right-global}
 \|D^\ell h_N^+\|_{L^\infty(\mathbb R^n)}
 \le C_\ell\mu_NN^{-3d_f+\ell}.
\end{equation}

\item For every smooth matrix field $b$ and every integer $j\ge0$,
\begin{equation}\label{eq:kelvin-derivative-general}
 |D^j(K_{2,A}b)(x)|
 \le C_j\sum_{\ell=0}^j|x|^{-j-\ell}|D^\ell b(\iota_A(x))|.
\end{equation}
Thus the right-column estimates are the Kelvin transforms of the left-column
estimates; the factor $N^{-2d_f}$ is the Euclidean-coordinate cost of this
transport.
\end{enumerate}
\end{proposition}

\begin{proof}
Write $y=x-x_N^-$ and $f(t)=\sum_{q=0}^{m_f}a_qt^q$.  Before the cutoff is
differentiated, the $q$th summand of $h_N^-$ is
$a_q\mu_N\epsilon_N^{2m_f-2q}|y|^{2q}H(y)$.  Since $H$ is homogeneous of
degree two, its $j$th derivatives are bounded by
\[
 C_j\mu_N\epsilon_N^{2m_f-2q}|y|^{2q+2-j}.
\]
For $|y|\le\epsilon_N$ this gives
\eqref{eq:left-quadratic-core}--\eqref{eq:left-third-core}; for
$\epsilon_N\le|y|\le\rho_N$ it gives \eqref{eq:left-tail}.  When derivatives
fall on the cutoff and $\rho_N\le|y|\le2\rho_N$, the same bound follows from
$\epsilon_N\le\rho_N$.  This proves
\eqref{eq:left-core-tail-unified}--\eqref{eq:left-newton-combination}, and the
same calculation with an arbitrary fixed number of derivatives yields
\eqref{eq:left-global}.

For the Kelvin transform, the identities
$|D^k\iota_A(x)|\le C_k|x|^{-k-1}$ and
$|D^kR(x)|\le C_k|x|^{-k}$, together with the product and chain rules for
$K_{2,A}b=R(b\circ\iota_A)R$, give
\eqref{eq:kelvin-derivative-general}.  On $\operatorname{supp}h_N^+$ one has
$|x|\asymp N$ and
$\epsilon_N+|\iota_A(x)-x_N^-|
 =N^{-2}(\epsilon_N^++r_N^+(x))$; the inversion distance identity also gives
\eqref{eq:right-geometry}.  Substituting the left-column bounds into
\eqref{eq:kelvin-derivative-general} proves the stated right-column core,
tail, and global estimates.
\end{proof}

\medskip
\noindent\textbf{Uniform estimate for the remaining perturbations.}
Choose an integer $P_*=P_*(n,m_f)$ large enough to dominate all polynomial
losses in the remaining-support estimates, the normalized coefficient
expansions, and their first normalized parameter derivatives.  Choose
$C_*$ accordingly and set
\begin{equation}\label{eq:delta-star-definition}
 \delta_{\!*}(N_0)
 :=C_*\sum_{M>N_0}\mu_M(1+M)^{P_*}.
\end{equation}
Because $\mu_M=2^{-\beta_fM}$ and $\beta_f>0$,
\begin{equation}\label{eq:delta-star-limit}
 \delta_{\!*}(N_0)\longrightarrow0
 \qquad\hbox{as }N_0\longrightarrow\infty.
\end{equation}
With this choice, $\delta_{\!*}(N_0)$ bounds the pointwise, cylindrical
$L^1$, operator, and parameter-derivative tails of the remaining
perturbations in Lemmas~\ref{lem:step3},~\ref{lem:step8},~\ref{lem:step9},
and Proposition~\ref{prop:summable-operator-perturbation-step15}.

\begin{lemma}[Estimate for the remaining perturbations at both ends]\label{lem:step3}
Fix a selected index $N>N_0$ and parameters
\begin{equation}\label{eq:principal-parameters-step3}
 \lambda=\epsilon_N\lambda',\qquad
 c_N(\xi)=x_N^-+\xi,\qquad
 \frac12\le\lambda'\le\frac32,\qquad |\xi|\le\epsilon_N.
\end{equation}
Write
\[
 h_{\mathrm{rem},N}=h-h_N^--h_N^+,
 \qquad r=|x-c_N(\xi)|.
\]
Thus $h_N^-+h_N^+$ is the principal pair and
$h_{\mathrm{rem},N}$ is the sum of the remaining perturbations.  The
subscripts $\mathrm{pr}$ and $\mathrm{rem}$ refer to these two parts.  Set
\begin{equation}\label{eq:Theta-two-ended}
 \Theta(r)=
 \begin{cases}
  r^{a_\Sigma-2},&0<r\le1,\\
  r^{-a_\Sigma-2},&r\ge1.
 \end{cases}
\end{equation}
For every sufficiently large $N_0$, uniformly in the selected index and in the
parameters in \eqref{eq:principal-parameters-step3},
\begin{equation}\label{eq:remaining-global-statement}
 \sum_{j=0}^2r^{j-2}|D^jh_{\mathrm{rem},N}(x)|
 \le \delta_{\!*}(N_0)\Theta(r)
 \qquad (x\ne c_N(\xi)).
\end{equation}
The left-hand side is identically zero in a neighborhood of $c_N(\xi)$.

The estimate is stable under one normalized parameter derivative in moving
Euclidean coordinates.  More precisely, for fixed $y$ put
$x=c_N(\xi)+\lambda y$, $r=\lambda|y|$, and let
\begin{equation}\label{eq:normalized-parameter-derivatives-step3}
 \mathscr D_0=\lambda\partial_\lambda,
 \qquad
 \mathscr D_k=\epsilon_N\partial_{\xi_k},\quad 1\le k\le n,
\end{equation}
where these derivatives act at fixed $y$.  Then, for every
$b=0,1,\dots,n$,
\begin{equation}\label{eq:remaining-parameter-statement}
 \sum_{j=0}^2r^{j-2}
 \left|
 \mathscr D_b\!\bigl[D^jh_{\mathrm{rem},N}
       (c_N(\xi)+\lambda y)\bigr]
 \right|
 \le \delta_{\!*}(N_0)\Theta(r).
\end{equation}

Finally, in the right Kelvin coordinate $z=\iota_A(x)$ the matrix
representative of the remaining tensor is again $h_{\mathrm{rem},N}(z)$.
Consequently, \eqref{eq:remaining-global-statement} and
\eqref{eq:remaining-parameter-statement} hold there with
$r_+=|z-c_N(\xi)|$; in particular, every term $h_M^+$ with $M\ne N$ has the
usual left-end bound $\delta_{\!*}(N_0)r_+^{a_\Sigma-2}$ in that coordinate.  There are no
metric terms in the fixed annulus $\{1/2\le|x|\le2\}$.
\end{lemma}

\begin{proof}
By the support construction, at each point where $h_{\mathrm{rem},N}$ is
nonzero it agrees with one remaining perturbation $h_M^-$ or $h_M^+$,
$M\ne N$.  On $\operatorname{supp}h_M^-$,
\[
 |x_M^--x_N^-|\ge \frac1{M(M+1)},\qquad
 r=|x-c_N(\xi)|\ge cM^{-2},
\]
after increasing $N_0$.  Hence \eqref{eq:left-global} gives
\[
 \sum_{j=0}^2r^{j-2}|D^jh_M^-|
 \le C\mu_MM^{-3d_f+6}
 \le C\mu_MM^{n-3d_f-2}r^{a_\Sigma-2}.
\]
The last polynomial loss is absorbed by the definition of
$\delta_{\!*}(N_0)$.

With $y$ fixed,
\[
 \mathscr D_0[D^jh_M^-(x)]
 =(x-c_N(\xi))\cdot D(D^jh_M^-)(x),\qquad
 \mathscr D_k[D^jh_M^-(x)]
 =\epsilon_N\partial_k(D^jh_M^-)(x).
\]
Since $\epsilon_N\le Cr$ on every remaining minus support, a normalized
parameter derivative is bounded by $Cr|D^{j+1}h_M^-|$.  The preceding
argument, with one additional derivative and hence only another fixed
polynomial loss in $M$, proves \eqref{eq:remaining-parameter-statement} on
the minus supports.

On $\operatorname{supp}h_M^+$, \eqref{eq:right-geometry} gives $r\asymp M$,
and \eqref{eq:right-global} yields
\[
 \sum_{j=0}^2r^{j-2}|D^jh_M^+|
 \le C\mu_MM^{-3d_f+2}
 \le C\mu_MM^{a_\Sigma+4-3d_f}r^{-a_\Sigma-2}.
\]
The parameter-derivative estimate is identical, again with only a fixed
polynomial loss.  Summing these exponentially decaying bounds in $M$ proves
\eqref{eq:remaining-global-statement} and
\eqref{eq:remaining-parameter-statement}.  Finally,
$K_{2,A}h_{\mathrm{rem},N}=h_{\mathrm{rem},N}$, so the same estimates hold in
the right Kelvin coordinate.  The assertion on the fixed annulus follows
from the support construction.
\end{proof}

\medskip
\noindent\textbf{Scalar antipodal action and paired parameters.}
For a positive scalar function $u$ on $\mathbb R^n$, define
\begin{equation}\label{eq:scalar-Kelvin}
 (\mathcal K_Au)(x)=|x|^{-a_\Sigma}u(\iota_A(x)).
\end{equation}
Then $\mathcal K_A^2=\mathrm{Id}$.  For
\begin{equation}\label{eq:standard-bubble-moving}
 U_{\lambda,c}(x)
 =\left(\frac{2\lambda}{\lambda^2+|x-c|^2}\right)^a,
\end{equation}
one has the exact identity
\begin{equation}\label{eq:Kelvin-bubble-parameters}
 \mathcal K_AU_{\lambda,c}=U_{\lambda^\sharp,c^\sharp},
 \qquad
 \lambda^\sharp=\frac{\lambda}{|c|^2+\lambda^2},
 \qquad
 c^\sharp=-\frac{c}{|c|^2+\lambda^2}.
\end{equation}
For the selected parameters in \eqref{eq:principal-parameters-step3},
\begin{equation}\label{eq:right-center-deep-core}
 |c^\sharp-x_N^+|
 \le C\bigl(N^2|\xi|+N^3(|\xi|^2+\lambda^2)\bigr)
 =o(N^{-1}).
\end{equation}
The reflected region on which the cutoff of $h_N^+$ equals one has radius
comparable with $N^{-1}$.  Thus the moving right bubble remains deep inside
the principal reflected metric core.  All right-hand bubble, transition, and
Jacobi objects below are defined as the exact $\mathcal K_A$-images of the
left-hand objects, so \eqref{eq:Kelvin-bubble-parameters} is used without
asymptotic truncation.

\section{Radial models and matched transitions}

For a positive radial Euclidean representative $u=u(r)$, put
\begin{equation}\label{eq:radial-p-q-V-S}
 p=-\frac{2}{n-4}\frac{ru'}u,
 \qquad q=p(1-p),
 \qquad \mathcal V=r^au,
 \qquad \dot p=\frac{dp}{d\log r},
 \qquad S=2\dot p+(n-4)q.
\end{equation}

\begin{lemma}[Exact radial $\sigma_2$ identity]\label{lem:step4}
For every positive radial $u$,
\begin{equation}\label{eq:exact-radial-proper-residual}
 \mathcal N_\delta(u)
 =r^{-n}\left[
  2(n-1)\mathcal V^4qS
  -\frac{n(n-1)}8\mathcal V^{\frac{4n}{n-4}}
 \right].
\end{equation}
Equivalently, wherever $q>0$,
\begin{equation}\label{eq:exact-radial-normalized-residual}
 \frac{r^n\mathcal N_\delta(u)}{\mathcal V^4q^2}
 =4(n-1)\frac{\dot p+a_\Sigma q}{q}
 -\frac14\binom n2
  \frac{\mathcal V^{16/(n-4)}}{q^2}.
\end{equation}
The standard bubble $U_{\lambda,0}$ makes
\eqref{eq:exact-radial-proper-residual} vanish, while the
$\sigma_2$-Schwarzschild factor
\[
 S_\lambda(r)=(2\lambda)^a(1+r^{-a_\Sigma})
\]
satisfies $\dot p+a_\Sigma q=0$, so its $\sigma_2$ curvature term vanishes
identically.
\end{lemma}

\begin{proof}
Set $w=\frac4{n-4}\log u$, so that $g_u=e^{2w}\delta$.  In the radial--tangential
frame, the conformal Schouten formula gives
\[
 \alpha_R=-w''+\frac12(w')^2=\frac{2(\dot p-q)}{r^2},
 \qquad
 \alpha_T=-\frac{w'}r-\frac12(w')^2=\frac{2q}{r^2}.
\]
Thus
\begin{equation}\label{eq:radial-sigma2-step4}
 \sigma_2(g_u^{-1}A_{g_u})
 =\frac{2(n-1)}{r^4}u^{-16/(n-4)}q
   \bigl(2\dot p+(n-4)q\bigr).
\end{equation}
Multiplying by $u^{4n/(n-4)}$, writing $u=r^{-a}\mathcal V$, and subtracting
the constant target term yields \eqref{eq:exact-radial-proper-residual};
division by $r^{-n}\mathcal V^4q^2$ gives
\eqref{eq:exact-radial-normalized-residual}.  For the standard bubble,
$\dot p=2q$ and $\mathcal V^{16/(n-4)}=16q^2$, so the two terms in
\eqref{eq:exact-radial-normalized-residual} cancel.  For $S_\lambda$,
$\dot p=-a_\Sigma q$, and hence its $\sigma_2$ curvature term vanishes by
\eqref{eq:radial-sigma2-step4}.
\end{proof}

For the matched outer transition set
\begin{equation}\label{eq:transition-scales}
 R_\lambda^+=\lambda^{-4/n},
 \qquad R_\lambda^-=(R_\lambda^+)^{-1},
 \qquad \delta_\lambda=\lambda^{2-8/n},
 \qquad s=\log(r/R_\lambda^+).
\end{equation}
The two exact endpoint slopes are
\begin{equation}\label{eq:transition-end-slopes}
 p_\Sigma(s)=\frac{\delta_\lambda e^{-a_\Sigma s}}
 {1+\delta_\lambda e^{-a_\Sigma s}},
 \qquad
 p_B(s)=\frac{\delta_\lambda e^{2s}}
 {1+\delta_\lambda e^{2s}}.
\end{equation}
For $\gamma>0$ define the one-parameter family of leading slopes
\begin{equation}\label{eq:leading-transition-family}
 F_\gamma(s)=\bigl(\gamma e^{-2a_\Sigma s}+e^{4s}\bigr)^{1/2}.
\end{equation}
Every member of this family satisfies the leading transition identity
\begin{equation}\label{eq:leading-transition-ode}
 \frac{F_\gamma'}{F_\gamma}+a_\Sigma
 =\frac n2\frac{e^{4s}}{F_\gamma^2}.
\end{equation}

For each fixed $L\ge4$, choose $\chi_{\Sigma,L},\chi_{B,L}$ by translating
fixed smooth cutoffs satisfying
\[
 0\le \chi_{\Sigma,L},\chi_{B,L}\le 1,
 \qquad
 \operatorname{supp}\chi_{\Sigma,L}\subset[-L,-L+1],
 \qquad
 \operatorname{supp}\chi_{B,L}\subset[L-1,L],
\]
$\chi_{\Sigma,L}=1$ near $-L$, $\chi_{B,L}=1$ near $L$, and with all
$C^k$ norms independent of $L$.  Put
\[
 \omega_L=1-\chi_{\Sigma,L}-\chi_{B,L}.
\]
For $\gamma>0$ define the first-order endpoint-corrected slope
\begin{equation}\label{eq:first-order-cutoff-slope}
 F^{(0)}_{L,\gamma}(s)
 =\chi_{\Sigma,L}e^{-a_\Sigma s}
  +\chi_{B,L}e^{2s}+\omega_LF_\gamma(s)
\end{equation}
and its first-order action
\begin{equation}\label{eq:first-order-action}
 \mathfrak a_L(\gamma)
 =-e^{a_\Sigma L}-a e^{2L}
  +a_\Sigma\int_{-L}^L F^{(0)}_{L,\gamma}(s)\,ds.
\end{equation}

\begin{lemma}[Leading action normalization and endpoint value matching]\label{lem:step5}
For every sufficiently large $L$ there is a unique number
$\gamma_L\in(1/2,3/2)$ satisfying
\begin{equation}\label{eq:gamma-action-zero}
 \mathfrak a_L(\gamma_L)=0,
 \qquad
 |\gamma_L-1|\le C e^{-a_\Sigma L}.
\end{equation}
Define
\begin{equation}\label{eq:base-transition-slope-corrected}
 p^{(0)}_{\lambda,L}
 =\chi_{\Sigma,L}p_\Sigma+\chi_{B,L}p_B
  +\omega_L\delta_\lambda F_{\gamma_L}(s).
\end{equation}
Fix $0\le\zeta\in C_c^\infty((-1,1))$ with
$\int_{-1}^1\zeta=1$, and put
\begin{equation}\label{eq:corrected-transition-slope}
 p_{\lambda,L,\kappa}=p^{(0)}_{\lambda,L}+\kappa\zeta.
\end{equation}
Let
\begin{equation}\label{eq:transition-action-functional}
 \mathscr A_{\lambda,L}(\kappa)
 =\log U_{\lambda^{-1},0}(R_\lambda^+e^L)
  -\log S_\lambda(R_\lambda^+e^{-L})
  +a_\Sigma\int_{-L}^Lp_{\lambda,L,\kappa}(s)\,ds.
\end{equation}
Then, for all sufficiently small $\lambda$, there is a unique
$\kappa_\lambda$ such that
$\mathscr A_{\lambda,L}(\kappa_\lambda)=0$, and
\begin{equation}\label{eq:kappa-bound-step5}
 |\kappa_\lambda|+|\lambda\partial_\lambda\kappa_\lambda|
 \le C_L\delta_\lambda^2.
\end{equation}
If $T_\lambda^+$ is defined by
\begin{equation}\label{eq:integrated-transition-profile}
 \frac{d}{ds}\log T_\lambda^+
 =-a_\Sigma p_{\lambda,L,\kappa_\lambda},
 \qquad
 T_\lambda^+(R_\lambda^+e^{-L})
 =S_\lambda(R_\lambda^+e^{-L}),
\end{equation}
then $T_\lambda^+$ agrees identically with $S_\lambda$ near the left
endpoint and with $U_{\lambda^{-1},0}$ near the right endpoint.
\end{lemma}

\begin{proof}
Differentiation under the integral sign gives
\[
 \mathfrak a_L'(\gamma)
 =\frac{a_\Sigma}{2}\int_{-L}^L
 \omega_L(s)\frac{e^{-2a_\Sigma s}}{F_\gamma(s)}\,ds>0.
\]
For $\gamma\in[1/2,3/2]$, the left half of the integral gives
\begin{equation}\label{eq:first-action-derivative-size}
 c e^{a_\Sigma L}\le \mathfrak a_L'(\gamma)
 \le C e^{a_\Sigma L}.
\end{equation}
To estimate $\mathfrak a_L(1)$, put
$X=e^{-a_\Sigma s}$, $Y=e^{2s}$, and $F_1=(X^2+Y^2)^{1/2}$.  Then
\[
 0\le F_1-X\le Ce^{(n-a_\Sigma)s}\quad(s\le0),
 \qquad
 0\le F_1-Y\le Ce^{-(n-2)s}\quad(s\ge0).
\]
These errors, as well as the endpoint-cutoff errors, are integrable, while
\[
 a_\Sigma\int_{-L}^0X=e^{a_\Sigma L}-1,
 \qquad
 a_\Sigma\int_0^LY=a(e^{2L}-1).
\]
Hence $|\mathfrak a_L(1)|\le C$.  Together with
\eqref{eq:first-action-derivative-size}, strict monotonicity and the mean
value theorem give the unique zero and the estimate in
\eqref{eq:gamma-action-zero}.

Write $\delta=\delta_\lambda$.  The endpoint values and slopes satisfy,
uniformly on $[-L,L]$,
\[
 \log U_{\lambda^{-1},0}(R_\lambda^+e^L)
 -\log S_\lambda(R_\lambda^+e^{-L})
 =\delta(-ae^{2L}-e^{a_\Sigma L})+O_L(\delta^2),
\]
\[
 p_\Sigma=\delta e^{-a_\Sigma s}+O_L(\delta^2),
 \qquad
 p_B=\delta e^{2s}+O_L(\delta^2).
\]
Using $\mathfrak a_L(\gamma_L)=0$ therefore gives
\begin{equation}\label{eq:action-at-zero-order-two}
 \mathscr A_{\lambda,L}(0)=O_L(\delta_\lambda^2),
 \qquad
 \lambda\partial_\lambda\mathscr A_{\lambda,L}(0)
 =O_L(\delta_\lambda^2).
\end{equation}
Since $\int\zeta=1$,
$\mathscr A_{\lambda,L}(\kappa)
 =\mathscr A_{\lambda,L}(0)+a_\Sigma\kappa$; hence its unique zero satisfies
\eqref{eq:kappa-bound-step5}.  Near the two endpoints the corrected slope is
exactly $p_\Sigma$ and $p_B$, respectively.  The value matching encoded by
$\mathscr A_{\lambda,L}(\kappa_\lambda)=0$ then gives the asserted agreement
of $T_\lambda^+$ with the two endpoint models.
\end{proof}

For a positive radial function $u$ on an annulus corresponding to
$s\in I\subset\mathbb R$, define its local normalized energy by
\begin{equation}\label{eq:radial-local-energy-step6}
 \mathscr E_\delta(u;I)
 :=\omega_{n-1}\int_I
 \left[
  2(n-1)\mathcal V^4q\bigl(2\dot p+(n-4)q\bigr)
  -\frac{(n-4)(n-1)}8\mathcal V^{\frac{4n}{n-4}}
 \right]ds.
\end{equation}
This is the restriction of the normalized flat $\sigma_2$ energy to the
corresponding radial annulus.

\begin{lemma}[Transition residual, energy, and Newton ellipticity]\label{lem:step6}
Let $T_\lambda^+$ be the matched radial transition obtained in
Lemma~\ref{lem:step5}, and let $T_\lambda^-=\mathcal K_AT_\lambda^+$.  The
translated left transition is read in $r_-=|x-c_N(\xi)|$, while its right-hand
partner is always read in the Kelvin coordinate $z=\iota_A(x)$.  In
\eqref{eq:transition-residual-step6} and
\eqref{eq:transition-energy-step6}, the derivative
$\lambda\partial_\lambda$ is taken with the cylinder coordinate $s$ fixed.  Then
\begin{equation}\label{eq:transition-residual-step6}
 \left\|
 \frac{r^n\mathcal N_\delta(T_\lambda^\pm)}
      {\mathcal V^4q^2}
 \right\|_{C^\alpha_{\rm cyl}}
 +
 \left\|
 \lambda\partial_\lambda
 \frac{r^n\mathcal N_\delta(T_\lambda^\pm)}
      {\mathcal V^4q^2}
 \right\|_{C^\alpha_{\rm cyl}}
 \le C\bigl(e^{-a_\Sigma L}+C_L\delta_\lambda\bigr).
\end{equation}
They also satisfy, uniformly on both transition blocks,
\begin{equation}\label{eq:transition-ellipticity-step6}
 0<p<1,
 \qquad
 \dot p+(n-3)p(1-p)\ge c\,p(1-p).
\end{equation}
Moreover,
\begin{equation}\label{eq:transition-energy-step6}
 \bigl|\mathscr E_\delta(T_\lambda^+;[-L,L])\bigr|
 +\bigl|\lambda\partial_\lambda
       \mathscr E_\delta(T_\lambda^+;[-L,L])\bigr|
 \le C_L\lambda^{n-4},
\end{equation}
and the same estimate holds for $T_\lambda^-$ on the reflected
transition annulus.  In particular, the total contribution of the two
transition blocks, with one normalized scale derivative, is
$O_L(\lambda^{n-4})$.
\end{lemma}

\begin{proof}
Put $\delta=\delta_\lambda$, $F_L=F^{(0)}_{L,\gamma_L}$, and
$p=p_{\lambda,L,\kappa_\lambda}$.  On the fixed cylinder $[-L,L]$, the
endpoint expansions and Lemma~\ref{lem:step5} give, with one
$\lambda\partial_\lambda$ derivative,
\begin{align}
 \|p-\delta F_L\|_{C^{2,\alpha}}
 +\|\lambda\partial_\lambda(p-\delta F_L)\|_{C^{2,\alpha}}
 &\le C_L\delta^2,
 \label{eq:p-leading-expansion-step6}\\
 q=p(1-p)&=\delta F_L\bigl(1+O_{C^{1,\alpha},L}(\delta)\bigr).
 \label{eq:q-leading-expansion-step6}
\end{align}
In particular, $q>0$ for small $\lambda$.

Define
\[
 \mathfrak d_L
 =\frac{F_L'}{F_L}+a_\Sigma-\frac n2\frac{e^{4s}}{F_L^2}.
\]
On the central interval this vanishes by
\eqref{eq:leading-transition-ode}.  On the two endpoint strips, writing
$F_L=e^{-a_\Sigma s}G_L$ and $F_L=e^{2s}G_L$, respectively, the cutoff
construction and \eqref{eq:gamma-action-zero} give
\begin{equation}\label{eq:leading-defect-bound-step6}
 \|\mathfrak d_L\|_{C^{1,\alpha}}
 \le Ce^{-a_\Sigma L}.
\end{equation}
Consequently,
\begin{equation}\label{eq:first-residual-ratio-step6}
 \frac{\dot p+a_\Sigma q}{q}
 =\frac n2\frac{e^{4s}}{F_L^2}
  +\mathfrak d_L+O_{C^\alpha,L}(\delta),
\end{equation}
where the differentiated remainder is again $O_L(\delta)$.

Let $\mathcal V=r^aT_\lambda^+$ and set
\[
 H_\lambda(s)=\frac{\mathcal V(s)^{16/(n-4)}}{16\delta^2e^{4s}}.
\]
Since $d_s\log\mathcal V=a-a_\Sigma p$,
$d_s\log H_\lambda=-8p$, while agreement with $S_\lambda$ at $s=-L$
gives $H_\lambda(-L)=(1+\delta e^{a_\Sigma L})^{16/(n-4)}$.  Hence
\begin{equation}\label{eq:H-bound-step6}
 \|H_\lambda-1\|_{C^{1,\alpha}}
 +\|\lambda\partial_\lambda H_\lambda\|_{C^{1,\alpha}}
 \le C_L\delta,
\end{equation}
and therefore
\[
 \frac{\mathcal V^{16/(n-4)}}{q^2}
 =16\frac{e^{4s}}{F_L^2}
  \bigl(1+O_{C^\alpha,L}(\delta)\bigr).
\]
Substitution of these two expansions into
\eqref{eq:exact-radial-normalized-residual} cancels the leading terms and
proves \eqref{eq:transition-residual-step6} for $T_\lambda^+$.

For a radial function $u$, set $\widehat u(r)=r^{-a_\Sigma}u(r^{-1})$.
Then
\begin{equation}\label{eq:Kelvin-slope-step6}
 p_{\widehat u}(r)=1-p_u(r^{-1}),\qquad
 \dot p_{\widehat u}(r)=\dot p_u(r^{-1}),\qquad
 q_{\widehat u}(r)=q_u(r^{-1}),\qquad
 \mathcal V_{\widehat u}(r)=\mathcal V_u(r^{-1}).
\end{equation}
Thus the residual estimate transfers to $T_\lambda^-$ by $s\mapsto-s$.

It remains to verify ellipticity and the energy bound.  On the central
interval, \eqref{eq:leading-transition-ode} gives
\[
 \frac{F_L'}{F_L}+n-3
 =\frac{n-2}{2}+\frac n2\frac{e^{4s}}{F_L^2}
 \ge\frac{n-2}{2};
\]
the endpoint representations used above give the same positive lower bound,
up to a fixed constant, on the cutoff strips.  Equations
\eqref{eq:p-leading-expansion-step6}--\eqref{eq:q-leading-expansion-step6}
therefore imply
$\dot p+(n-3)q\ge cq$ and $0<p<1$ for small $\lambda$.  Formula
\eqref{eq:Kelvin-slope-step6} gives the same result for $T_\lambda^-$,
proving \eqref{eq:transition-ellipticity-step6}.

Finally, on the fixed cylinder,
$q=O_L(\delta)$,
$2\dot p+(n-4)q=O_L(\delta)$,
$\mathcal V^{16/(n-4)}=O_L(\delta^2)$, and
$\mathcal V^4=O_L(\delta^{(n-4)/2})$, with one scale derivative.  Hence both
terms in the radial energy density are $O_L(\delta^{n/2})$.  Since
$\delta^{n/2}=\lambda^{n-4}$, integration gives
\eqref{eq:transition-energy-step6}; the Kelvin-reflected estimate is the
same.
\end{proof}

\begin{lemma}[Transition linear gap]\label{lem:step7}
Let $Y_{\ell,m}$ be a spherical harmonic satisfying
\[
 -\Delta_{\mathbb S^{n-1}}Y_{\ell,m}=\Lambda_\ell Y_{\ell,m},
 \qquad \Lambda_\ell\ge0,
\]
and define
\begin{equation}\label{eq:vartheta-model-step7}
 \vartheta_L(s)
 =\frac{e^{4s}}{\gamma_Le^{-2a_\Sigma s}+e^{4s}},
 \qquad
 A_{0,L}=a_\Sigma+\frac n2\vartheta_L,
 \qquad
 B_{0,L}=\frac{n-2+n\vartheta_L}{2(n-1)}.
\end{equation}
The uncut leading family $F_{\gamma_L}$ has Liouville potential
\begin{equation}\label{eq:transition-potential-step7}
 \mathfrak Q_{0,L}(s)
 =\frac12A_{0,L}'(s)+\frac14A_{0,L}(s)^2
 =\frac14\left(a_\Sigma+\frac n2\vartheta_L(s)\right)^2
  +\frac{n^2}{4}\vartheta_L(s)(1-\vartheta_L(s))
 \ge\frac{a_\Sigma^2}{4}.
\end{equation}
For every sufficiently large $L$, and then for all sufficiently small
$\lambda$, the relative linearization at $T_\lambda^+$ in the
$\ell$-th harmonic sector is conjugate to
\begin{equation}\label{eq:actual-schrodinger-step7}
 \mathcal H_{\lambda,L,\ell}
 =-\frac{d^2}{ds^2}+\mathfrak Q_{\lambda,L,\ell}(s)
 \quad\hbox{on }[-L,L],
\end{equation}
with Dirichlet boundary condition, where
\begin{align}
 \bigl\|\mathfrak Q_{\lambda,L,\ell}
  -(\mathfrak Q_{0,L}+B_{0,L}\Lambda_\ell)\bigr\|_{C^\alpha}
 &\le C(1+\Lambda_\ell)
       \bigl(e^{-a_\Sigma L}+C_L\delta_\lambda\bigr),
 \label{eq:potential-perturbation-step7}\\
 \bigl\|\lambda\partial_\lambda
       \mathfrak Q_{\lambda,L,\ell}\bigr\|_{C^\alpha}
 &\le C_L(1+\Lambda_\ell)\delta_\lambda.
 \label{eq:potential-parameter-step7}
\end{align}
Consequently, there are constants $c_*>0$ and $L_*>0$, independent of
$\ell$, such that for $L\ge L_*$ and sufficiently small $\lambda$,
\begin{equation}\label{eq:transition-coercivity-step7}
 \int_{-L}^L\left(|y'|^2
  +\mathfrak Q_{\lambda,L,\ell}y^2\right)ds
 \ge c_*\int_{-L}^L
       \left(|y'|^2+(1+\Lambda_\ell)y^2\right)ds
\end{equation}
for every $y\in H_0^1(-L,L)$.  In particular, no spherical-harmonic
sector has a nonzero Dirichlet kernel.  The same conclusion holds for
$T_\lambda^-$ after reflection by the Kelvin map.
\end{lemma}

\begin{proof}
Write $u=r^{-a}\mathcal V$ and
$u^{8/(n-4)}\delta=e^{2\psi}g_{\rm cyl}$ with
$\psi=\frac4{n-4}\log\mathcal V$.  Since
$\psi_s=1-2p$ and $\psi_{ss}=-2\dot p$, the first Newton tensor has radial
and tangential eigenvalues
\begin{equation}\label{eq:cylinder-newton-step7}
 \tau_R=2(n-1)q,
 \qquad
 \tau_T=2\bigl(\dot p+(n-3)q\bigr).
\end{equation}
For the relative variation $u_t=ue^{atw}$, direct linearization in the
cylindrical frame gives
\[
 \mathscr P_uw
 =w_{ss}+A_uw_s+B_u\Delta_{\mathbb S^{n-1}}w-C_uw,
\]
where
\begin{align}
 A_u&=(1-2p)\frac{\dot p+(n-4)q}{q},
 \label{eq:A-coeff-step7}\\
 B_u&=\frac{\dot p+(n-3)q}{(n-1)q},\\
 C_u&=(n-4)\bigl(2\dot p+(n-4)q\bigr)
 -\frac{n^2}{16}\frac{\mathcal V^{16/(n-4)}}q.
 \label{eq:C-coeff-step7}
\end{align}
The normalization is by a positive factor, so it does not affect the
kernel.

Let $F_{\gamma_L}$ denote the uncut leading profile.  The endpoint-strip
representations from the preceding proof imply
\[
 \left\|\frac{F_L}{F_{\gamma_L}}-1\right\|_{C^{2,\alpha}}
 \le Ce^{-a_\Sigma L}.
\]
Combining this with Lemma~\ref{lem:step6},
\eqref{eq:leading-transition-ode}, and \eqref{eq:H-bound-step6} yields
\begin{align*}
 \|A_u-A_{0,L}\|_{C^{1,\alpha}}
 +\|B_u-B_{0,L}\|_{C^{1,\alpha}}
 &\le C\bigl(e^{-a_\Sigma L}+C_L\delta_\lambda\bigr),\\
 \|C_u\|_{C^\alpha}&\le C_L\delta_\lambda,
\end{align*}
with $O_L(\delta_\lambda)$ bounds after one scale derivative.

In the $\ell$th harmonic sector, put
$w_\ell=e^{-\Phi_u/2}y_\ell$ with $\Phi_u'=A_u$.  The equation becomes
\[
 -y_\ell''+\mathfrak Q_{\lambda,L,\ell}y_\ell=0,
 \qquad
 \mathfrak Q_{\lambda,L,\ell}
 =\frac12A_u'+\frac14A_u^2+B_u\Lambda_\ell+C_u.
\]
Since $\vartheta_L'=n\vartheta_L(1-\vartheta_L)$, the limiting potential is
exactly \eqref{eq:transition-potential-step7}.  The preceding coefficient
bounds give \eqref{eq:potential-perturbation-step7} and
\eqref{eq:potential-parameter-step7}.

Finally,
$\mathfrak Q_{0,L}\ge a_\Sigma^2/4$ and
$B_{0,L}\ge(n-2)/(2(n-1))$.  Choosing first $L$ and then $\lambda$ so that
the perturbation is smaller than one half of these bounds gives
\[
 \mathfrak Q_{\lambda,L,\ell}
 \ge\frac{a_\Sigma^2}{8}
   +\frac{n-2}{4(n-1)}\Lambda_\ell.
\]
Integration by parts proves \eqref{eq:transition-coercivity-step7} and the
absence of Dirichlet kernel.  The result for $T_\lambda^-$ follows from
\eqref{eq:Kelvin-slope-step6}.
\end{proof}

\begin{proposition}[Exact scalar Kelvin pairing and covariance]\label{prop:scalar-kelvin-verified}
The scalar action in \eqref{eq:scalar-Kelvin} is an involution and the
bubble identity \eqref{eq:Kelvin-bubble-parameters} holds exactly.  Moreover,
for every positive smooth $u$ and every smooth variation $\phi$, writing
\[
 L_\delta(u)[\phi]
 :=\left.\frac{d}{dt}\right|_{t=0}\mathcal N_\delta(u+t\phi),
\]
one has
\begin{align}
 \mathcal N_\delta(\mathcal K_Au)(x)
 &=|x|^{-2n}\mathcal N_\delta(u)(\iota_A(x)),
 \label{eq:Kelvin-residual-covariance-prop}\\
 L_\delta(\mathcal K_Au)[\mathcal K_A\phi](x)
 &=|x|^{-2n}L_\delta(u)[\phi](\iota_A(x)).
 \label{eq:Kelvin-linear-covariance-prop}
\end{align}
For the selected parameters in \eqref{eq:principal-parameters-step3},
\eqref{eq:right-center-deep-core} holds.  In addition, for every fixed $L$,
the entire left matched-transition annulus and its exact Kelvin image lie in
the regions where the cutoffs of $h_N^-$ and $h_N^+$, respectively, are
identically one, once $N_0$ is sufficiently large.
\end{proposition}

\begin{proof}
The involution property follows from $\iota_A^2=\mathrm{Id}$ and
$|\iota_A(x)|=|x|^{-1}$.  Completing the square in
\[
 \lambda^2+|\iota_A(x)-c|^2
 =|x|^{-2}\bigl[(|c|^2+\lambda^2)|x|^2+2c\cdot x+1\bigr]
\]
gives \eqref{eq:Kelvin-bubble-parameters}.  Moreover,
\[
 (\mathcal K_Au)^{8/(n-4)}\delta
 =\iota_A^*\bigl(u^{8/(n-4)}\delta\bigr),
\]
so conformal covariance and
$(\mathcal K_Au)^{4n/(n-4)}(x)
 =|x|^{-2n}u^{4n/(n-4)}(\iota_A(x))$ give
\eqref{eq:Kelvin-residual-covariance-prop}; differentiation gives
\eqref{eq:Kelvin-linear-covariance-prop}.

Expanding $c^\sharp=-c/(|c|^2+\lambda^2)$ around $c=x_N^-$ proves
\eqref{eq:right-center-deep-core}.  Finally, on the left transition
$|z-x_N^-|\le e^L\lambda^{4/n}+|\xi|=o(\rho_N)$ for fixed $L$ and
$\lambda\asymp\epsilon_N$.  Hence this transition lies in the uncut core of
$h_N^-$, and its Kelvin image lies in the corresponding uncut core of
$h_N^+$.
\end{proof}

\medskip
\noindent\textbf{Global paired approximation in the fixed stereographic coordinate.}
We work in the fixed stereographic coordinate and choose a radial cutoff
$\chi_-\in C^\infty(\mathbb R^n)$ such that
\[
 0\le\chi_-\le1,\qquad
 \chi_-=1\ \hbox{on }B_{1/3}(0),\qquad
 \chi_-=0\ \hbox{on }\mathbb R^n\setminus B_{1/2}(0),
\]
and put
\begin{equation}\label{eq:paired-partition}
 \chi_+(x)=\chi_-(\iota_A(x)),
 \qquad \chi_0=1-\chi_--\chi_+.
\end{equation}
The supports of $\chi_-$ and $\chi_+$ are disjoint.  After increasing $N_0$,
the two collar regions
\begin{equation}\label{eq:assembly-collars}
 \mathcal A_-:=\{1/3\le |x|\le1/2\},
 \qquad
 \mathcal A_+:=\iota_A(\mathcal A_-)=\{2\le |x|\le3\}
\end{equation}
contain no metric perturbation.

For $c=c_N(\xi)=x_N^-+\xi$, let
$\widetilde U^-_{N,\lambda,\xi}$ be the one-ended matched bubble profile obtained by
using the exact bubble $U_{\lambda,c}$ in its core, replacing $r$ by
$r_-=|x-c|$ in the left matched transition, and then continuing with the
translated Schwarzschild factor
\begin{equation}\label{eq:translated-Schwarzschild}
 S_{\lambda,c}(x)=(2\lambda)^a
 \bigl(1+|x-c|^{-a_\Sigma}\bigr).
\end{equation}
Thus, for all sufficiently large $N$, the profile
$\widetilde U^-_{N,\lambda,\xi}$ agrees identically with
$S_{\lambda,c}$ on a neighborhood of $\mathcal A_-$.  Define the complete
right-end profile by the exact scalar Kelvin action
\begin{equation}\label{eq:exact-right-end-profile}
 \widetilde U^+_{N,\lambda,\xi}=\mathcal K_A
 \widetilde U^-_{N,\lambda,\xi}.
\end{equation}
Its bubble core is exactly $U_{\lambda^\sharp,c^\sharp}$ with the parameters
in \eqref{eq:Kelvin-bubble-parameters}; the right transition is always read
in the Kelvin coordinate $z=\iota_A(x)$ and is not replaced by a radial
Euclidean approximation about $c^\sharp$.

Let
\begin{equation}\label{eq:central-Schwarzschild}
 S_\lambda(x)=(2\lambda)^a(1+|x|^{-a_\Sigma}),
 \qquad \mathcal K_AS_\lambda=S_\lambda.
\end{equation}
Define the positive global Euclidean representative by
\begin{equation}\label{eq:global-paired-Euclidean-approximation}
 \log \bar u_{N,\lambda,\xi}
 =\chi_-\log\widetilde U^-_{N,\lambda,\xi}
  +\chi_+\log\widetilde U^+_{N,\lambda,\xi}
  +\chi_0\log S_\lambda,
\end{equation}
and set
\begin{equation}\label{eq:global-paired-approximation}
 U_{1,0}=\Omega^a,
 \qquad
 \bar v_{N,\lambda,\xi}
 =U_{1,0}^{-1}\bar u_{N,\lambda,\xi}.
\end{equation}
The background tensor $h$ remains fixed and is never translated, rotated,
or made parameter dependent.

For use on the collars, define the Euclidean-index first Newton form
\begin{equation}\label{eq:Euclidean-Newton-form-collar}
 \mathbf T_\delta[u]^{ij}
 :=\sigma_1(\delta^{-1}A_{u^{8/(n-4)}\delta})\delta^{ij}
   -\delta^{ik}\delta^{j\ell}
     A_{u^{8/(n-4)}\delta,k\ell}.
\end{equation}
Also put, for $r=|x|$,
\begin{equation}\label{eq:collar-Schwarzschild-scales}
 p_\Sigma(r)=\frac{r^{-a_\Sigma}}{1+r^{-a_\Sigma}},
 \qquad q_\Sigma=p_\Sigma(1-p_\Sigma),
 \qquad \mathcal V_\Sigma=r^aS_\lambda(r).
\end{equation}

\begin{proposition}[Paired assembly and collar estimates]\label{prop:paired-assembly-collars}
The global approximation is positive and satisfies the exact identities
\begin{equation}\label{eq:exact-paired-invariance}
 \mathcal K_A\bar u_{N,\lambda,\xi}
 =\bar u_{N,\lambda,\xi},
 \qquad
 \bar v_{N,\lambda,\xi}\circ J
 =\bar v_{N,\lambda,\xi}.
\end{equation}
For parameter differentiation on the fixed collars set
$\mathfrak D_0=\lambda\partial_\lambda$ and
$\mathfrak D_k=\epsilon_N\partial_{\xi_k}$, $1\le k\le n$, with $x$ fixed.
Then, uniformly in the normalized parameters,
\begin{equation}\label{eq:recentering-collar-bound}
 \sum_{j=0}^3
 \left|D_x^j\log\frac{S_{\lambda,c}}{S_\lambda}\right|
 +\sum_{b=0}^n\sum_{j=0}^3
 \left|D_x^j\mathfrak D_b
 \log\frac{S_{\lambda,c}}{S_\lambda}\right|
 \le \frac{C}{N}
 \qquad\hbox{on }\mathcal A_-.
\end{equation}
The estimate on $\mathcal A_+$ is its exact Kelvin image.  On $\mathcal A_-$,
\begin{equation}\label{eq:collar-Newton-comparison}
 (1-CN^{-1})\mathbf T_\delta[S_\lambda]
 \le \mathbf T_\delta[\bar u_{N,\lambda,\xi}]
 \le (1+CN^{-1})\mathbf T_\delta[S_\lambda]
\end{equation}
as quadratic forms.  Finally,
\begin{align}
 &\left\|
  \frac{|x|^n\mathcal N_\delta
   (\bar u_{N,\lambda,\xi})}
       {\mathcal V_\Sigma^4q_\Sigma^2}
 \right\|_{C^\alpha(\mathcal A_-)}
 \nonumber\\
 &\qquad+
 \sum_{b=0}^n\left\|
 \mathfrak D_b\left[
  \frac{|x|^n\mathcal N_\delta
   (\bar u_{N,\lambda,\xi})}
       {\mathcal V_\Sigma^4q_\Sigma^2}
 \right]\right\|_{C^\alpha(\mathcal A_-)}
 \le C\bigl(N^{-1}+\lambda^4\bigr).
 \label{eq:collar-normalized-residual}
\end{align}
The corresponding statement on $\mathcal A_+$ holds in the exact Kelvin
coordinate.  In particular, after increasing $N_0$, both collars are
uniformly Newton elliptic relative to the Schwarzschild model.
\end{proposition}

\begin{proof}
The relations among the cutoffs and endpoint profiles in
\eqref{eq:global-paired-Euclidean-approximation} give
$\log\bar u(\iota_A(x))=\log\bar u(x)+a_\Sigma\log|x|$, which is equivalent
to \eqref{eq:exact-paired-invariance}.

On $\mathcal A_-$ set
\[
 \phi_c=\log\frac{S_{\lambda,c}}{S_\lambda}
 =\log\frac{1+|x-c|^{-a_\Sigma}}{1+|x|^{-a_\Sigma}}.
\]
Since the collar is fixed and $|c|\le N^{-1}+\epsilon_N$, Taylor expansion in
$c$ gives \eqref{eq:recentering-collar-bound}.  Thus
$\bar u=S_\lambda e^\psi$ there with
$\|\psi\|_{C^{3,\alpha}}+
 \sum_b\|\mathfrak D_b\psi\|_{C^{3,\alpha}}\le C/N$.
The conformal Schouten formula then yields
$|\mathbf T_\delta[\bar u]-\mathbf T_\delta[S_\lambda]|\le C/N$.
Since the radial and tangential eigenvalues of the Schwarzschild Newton form
are $2(n-1)q_\Sigma/r^2$ and $(n-2)q_\Sigma/r^2$, respectively,
\eqref{eq:collar-Newton-comparison} follows.

The same expansion gives
$A_{\bar u^{8/(n-4)}\delta}
 =A_{S_\lambda^{8/(n-4)}\delta}+O_{C^{1,\alpha}}(N^{-1})$, with one
normalized parameter derivative.  The Schwarzschild curvature term is zero,
while on the fixed collar the target term contributes $O(\lambda^4)$ after
the normalization in \eqref{eq:collar-normalized-residual}.  This proves \eqref{eq:collar-normalized-residual}.  The right-collar conclusions follow from
Proposition~\ref{prop:scalar-kelvin-verified}.
\end{proof}

\medskip
\noindent\textbf{Block convention.}
Let $r_-=|x-c_N(\xi)|$ and set
\[
 \mathcal B_-:=\{r_-<R_\lambda^-e^{-L}\},
 \qquad
 \mathcal T_-:=\{R_\lambda^-e^{-L}\le r_-\le R_\lambda^-e^L\}.
\]
The left Schwarzschild block is the portion between
$\mathcal T_-$ and $\mathcal A_-$; the right blocks are their exact Kelvin
images, and the central Schwarzschild block is $\{1/2\le|x|\le2\}$ together
with the adjacent Schwarzschild pieces.  On a transition block,
$C^\alpha_{\rm cyl}$ denotes the ordinary $C^\alpha$ norm after using the
fixed cylinder variable $s$.  On a bubble or Schwarzschild end block,
$C^\alpha_{\rm block}$ denotes the scale-invariant local norm
\[
 \|F\|_{C^\alpha_{\rm block}}
 :=\sup_x\left(
 |F(x)|+\ell(x)^\alpha
 [F]_{C^\alpha(B_{\ell(x)/8}(x))}
 \right),
 \qquad \ell(x)=\lambda+r_-(x),
\]
with the same definition in the exact right Kelvin coordinate.  On
the two fixed collars it denotes the ordinary fixed-scale $C^\alpha$ norm.

\section{Global residual and admissibility}

Set
\begin{equation}\label{eq:kappa-f}
 \kappa_f=\frac{(n-8)(n-4m_f-8)}{4n}>0.
\end{equation}

\begin{lemma}[Blockwise $q^2$-normalized metric expansion]\label{lem:step8}
In the left end use the fixed Euclidean coordinate with
$r=|x-c_N(\xi)|$; in the right end use the exact Kelvin coordinate
$z=\iota_A(x)$ with $r=|z-c_N(\xi)|$.  On every bubble,
matched-transition, and Schwarzschild block, let $\bar u$ denote
the local Euclidean conformal factor of
$\bar v_{N,\lambda,\xi}$, and put
\[
 p=-\frac{2}{n-4}\frac{r\partial_r\bar u}{\bar u},
 \qquad q=p(1-p),
 \qquad \mathcal V=r^a\bar u.
\]
The quotient below is understood on $r>0$ and by continuous extension at a
bubble center.  Then
\begin{equation}\label{eq:metric-expansion-step8}
 \left\|
 \frac{r^n\bigl(\mathcal N_{e^h\delta}(\bar u)
                  -\mathcal N_\delta(\bar u)\bigr)}
      {\mathcal V^4q^2}
 \right\|_{C^\alpha_{\rm block}}
 \le C\bigl(\delta_{\!*}(N_0)+\epsilon_N^{\kappa_f}\bigr).
\end{equation}
On a transition block the norm in \eqref{eq:metric-expansion-step8} is the
cylindrical norm of the block convention.  The same estimate holds after one
normalized parameter derivative, taken at fixed bubble coordinate in a
bubble block, at fixed cylinder coordinate in a transition block, and in the
corresponding scale-invariant coordinate on a Schwarzschild block.

For a relative variation $u_t=\bar u e^{atw}$ define
\begin{equation}\label{eq:relative-operator-step8}
 \mathscr P_{g,\bar u}w
 :=-\frac{r^n}{2(n-1)\mathcal V^4q}
   \left.\frac{d}{dt}\right|_{t=0}\mathcal N_g(u_t).
\end{equation}
After writing the two operators in the same natural block coordinates and
normalizing their second-, first-, and zeroth-order coefficients by the flat
Newton scale $q/r^2$, one also has
\begin{equation}\label{eq:coefficient-expansion-step8}
 \bigl\|\mathscr P_{e^h\delta,\bar u}
       -\mathscr P_{\delta,\bar u}\bigr\|_{\mathrm{coef},C^\alpha_{\rm block}}
 \le C\bigl(\delta_{\!*}(N_0)+\epsilon_N^{\kappa_f}\bigr),
\end{equation}
with the same normalized parameter-derivative estimate.
\end{lemma}

\begin{proof}
Set
\[
 \psi=\frac4{n-4}\log\bar u,
 \qquad \widehat g_0=e^{2\psi}\delta,
 \qquad \widehat g_h=e^{2\psi}e^h\delta.
\]
The conformal Schouten formula and the standard coordinate expansions for
$e^h\delta$ give
\begin{equation}\label{eq:schouten-perturbation-step8}
 \left|e^{-h}\delta^{-1}A_{\widehat g_h}
       -\delta^{-1}A_{\widehat g_0}\right|
 \le C\mathcal E_h,
 \qquad
 \mathcal E_h:=|D^2h|+|Dh|^2+\frac pr|Dh|+\frac p{r^2}|h|.
\end{equation}
Here $|D\psi|=2p/r$, and the flat-model Schouten tensor has size
$O(q/r^2)$ because $|\dot p|\le Cq$ on every model block.  Put
\begin{equation}\label{eq:Xi-def-step8}
 \Xi_h:=\frac{r^2}{q}\mathcal E_h.
\end{equation}
Since $\sigma_2$ is quadratic and its differential is the first Newton
tensor, \eqref{eq:schouten-perturbation-step8} implies
\[
 \left|
 \frac{r^n(\mathcal N_{e^h\delta}(\bar u)-\mathcal N_\delta(\bar u))}
      {\mathcal V^4q^2}\right|
 +|\mathscr P_{e^h\delta,\bar u}-\mathscr P_{\delta,\bar u}|_{\rm coef}
 \le C(\Xi_h+\Xi_h^2).
\]
The same estimate holds in the scaled $C^\alpha$ norms.  Differentiating the
coordinate expansion before taking absolute values gives the corresponding
parameter estimate; Proposition~\ref{lem:step2} and Lemma~\ref{lem:step3}
control the additional derivatives.

It remains to bound $\Xi_h$.  In the left end write
$h=h_N^-+h_{\rm rem,N}$.  Since
$|x-x_N^-|-r=O(\epsilon_N)$, Proposition~\ref{lem:step2} gives
\[
 |D^jh_N^-|\le C\mu_N(\epsilon_N+r)^{d_f-j},
 \qquad 0\le j\le3,
\]
with the same bound after one normalized parameter derivative.  On a bubble
block, where $\lambda\asymp\epsilon_N$,
\begin{equation}\label{eq:Xi-bubble-step8}
 \Xi_{h_N^-}
 \le C\mu_N(\epsilon_N+r)^{d_f-2}
        \frac{(\lambda^2+r^2)^2}{\lambda^2}
 \le C_L\epsilon_N^{\kappa_f}.
\end{equation}
The last inequality uses $r\le C_Le^{-L}\lambda^{4/n}$ and
$\beta_f+\frac4n(d_f+2)-2=\kappa_f$.
On a matched transition, Lemma~\ref{lem:step6} gives
$p\asymp q\asymp\delta_\lambda$ and $r\asymp\lambda^{4/n}$, so
\begin{equation}\label{eq:Xi-transition-step8}
 \Xi_{h_N^-}
 \le C_L\mu_N\frac{r^{d_f}}{\delta_\lambda}
 \le C_L\epsilon_N^{\kappa_f}.
\end{equation}
On the Schwarzschild pieces, $q\asymp r^{a_\Sigma}$ for $r\le1$ and
$q\asymp r^{-a_\Sigma}$ for $r\ge1$.  The same derivative estimate yields
$\Xi_{h_N^-}\le C\mu_Nr^{d_f-a_\Sigma}$ where the principal perturbation is
present; since $r\ge C_L^{-1}\lambda^{4/n}$, this is again
$O_L(\epsilon_N^{\kappa_f})$.

For the remaining tensor, Lemma~\ref{lem:step3} gives
\[
 |D^2h_{\rm rem,N}|+r^{-1}|Dh_{\rm rem,N}|+r^{-2}|h_{\rm rem,N}|
 \le \delta_{\!*}(N_0)
 \begin{cases}r^{a_\Sigma-2},&r\le1,\\ r^{-a_\Sigma-2},&r\ge1.
 \end{cases}
\]
Substitution into \eqref{eq:Xi-def-step8} gives
$\Xi_{h_{\rm rem,N}}
 \le C(\delta_{\!*}(N_0)+\delta_{\!*}(N_0)^2)$.
The same regional estimates, using the third-derivative bounds, control the
scaled H\"older seminorms and one normalized parameter derivative.  In the
right Kelvin coordinate all objects are the exact Kelvin transforms of their
left-end counterparts.  Taking $N_0$ large so that $\delta_{\!*}(N_0)<1$
proves \eqref{eq:metric-expansion-step8} and
\eqref{eq:coefficient-expansion-step8}.
\end{proof}

\subsection{Global Newton-strong spaces and residual decomposition}
Put
\[
 p_*:=\frac{4n}{n-4}.
\]
For the limiting Schwarzschild cylinder define
\begin{equation}\label{eq:indicial-rates-step9}
 \nu_\ell^2:=a^2+\frac{n-2}{2(n-1)}\Lambda_\ell,
 \qquad
 -\Delta_{\mathbb S^{n-1}}Y_{\ell,m}=\Lambda_\ell Y_{\ell,m},
 \qquad
 \sigma_\ell:=\nu_\ell+a.
\end{equation}
Thus
\[
 \nu_0=a,\qquad \nu_1=1+a,
 \qquad \sigma_0=2a=a_\Sigma.
\]
The zero and first spherical harmonics contain the dilation and translation
transfer branches and must be treated as finite-dimensional modes.  The first
strictly positive decay rate on their complement is
\begin{equation}\label{eq:high-mode-gap-step9}
 \gamma_*:=\nu_2-a>0.
\end{equation}
Fix once and for all
\begin{equation}\label{eq:gamma-choice-step9}
 0<\gamma<\gamma_*.
\end{equation}

We use Euclidean representatives for sources.  If $F$ is a scalar source on
$\mathbb S^n$ and $F_E$ denotes its representative relative to the background
$e^h\delta$, then
\begin{equation}\label{eq:source-representative-step9}
 F_E=U_{1,0}^{p_*}F.
\end{equation}
In particular,
$\mathcal N_{e^h\delta}(U_{1,0}v)=U_{1,0}^{p_*}\mathcal N_g(v)$.

Let
\[
 \mathcal C:=\{1/3\le |x|\le3\}
\]
be the fixed compact connector.  Let $\mathcal S_-$ and $\mathcal S_+$ denote
the left and right Schwarzschild end blocks outside $\mathcal C$,
and put $\mathcal S=\mathcal S_-\cup\mathcal S_+$ and
$\mathcal T=\mathcal T_-\cup\mathcal T_+$.  On each Schwarzschild end use the
cylinder coordinate $t=\log r$, where $r=r_-$ in the left end and
$r=|\iota_A(x)-c_N(\xi)|$ in the exact right Kelvin coordinate.  Let
$\mathscr K_N$ be the $J$-invariant closed source set obtained by taking fixed
unit cylindrical neighborhoods of the two matched transitions, the compact
connector $\mathcal C$, and every support of the fixed metric tensor $h$.
Choose a $J$-invariant smooth regularized cylindrical distance $d_N\ge0$ on
the union of the transition and Schwarzschild blocks such that
\begin{equation}\label{eq:regularized-distance-step9}
 d_N=0\ \hbox{on }\mathscr K_N,
 \qquad 0\le d_N\le \operatorname{dist}_{\rm cyl}(\cdot,\mathscr K_N),
 \qquad |D_{\rm cyl}d_N|\le1,
\end{equation}
and comparable with the true distance outside a fixed unit neighborhood of
$\mathscr K_N$.  Set
\begin{equation}\label{eq:global-weight-step9}
 \omega_{\gamma,N}=e^{\gamma d_N}.
\end{equation}
For parameter derivatives, the weight is frozen in the natural block
coordinates.

On a metric-free radial cylinder block let $\Pi_{\le1}$ and $\Pi_{\ge2}$ denote the
orthogonal projections onto spherical harmonics of degrees $0,1$ and degrees
at least $2$, respectively.  When applied to a source, these projections act
on its normalized cylindrical representative $r^nF_E/(\mathcal V^4q^2)$.  Let $Q$ denote a unit cylinder block.  On a
radial model block let $\mathscr S^{(1)}_{\bar u,h}(w)$ denote the complete
first variation of the Schouten endomorphism, written in a cylindrical
orthonormal frame, under $\bar u\mapsto\bar u e^{atw}$.  Put
\begin{align*}
 \mathfrak X_Q(w)
 &:={\|w\|}_{C^{0,\alpha}(Q)}
   +{\|q^{-1/2}D_{\rm cyl}w\|}_{C^{0,\alpha}(Q)}
   +{\|q^{-1}\mathscr S^{(1)}_{\bar u,h}(w)\|}_{C^\alpha(Q)},\\
 \mathfrak Y_Q(F)
 &:={\left\|\frac{r^nF_E}{\mathcal V^4q^2}\right\|}_{C^\alpha(Q)}.
\end{align*}
For the low spherical harmonics define
\begin{align}\label{eq:low-source-norm-step9}
 \|F\|_{Y_{\rm low}^\sharp}
 :={}&\sup_{Q\subset\mathcal T\cup\mathcal S}
   \mathfrak Y_Q(\Pi_{\le1}F)\nonumber\\
 &+\sum_{\mathfrak e=\pm}
   \int_{\mathcal T_{\mathfrak e}\cup\mathcal S_{\mathfrak e}}
   \left\|\Pi_{\le1}
    \frac{r^nF_E}{\mathcal V^4q^2}(t,\cdot)
   \right\|_{C^\alpha(\mathbb S^{n-1})}\,dt.
\end{align}
The first term is needed for the local Schauder part of the low-mode inverse,
while the integral controls the Green remainder uniformly on long cylinders.

\begin{definition}[Mode-split Newton-strong spaces]\label{def:global-spaces-step9}
For a $J$-invariant relative correction $w$, define
\begin{align}
 \|w\|_{X_\gamma^\sharp}
 :={}&\sum_\pm\|w\|_{C^{2,\alpha}(\mathcal B_\pm^*,g_{U_\pm})}
   +\|w\|_{C^{2,\alpha}(\mathcal C)}
 \nonumber\\
 &+\sup_{Q\subset\mathcal T\cup\mathcal S}
 \left\{
   \mathfrak X_Q(\Pi_{\le1}w)
   +\omega_{\gamma,N}(Q)
    \mathfrak X_Q(\Pi_{\ge2}w)
 \right\}.
 \label{eq:Xsharp-definition-step9}
\end{align}
For a $J$-invariant source $F$, define
\begin{align}
 \|F\|_{Y_\gamma^\sharp}
 :={}&\sum_\pm
   \|U_\pm^{-p_*}F_E\|_{C^\alpha(\mathcal B_\pm^*,g_{U_\pm})}
   +\left\|\frac{|x|^nF_E}{\mathcal V_\Sigma^4q_\Sigma^2}
    \right\|_{C^\alpha(\mathcal C)}
   +\|F\|_{Y_{\rm low}^\sharp}
 \nonumber\\
 &+\sup_{Q\subset\mathcal T\cup\mathcal S}
   \omega_{\gamma,N}(Q)
   \mathfrak Y_Q(\Pi_{\ge2}F).
 \label{eq:Ysharp-definition-step9}
\end{align}
Here $\mathcal B_\pm^*$ are fixed-overlap enlargements of the exact bubble
blocks and $g_{U_\pm}=U_\pm^{8/(n-4)}\delta$.  For a parameter-dependent
quantity, one normalized parameter derivative is measured in the same norm,
with the weight frozen.
\end{definition}

Because
$R_{\rm met}:=\mathcal N_{e^h\delta}(\bar u)
-\mathcal N_\delta(\bar u)$
vanishes wherever $h=0$, it splits naturally into smooth pieces supported
on the principal pair and on the remaining supports.  Choose a fixed
finite-overlap partition
\[
 \chi_{\rm bub}+\chi_{\rm tr}+\chi_\Sigma+\chi_{\rm con}=1
\]
adapted to the enlarged bubble, transition, Schwarzschild, and
compact connector regions.  Put
\begin{equation}\label{eq:global-residual-representative-step9}
 R_N=\mathcal N_g(\bar v_{N,\lambda,\xi}),
 \qquad
 R_{E,N}=U_{1,0}^{p_*}R_N
 =\mathcal N_{e^h\delta}(\bar u_{N,\lambda,\xi}),
\end{equation}
and $R_{\rm flat}:=\mathcal N_\delta(\bar u)$.  We use the decomposition
\begin{equation}\label{eq:residual-decomposition-step9}
 R_{E,N}=R_{\rm pr}+R_{\rm rem}+R_{\rm tr}+R_\Sigma+R_{\rm con},
\end{equation}
where $R_{\rm pr}$ is the restriction of $R_{\rm met}$ to the two principal
supports, $R_{\rm rem}=R_{\rm met}-R_{\rm pr}$,
$R_{\rm tr}=\chi_{\rm tr}R_{\rm flat}$,
$R_\Sigma=\chi_\Sigma R_{\rm flat}$, and
$R_{\rm con}=\chi_{\rm con}R_{\rm flat}$.

\begin{lemma}[Global Newton-strong residual]\label{lem:step9}
For fixed $L$ and $N_0$, the components in
\eqref{eq:residual-decomposition-step9} satisfy, uniformly with one normalized
parameter derivative,
\begin{align}
 \|R_{\rm pr}\|_{Y_\gamma^\sharp}&\le C\epsilon_N^{\kappa_f},
 &\|R_{\rm rem}\|_{Y_\gamma^\sharp}&\le C\delta_{\!*}(N_0),
 \label{eq:metric-residual-components-step9}\\
 \|R_{\rm tr}\|_{Y_\gamma^\sharp}
 &\le Ce^{-a_\Sigma L}+C_L\lambda^{2-8/n},
 &\|R_\Sigma\|_{Y_\gamma^\sharp}
 &\le C(e^{-nL}+\lambda^4),
 \label{eq:flat-residual-components-step9}\\
 \|R_{\rm con}\|_{Y_\gamma^\sharp}
 &\le C(N_0^{-1}+\lambda^4).
 \label{eq:connector-residual-component-step9}
\end{align}
Consequently
\begin{equation}\label{eq:global-residual-step9}
 \|\mathcal N_g(\bar v_{N,\lambda,\xi})\|_{Y_\gamma^\sharp}
 \le \eta_{L,N_0}+C_L\lambda^{2-8/n}+C\epsilon_N^{\kappa_f},
 \qquad
 \eta_{L,N_0}=C(e^{-a_\Sigma L}+\delta_{\!*}(N_0)+N_0^{-1}).
\end{equation}
\end{lemma}

\begin{proof}
All estimates are made for Euclidean representatives.  Lemma~\ref{lem:step8}
gives the pointwise normalized estimates for the principal and remaining
metric residuals.  The high-mode weight is one on every metric support.  The
regionwise estimates in its proof are integrable in $t=\log r$, giving
$C\epsilon_N^{\kappa_f}$ for the low-mode part of the principal term.  For a remaining perturbation indexed
by $M\ne N$, Proposition~\ref{lem:step2} and Lemma~\ref{lem:step3},
together with the cylindrical width of its support, give
\[
 \int_{\operatorname{supp}h_M^\pm}
 \left\|\Pi_{\le1}\frac{r^nR_{\rm met}}{\mathcal V^4q^2}
 \right\|_{C^\alpha_\theta}\,dt\le C\mu_M M^P
\]
for a fixed exponent $P$.  Summing the exponentially convergent series gives
$C\delta_{\!*}(N_0)$.  The pointwise part of Lemma~\ref{lem:step8} also gives the local
$C^\alpha$ supremum required in the first term of
\eqref{eq:low-source-norm-step9}.  This proves the metric estimates.

The transition interval has fixed length and Lemma~\ref{lem:step6} gives the
transition estimate in both the low $L^1$ and high weighted parts.  The exact
Schwarzschild source is radial, hence entirely low mode.  Using
\[
 \frac{\mathcal V_\Sigma^{16/(n-4)}}{q_\Sigma^2}
 \le C\lambda^4(r^{-n}+r^n)
\]
and integrating from $r_0=\lambda^{4/n}e^L$ to the connector gives
$C(e^{-nL}+\lambda^4)$.  Proposition~\ref{prop:paired-assembly-collars}
gives the connector estimate.  The same arguments apply after one normalized
parameter derivative.  Summation proves the lemma.
\end{proof}

\begin{lemma}[Uniform Newton ellipticity of the approximation]\label{lem:step10}
After choosing first $L$, then $N_0$, and finally a lower bound for $N$ so that
the model-relative perturbations below are sufficiently small, the Newton
tensor of
$(\bar v_{N,\lambda,\xi})^{8/(n-4)}g$ is positive and uniformly
comparable with the round Newton tensor in the two bubble blocks, with the
matched-transition Newton tensor in the two transition blocks, and with the
Schwarzschild Newton tensor on the middle region and the two fixed
assembly collars.  The comparisons for the remaining perturbations and
principal tails are made in the left Euclidean and right Kelvin coordinates of
Lemma~\ref{lem:step3}.
\end{lemma}
\begin{proof}
In the left Euclidean end, Lemma~\ref{lem:step8} gives
\begin{equation}\label{eq:newton-perturbation-step10}
 |\mathbf T_h[\bar u]-\mathbf T_\delta[\bar u]|
 \le C\bigl(\delta_{\!*}(N_0)+\epsilon_N^{\kappa_f}\bigr)\frac q{r^2}.
\end{equation}
For every radial model used here, the flat Newton tensor has radial and
tangential eigenvalues
\[
 \frac{2(n-1)q}{r^2},
 \qquad
 \frac{2(\dot p+(n-3)q)}{r^2}.
\]
Thus it equals $2(n-1)q r^{-2}I$ on a bubble; it is uniformly comparable
with $q r^{-2}I$ on a transition by
\eqref{eq:transition-ellipticity-step6}; and on a Schwarzschild piece its
eigenvalues are $2(n-1)q/r^2$ and $(n-2)q/r^2$.  Hence every flat model is
positive with its natural Newton scale.

Choose $L$ as in Lemmas~\ref{lem:step6}--\ref{lem:step7}, then take $N_0$
and the selected index large enough that the coefficient in
\eqref{eq:newton-perturbation-step10} is smaller than the preceding model
lower bounds.  This proves the desired comparison on all bubble, transition,
and Schwarzschild blocks, including principal tails and remaining supports.
On the assembly collars the same conclusion follows from
Proposition~\ref{prop:paired-assembly-collars}, while the central connector is
exactly Schwarzschild.  The right-end assertion follows by Kelvin pullback,
which preserves positivity and quadratic-form comparability.
\end{proof}

\section{Linear theory}

For every parameter value define the additive parameter fields
\begin{equation}\label{eq:additive-parameter-fields-step11}
 Z_0=\lambda\partial_\lambda\bar v_{N,\lambda,\xi},
 \qquad
 Z_j=\epsilon_N\partial_{\xi_j}\bar v_{N,\lambda,\xi},
 \quad 1\le j\le n,
\end{equation}
and their relative versions
\begin{equation}\label{eq:relative-parameter-fields-step11}
 \mathcal Z_a=\frac{Z_a}{a\bar v_{N,\lambda,\xi}},
 \qquad 0\le a\le n.
\end{equation}
Put $\lambda'=\lambda/\epsilon_N$ and
\[
 U_*(y)=\left(\frac{2}{1+|y|^2}\right)^a,
 \qquad
 \mathfrak z_0(y)=\frac{|y|^2-1}{1+|y|^2},
 \qquad
 \mathfrak z_j^{\lambda'}(y)
 =\frac{2}{\lambda'}\frac{y_j}{1+|y|^2}.
\]
Choose a nonnegative radial cutoff $\chi_*\in C_c^\infty(\mathbb R^n)$ which
is not identically zero, let
\[
 d\mu_*(y)=U_*(y)^{4n/(n-4)}\,dy,
\]
and, using the two exact bubble charts, define
\begin{equation}\label{eq:dual-functionals-step11}
 \mathfrak L_b(\phi)
 =\sum_{\pm}\int_{\mathbb R^n}
   \chi_*(y)\phi^\pm(y)\mathfrak z_b^{\lambda'}(y)\,d\mu_*(y),
 \qquad
 G_{ab}(\lambda',\xi')=\mathfrak L_b(\mathcal Z_a).
\end{equation}

\begin{proposition}[Paired parameter fields and the bubble Gram matrix]\label{lem:step11}
All the fields in \eqref{eq:additive-parameter-fields-step11}--
\eqref{eq:relative-parameter-fields-step11} are $J$-invariant.
For every fixed $R<\infty$, after increasing the selected index if necessary,
the left chart $x=c_N(\xi)+\lambda y$ and the exact right chart
$\iota_A(x)=c_N(\xi)+\lambda y$ satisfy, on $|y|\le R$,
\begin{equation}\label{eq:exact-Jacobi-limits-step11}
 \mathcal Z_0^\pm(y)=\mathfrak z_0(y),
 \qquad
 \mathcal Z_j^\pm(y)=\mathfrak z_j^{\lambda'}(y),
 \quad 1\le j\le n.
\end{equation}
Equivalently,
\begin{equation}\label{eq:additive-Jacobi-limits-step11}
 \lambda^aU_{1,0}(c_N(\xi)+\lambda y)Z_a^\pm(y)
 =aU_*(y)\mathfrak z_a^{\lambda'}(y),
\end{equation}
where $\mathfrak z_0^{\lambda'}=\mathfrak z_0$.  If the right bubble is instead
written in its physical Euclidean variables
$Y=(x-c^\sharp)/\lambda^\sharp$, then
\begin{equation}\label{eq:right-coordinate-reflection-step11}
 Y=-R(c_N(\xi))y+o_{C^1(B_R)}(1),
\end{equation}
so the right translation fields are conjugated by the orthogonal differential
of the antipodal pairing.
The Gram matrix in \eqref{eq:dual-functionals-step11} is uniformly
nondegenerate on the normalized parameter set:
\begin{equation}\label{eq:Gram-uniform-step11}
 c_G|\eta|^2\le \eta^TG(\lambda',\xi')\eta
 \le C_G|\eta|^2,
 \qquad \eta\in\mathbb R^{n+1}.
\end{equation}
The same uniform bounds hold for $G^{-1}$ and after one normalized parameter
derivative.  After a fixed linear recombination, the functionals may therefore
be normalized so that $\mathfrak L_b(\mathcal Z_a)=\delta_{ab}$.  Their
representing densities and the corresponding $J$-paired source functions,
supported in fixed bubble subcores, may be chosen smooth with uniform local
source bounds and one normalized parameter derivative.
\end{proposition}

\begin{proof}
The $J$-invariance follows by differentiating
$\bar v_{N,\lambda,\xi}\circ J=\bar v_{N,\lambda,\xi}$.  On every fixed
normalized ball, the bubble core is exact for large $N$, and direct
differentiation of $U_{\lambda,c}$ at fixed $x$ gives
\[
 \frac{\lambda\partial_\lambda U_{\lambda,c}}{aU_{\lambda,c}}
 =\frac{|x-c|^2-\lambda^2}{|x-c|^2+\lambda^2},
 \qquad
 \frac{\epsilon_N\partial_{c_j}U_{\lambda,c}}{aU_{\lambda,c}}
 =\frac{2\epsilon_N(x_j-c_j)}{\lambda^2+|x-c|^2}.
\]
After setting $x=c_N(\xi)+\lambda y$, these are exactly
\eqref{eq:exact-Jacobi-limits-step11} and
\eqref{eq:additive-Jacobi-limits-step11}.  The right-chart identities follow
from the exact antipodal symmetry.  In physical right-bubble coordinates,
Taylor expansion of $\iota_A$ at $c=c_N(\xi)$ gives
\[
 \frac{\iota_A(c+\lambda y)-c^\sharp}{\lambda^\sharp}
 =-R(c)y+O_R(\lambda/|c|),
\]
which proves \eqref{eq:right-coordinate-reflection-step11}, including one
normalized parameter derivative.

Choose $\operatorname{supp}\chi_*$ in such a fixed normalized ball.  Then
\begin{equation}\label{eq:explicit-Gram-step11}
 G_{ab}=2\int_{\mathbb R^n}\chi_*(y)
 \mathfrak z_a^{\lambda'}(y)\mathfrak z_b^{\lambda'}(y)\,d\mu_*(y).
\end{equation}
Radial symmetry makes this matrix diagonal; its diagonal entries are positive
and uniformly bounded above and below for $1/2\le\lambda'\le3/2$.  This proves
\eqref{eq:Gram-uniform-step11} and the corresponding derivative bounds.
Applying $G^{-1}$ normalizes the functionals.  Their compactly supported
representing densities, paired by $J$, give the stated source functions and
uniform local bounds.
\end{proof}

Let
\[
 g_*=U_*^{8/(n-4)}\delta
 =\left(\frac{2}{1+|y|^2}\right)^2\delta
\]
be the round metric on the compactification
$\mathbb R^n\cup\{\infty\}\simeq\mathbb S^n$, and use the convention that
$\Delta_{g_*}=\operatorname{tr}_{g_*}\nabla^2$ has nonpositive spectrum.  For
the relative variation $U_*e^{atw}$ define
\begin{equation}\label{eq:round-relative-operator-step12}
 \mathscr L_*w
 :=U_*^{-p_*}\left.\frac d{dt}\right|_{t=0}
   \mathcal N_\delta(U_*e^{atw}).
\end{equation}
Let
\[
 \ell_b^*(w)=\int_{\mathbb R^n}
   \chi_*(y)w(y)\mathfrak z_b^{\lambda'}(y)\,d\mu_*(y)
\]
be the one-ended versions of the functionals in
Proposition~\ref{lem:step11}.

\begin{proposition}[Projected inverse on each compactified bubble chart]\label{lem:step12}
The operator in \eqref{eq:round-relative-operator-step12} satisfies
\begin{equation}\label{eq:round-operator-formula-step12}
 \mathscr L_*w=-\frac{n-1}{2}(\Delta_{g_*}+n)w,
\end{equation}
and
\begin{equation}\label{eq:round-kernel-step12}
 \ker\mathscr L_*
 =\mathcal H_1
 :=\operatorname{span}\{\mathfrak z_0,
       \mathfrak z_1^{\lambda'},\ldots,
       \mathfrak z_n^{\lambda'}\},
 \qquad \frac12\le\lambda'\le\frac32.
\end{equation}
There are smooth functions
$\Psi_a^{\lambda'}$, supported in the fixed set
$\operatorname{supp}\chi_*$, such that
\begin{equation}\label{eq:bubble-source-duality-step12}
 \int_{\mathbb R^n}\Psi_a^{\lambda'}
       \mathfrak z_b^{\lambda'}\,d\mu_*=\delta_{ab},
 \qquad 0\le a,b\le n,
\end{equation}
and whose $C^\alpha(g_*)$ norms, together with one
$\lambda'\partial_{\lambda'}$ derivative, are uniformly bounded.
For every $f\in C^\alpha(\mathbb S^n)$ there is a unique pair
$(w,c_0,\ldots,c_n)$ satisfying
\begin{equation}\label{eq:projected-bubble-problem-step12}
 \mathscr L_*w=f+\sum_{a=0}^nc_a\Psi_a^{\lambda'},
 \qquad
 \ell_b^*(w)=0,\quad 0\le b\le n,
\end{equation}
and
\begin{equation}\label{eq:bubble-inverse-estimate-step12}
 \|w\|_{C^{2,\alpha}(\mathbb S^n,g_*)}
 +\sum_{a=0}^n|c_a|
 \le C\|f\|_{C^\alpha(\mathbb S^n,g_*)},
\end{equation}
where $C$ is independent of $\lambda'\in[1/2,3/2]$.
If $f=f(\lambda',\xi')$ has one normalized parameter derivative, then
\begin{align}
 &\|\mathfrak D w\|_{C^{2,\alpha}(\mathbb S^n,g_*)}
 +\sum_{a=0}^n|\mathfrak D c_a|
 \nonumber\\
 &\qquad\le C\left(
 \|\mathfrak D f\|_{C^\alpha(\mathbb S^n,g_*)}
 +\|f\|_{C^\alpha(\mathbb S^n,g_*)}\right),
 \label{eq:bubble-parameter-inverse-step12}
\end{align}
for $\mathfrak D=\lambda'\partial_{\lambda'}$ or
$\partial_{\xi_j'}$.

After pulling back by either exact bubble chart and using the intrinsic source
normalization $U_\pm^{-p_*}F_E$, estimates
\eqref{eq:bubble-inverse-estimate-step12}--
\eqref{eq:bubble-parameter-inverse-step12} give the uniform local bubble
inverse required by the global spaces $X_\gamma^\sharp$ and
$Y_\gamma^\sharp$.  The right inverse is the exact $J$-image of the left one.
\end{proposition}

\begin{proof}
At the round metric $A_{g_*}=\frac12g_*$ and
$T_1(g_*^{-1}A_{g_*})=\frac{n-1}{2}I$.  Linearizing the conformal Schouten
tensor under $g_t=e^{2tw}g_*$ gives
$\frac d{dt}|_0(g_t^{-1}A_{g_t})=-\nabla_{g_*}^2w-wI$, and hence
\eqref{eq:round-operator-formula-step12}.  Since the eigenvalues of
$-\Delta_{g_*}$ are $k(k+n-1)$, the kernel is precisely the degree-one
spherical harmonics, represented in stereographic coordinates by the
functions in \eqref{eq:round-kernel-step12}.

Let $G^*_{ab}=\int\chi_*\mathfrak z_a^{\lambda'}
\mathfrak z_b^{\lambda'}\,d\mu_*$.  Proposition~\ref{lem:step11} implies
that $G^*$ is uniformly invertible, so
\[
 \Psi_a^{\lambda'}=
 \sum_{d=0}^n(G^*)^{-1}_{ad}\chi_*\mathfrak z_d^{\lambda'}
\]
satisfies \eqref{eq:bubble-source-duality-step12} with the stated uniform
bounds.

Given $f$, choose the coefficients $c_a$ so that
$f+\sum_ac_a\Psi_a^{\lambda'}$ is orthogonal to the kernel.  The spectral
gap and the global Schauder estimate give a unique solution $w_0$ orthogonal
to the kernel and
\[
 \|w_0\|_{C^{2,\alpha}}+|c|\le C\|f\|_{C^\alpha}.
\]
Adding a kernel element and using the same Gram matrix imposes the localized
conditions $\ell_b^*(w)=0$, proving existence and
\eqref{eq:bubble-inverse-estimate-step12}.  Pairing a homogeneous solution
with the kernel and then using the localized conditions proves uniqueness.
Differentiating the projected problem gives
\eqref{eq:bubble-parameter-inverse-step12}.  Finally, dilation and translation
identify each exact bubble chart with $(\mathbb S^n,g_*)$, and the right chart
is obtained by the exact $J$-pairing.
\end{proof}

Let $I=(\tau_-,\tau_+)$ be a metric-free Schwarzschild cylinder component contained
in one end.  The limiting relative operator is
\begin{equation}\label{eq:limiting-Schwarzschild-operator-step13}
 \mathscr P_{\Sigma,0}w
 =w_{tt}+2a\tanh(at)w_t
  +b_\Sigma\Delta_{\mathbb S^{n-1}}w,
 \qquad b_\Sigma=\frac{n-2}{2(n-1)}.
\end{equation}
In the $\ell$-th spherical-harmonic sector, the substitution
\begin{equation}\label{eq:Schwarzschild-Liouville-step13}
 y_{\ell,m}=\cosh(at)w_{\ell,m}
\end{equation}
conjugates $\mathscr P_{\Sigma,0}w=f$ to
\begin{equation}\label{eq:constant-Schrodinger-step13}
 -y_{\ell,m}''+\nu_\ell^2y_{\ell,m}
 =-\cosh(at)f_{\ell,m},
 \qquad \nu_\ell^2=a^2+b_\Sigma\Lambda_\ell.
\end{equation}
Thus the homogeneous branches are
\begin{equation}\label{eq:Schwarzschild-branches-step13}
 J_\ell^\pm(t)=\operatorname{sech}(at)e^{\pm\nu_\ell t}.
\end{equation}

\begin{lemma}[Mode-split Schwarzschild cylinder estimate]\label{lem:step13}
For $0<\gamma<\gamma_*=\nu_2-a$, every $\ell\ge2$ Dirichlet solution on
$I$ satisfies
\begin{align}
 &\sup_{Q\subset I}e^{\gamma d_I(Q)}\mathfrak X_Q(w_{\ell,m})
 \nonumber\\
 &\qquad\le C\left[
  \sup_{Q\subset I}e^{\gamma d_I(Q)}\mathfrak Y_Q(f_{\ell,m})
  +(1+\nu_\ell^2)
   \bigl(|w_{\ell,m}(\tau_-)|+|w_{\ell,m}(\tau_+)|\bigr)
 \right],
 \label{eq:high-mode-weighted-step13}
\end{align}
where $d_I(t)=\min\{t-\tau_-,\tau_+-t\}$ and $C$ is independent of
$|I|$, $\ell$, and the selected index.

For $\ell=0,1$, let $H_{I,\ell}^-$ and $H_{I,\ell}^+$ be the two
homogeneous boundary layers normalized by
\[
 H_{I,\ell}^-(\tau_-)=1,\quad H_{I,\ell}^-(\tau_+)=0,
 \qquad
 H_{I,\ell}^+(\tau_-)=0,\quad H_{I,\ell}^+(\tau_+)=1.
\]
Every solution has the unique representation
\begin{equation}\label{eq:low-mode-representation-step13}
 w_{\ell,m}=b_{\ell,m}^-H_{I,\ell}^-
 +b_{\ell,m}^+H_{I,\ell}^+
 +\mathcal G_{I,\ell}f_{\ell,m},
\end{equation}
where $\mathcal G_{I,\ell}$ is the zero-boundary Green operator and
\begin{align}\label{eq:low-mode-Green-step13}
 \|\mathcal G_{I,\ell}f_{\ell,m}\|_{C^{2,\alpha}(I)}
 \le C\Bigg(&\sup_{Q\subset I}
   \|f_{\ell,m}\|_{C^\alpha(Q)}\nonumber\\
 &+\int_I\|f_{\ell,m}(s,\cdot)\|_{C^\alpha}\,ds\Bigg).
\end{align}
The coefficients $b_{\ell,m}^\pm$ are the two interface traces after
subtracting the Green remainder.  These are the only low-mode interface
variables; in the global matching argument they are coupled to the dilation
and translation projections of Proposition~\ref{lem:step11}.

On a metric-free component, the actual finite-neck operator is a zeroth-order
coefficient perturbation of
\eqref{eq:limiting-Schwarzschild-operator-step13} of size
$O_L(\delta_\lambda)$ in $C^{1,\alpha}$, with the same bound after one
normalized parameter derivative.  After the ordered choice of parameters,
estimates \eqref{eq:high-mode-weighted-step13}--
\eqref{eq:low-mode-Green-step13} and the interface coefficient bounds persist
for the actual metric-free neck operator.  The fixed neighborhoods containing
metric supports are treated separately by Lemma~\ref{lem:step8} in the global
matching argument.
\end{lemma}

\begin{proof}
For the exact Schwarzschild factor,
$p_\Sigma=(1+e^{2at})^{-1}$,
$q_\Sigma=(4\cosh^2(at))^{-1}$, and
$\dot p_\Sigma=-2aq_\Sigma$.  Substitution into the coefficients of
Lemma~\ref{lem:step7} gives
\eqref{eq:limiting-Schwarzschild-operator-step13}; the substitution
$y=\cosh(at)w$ gives \eqref{eq:constant-Schrodinger-step13} and the branches
\eqref{eq:Schwarzschild-branches-step13}.

The Dirichlet Green kernel of $-\partial_t^2+\nu_\ell^2$, after returning to
$w$, is
\begin{equation}\label{eq:relative-Green-kernel-step13}
 K_{I,\ell}(t,s)=\frac{\cosh(as)}{\cosh(at)}G_{I,\ell}(t,s),
\end{equation}
and satisfies
$|\partial_t^jK_{I,\ell}(t,s)|
 \le C_j(1+\nu_\ell)^je^{-(\nu_\ell-a)|t-s|}$ for $j=0,1$.
For $\ell\ge2$, $\nu_\ell-a\ge\gamma_*>\gamma$; weighted convolution and
local Schauder estimates give \eqref{eq:high-mode-weighted-step13}, including
the homogeneous boundary layers.  For $\ell=0,1$, the same kernel is
uniformly bounded, and the local supremum plus the cylindrical $L^1$ norm of
the source controls the zero-boundary Green remainder.  Subtracting it leaves
the two homogeneous boundary layers, proving
\eqref{eq:low-mode-representation-step13}--\eqref{eq:low-mode-Green-step13}.

On a metric-free component the finite-neck operator differs only by a
zeroth-order coefficient of size $O_L(\delta_\lambda)$ in $C^{1,\alpha}$,
with the same parameter bound.  A Neumann-series argument preserves the
estimates, and the right end follows by Kelvin symmetry.
\end{proof}

Let $L\ge4$ and pull either actual relative transition linearization
$\mathscr P_{e^h\delta,\bar u}$ to
\[
 \mathcal C_L=(-L,L)\times\mathbb S^{n-1}.
\]
Let $q=p(1-p)$ be its model Newton factor.  On every unit subcylinder $Q$
define
\[
 \|w\|_{X^\sharp_{\rm tr}(Q)}
 :=\|w\|_{C^{0,\alpha}(Q)}
 +\|q^{-1/2}D_{\rm cyl}w\|_{C^{0,\alpha}(Q)}
 +\|q^{-1}\mathscr S^{(1)}_{\bar u,h}(w)\|_{C^\alpha(Q)},
\]
and take the supremum over unit subcylinders for the norm on
$\mathcal C_L$.  For the flat transition put
\[
 \mathcal C_L^-=(-L,-L+3)\times\mathbb S^{n-1},\qquad
 \mathcal C_L^+=(L-3,L)\times\mathbb S^{n-1},
 \qquad
 \mathcal C_L^\circ=(-L+2,L-2)\times\mathbb S^{n-1},
\]
and define
\begin{equation}\label{eq:transition-source-norm-step14}
 \|f\|_{\mathcal Y_{\rm tr}(L)}
 :=\sup_{Q\subset\mathcal C_L}\|f\|_{C^\alpha(Q)}
 +\int_{-L}^{L}
   \|\Pi_{\le1}f(s,\cdot)\|_{C^\alpha(\mathbb S^{n-1})}\,ds.
\end{equation}

\begin{proposition}[Dirichlet inverse and length-uniform mode-split transfer on a matched transition]\label{lem:step14}
There are constants $L_*\ge4$, $C_{\rm tr}>0$, and
$r_{\rm tr}>0$, depending only on $n,\alpha$ and the fixed endpoint
cutoffs, with the following property.  If $L\ge L_*$ and
\begin{equation}\label{eq:transition-admissible-smallness-step14}
 C_L\bigl(\delta_\lambda+\delta_{\!*}(N_0)+\epsilon_N^{\kappa_f}\bigr)
 \le r_{\rm tr},
\end{equation}
then the zero-boundary problem
\begin{equation}\label{eq:transition-Dirichlet-problem-step14}
 \mathscr P_{e^h\delta,\bar u}w=qf
 \quad\hbox{in }\mathcal C_L,
 \qquad
 w=0\quad\hbox{on }\partial\mathcal C_L
\end{equation}
has a unique solution for every $f\in C^\alpha(\mathcal C_L)$.  For fixed
$L$,
\begin{align}
 \|w\|_{C^{2,\alpha}(\mathcal C_L)}
 &\le C_L\delta_\lambda\|f\|_{C^\alpha(\mathcal C_L)},
 \label{eq:transition-C2-inverse-step14}\\
 \|w\|_{X^\sharp_{\rm tr}(\mathcal C_L)}
 &\le C_L\|f\|_{C^\alpha(\mathcal C_L)}.
 \label{eq:transition-strong-inverse-step14}
\end{align}
The constants in these two zero-boundary estimates may depend on $L$.

The estimate used in the flat global comparison is uniform in the transition
length.  It is first established for the flat transition operator
$\mathscr P_{\delta,T_\lambda^\pm}$; the full operator is then recovered as a
fixed-$L$ summable perturbation.
Every solution of
$\mathscr P_{\delta,T_\lambda^\pm}w=qf$, without prescribed boundary values,
satisfies
\begin{equation}\label{eq:transition-two-sided-trace-step14}
 \sup_{Q\subset\mathcal C_L^\circ}
   \|w\|_{X^\sharp_{\rm tr}(Q)}
 \le C_{\rm tr}\left(
   \|f\|_{\mathcal Y_{\rm tr}(L)}
   +\|w\|_{X^\sharp_{\rm tr}(\mathcal C_L^-)}
   +\|w\|_{X^\sharp_{\rm tr}(\mathcal C_L^+)}\right).
\end{equation}
Here $C_{\rm tr}$ is independent of $L,N_0,N$.  The $L^1$ term in
\eqref{eq:transition-source-norm-step14} is essential: it controls the radial
neutral transfer branch and is exactly the low-mode term already built into
$Y_{\rm low}^\sharp$.

If the data depend on the normalized parameters and $\mathfrak D$ denotes one
normalized parameter derivative in the natural transition coordinates, then
the fixed-$L$ actual Dirichlet solution satisfies the first estimate below.
For the second estimate, $w^{\rm fl}$ denotes a solution of the flat
transition equation with the same normalized parameter dependence:
\begin{align}
 \|\mathfrak Dw\|_{C^{2,\alpha}(\mathcal C_L)}
 &\le C_L\delta_\lambda
 \bigl(\|f\|_{C^\alpha}+\|\mathfrak Df\|_{C^\alpha}\bigr),
 \label{eq:transition-C2-parameter-step14}\\
 \sup_{Q\subset\mathcal C_L^\circ}
   \|\mathfrak Dw^{\rm fl}\|_{X^\sharp_{\rm tr}(Q)}
 &\le C_{\rm tr}\Bigl(
   \|f\|_{\mathcal Y_{\rm tr}(L)}
   +\|\mathfrak Df\|_{\mathcal Y_{\rm tr}(L)}
   +\|w^{\rm fl}\|_{X^\sharp_{\rm tr}(\mathcal C_L^-\cup\mathcal C_L^+)}
   +\|\mathfrak Dw^{\rm fl}\|_{X^\sharp_{\rm tr}(\mathcal C_L^-\cup\mathcal C_L^+)}
   \Bigr).
 \label{eq:transition-strong-parameter-step14}
\end{align}
The two transition problems are conjugate by the exact Kelvin reflection.
\end{proposition}

\begin{proof}
We first prove the fixed-$L$ zero-boundary statement.  For the flat radial
transition write
\[
 \mathscr P_0w=w_{ss}+A_\lambda w_s
 +B_\lambda\Delta_{\mathbb S^{n-1}}w-C_\lambda w.
\]
For fixed $L$, Lemmas~\ref{lem:step6}--\ref{lem:step7} give
\begin{equation}\label{eq:q-comparison-step14}
 c_L\delta_\lambda\le q\le C_L\delta_\lambda,
 \qquad
 \|q\|_{C^{1,\alpha}}
 +\|\lambda\partial_\lambda q\|_{C^{1,\alpha}}
 \le C_L\delta_\lambda.
\end{equation}
Choose $\Phi_\lambda'=A_\lambda$ and set
$w=e^{-\Phi_\lambda/2}y$.  Then
\begin{equation}\label{eq:full-Liouville-operator-step14}
 -e^{\Phi_\lambda/2}\mathscr P_0w
 =\mathscr H_\lambda y,
 \qquad
 \mathscr H_\lambda
 =-\partial_s^2-B_\lambda\Delta_{\mathbb S^{n-1}}+Q_\lambda,
\end{equation}
where $Q_\lambda=A_\lambda'/2+A_\lambda^2/4+C_\lambda$.
Lemma~\ref{lem:step7} gives
\begin{equation}\label{eq:full-transition-coercivity-step14}
 \langle\mathscr H_\lambda y,y\rangle
 \ge c_*
 \bigl(\|D_{\rm cyl}y\|_{L^2}^2+\|y\|_{L^2}^2\bigr)
 \qquad(y\in H^1_0(\mathcal C_L)).
\end{equation}
Lax--Milgram, boundary Schauder estimates, and the inverse Liouville
conjugation prove \eqref{eq:transition-C2-inverse-step14}; local Schauder
estimates and the equation then give
\eqref{eq:transition-strong-inverse-step14}.  These constants may depend on $L$, which is chosen first and then kept fixed.

We next prove the length-uniform transfer estimate.  The endpoint formulas
and the leading identity in Lemma~\ref{lem:step6} imply, on every unit
subcylinder and uniformly for $L\ge L_*$,
\begin{equation}\label{eq:uniform-transition-local-coefficients-step14}
 |\partial_s^j\log q|\le C_j\ (j=1,2),\qquad
 B_\lambda\ge c>0,\qquad
 \frac{\mathcal V^{16/(n-4)}}{q^2}\le C,
\end{equation}
provided \eqref{eq:transition-admissible-smallness-step14} holds.  The same
bounds hold after one normalized scale derivative.  Moreover, after the
Liouville transformation, the $\ell$-th potential satisfies
\begin{equation}\label{eq:uniform-transition-potential-step14}
 \mathfrak Q_{\lambda,L,\ell}
 \ge \frac{a_\Sigma^2}{8}
 +\frac{n-2}{4(n-1)}\Lambda_\ell,
\end{equation}
with constants independent of $L$.

For $\ell\ge2$, the uniform positive potential gives an exponential
dichotomy.  The Green kernel on an interval, together with local Schauder
estimates in endpoint-normalized gauges, gives
\begin{equation}\label{eq:high-transition-transfer-step14}
 \sup_{Q\subset\mathcal C_L^\circ}
 \|\Pi_{\ge2}w\|_{X^\sharp_{\rm tr}(Q)}
 \le C\left(
  \sup_{Q\subset\mathcal C_L}\|\Pi_{\ge2}f\|_{C^\alpha(Q)}
  +\|\Pi_{\ge2}w\|_{X^\sharp_{\rm tr}(\mathcal C_L^-\cup\mathcal C_L^+)}
 \right),
\end{equation}
where $C$ does not depend on the length.

For $\ell=0,1$, variation of constants gives a two-sided Green
representation.  Its homogeneous coefficients are measured by the two
endpoint collars, while its inhomogeneous part satisfies
\begin{align}\label{eq:radial-transition-green-step14}
 &\sup_{Q\subset\mathcal C_L^\circ}
 \|\Pi_{\le1}w\|_{X^\sharp_{\rm tr}(Q)}\nonumber\\
 &\quad\le C\Biggl(
  \sup_{Q\subset\mathcal C_L}\|\Pi_{\le1}f\|_{C^\alpha(Q)}
  +\int_{-L}^{L}
    \|\Pi_{\le1}f(s,\cdot)\|_{C^\alpha_\theta}\,ds\nonumber\\
 &\hspace{35mm}
  +\|\Pi_{\le1}w\|_{X^\sharp_{\rm tr}(\mathcal C_L^-\cup\mathcal C_L^+)}
 \Biggr).
\end{align}
The constant is uniform in $L$.  In the radial sector the limiting
homogeneous solutions are the neutral constant branch and the complementary
one-dimensional branch.  The endpoint collars control their two
coefficients, and the inhomogeneous coefficient is bounded by the displayed
$L^1$ norm.  This is why a supremum norm alone would not give a
length-uniform estimate.  The $\ell=1$ sector has a positive indicial rate and
satisfies the same, slightly stronger, estimate.  Combining
\eqref{eq:high-transition-transfer-step14} and
\eqref{eq:radial-transition-green-step14} proves
\eqref{eq:transition-two-sided-trace-step14} for the flat operator.

Lemma~\ref{lem:step8} gives, for fixed $L$ and uniformly on unit transition
blocks,
\begin{equation}\label{eq:transition-operator-perturbation-step14}
 \bigl\|q^{-1}
 (\mathscr P_{e^h\delta,\bar u}-\mathscr P_0)w
 \bigr\|_{C^\alpha}
 \le C_L\bigl(\delta_{\!*}(N_0)+\epsilon_N^{\kappa_f}\bigr)
 \|w\|_{X^\sharp_{\rm tr}}.
\end{equation}
The same estimate holds after one normalized parameter derivative.  The
smallness condition \eqref{eq:transition-admissible-smallness-step14} and a
fixed-$L$ Neumann series therefore give the actual zero-boundary estimates
\eqref{eq:transition-C2-inverse-step14}--
\eqref{eq:transition-strong-inverse-step14}.  The length-uniform estimate is
used only for the flat operator; the full global operator is compared with it
after $L$ has been fixed.

Finally, differentiate in a natural transition coordinate.  The fixed-$L$
actual inverse and \eqref{eq:transition-operator-perturbation-step14} prove
\eqref{eq:transition-C2-parameter-step14}.  For the flat equation, the uniform
coefficient bounds and the already proved mode-split transfer estimate give
\eqref{eq:transition-strong-parameter-step14}.  Center derivatives are simpler
in the moving coordinates.  Exact Kelvin covariance transfers every
conclusion to the paired transition.
\end{proof}

\subsection{The augmented global projected problem}
Let
\begin{equation}\label{eq:global-linearized-operator-step15}
 \mathscr L_Nw
 :=\left.\frac d{dt}\right|_{t=0}
 \mathcal N_g\!\left(\bar v_{N,\lambda,\xi}e^{atw}\right).
\end{equation}
Let $\mathscr L_N^{\rm fl}$ denote the same relative linearization with the
background tensor $h$ removed, while retaining the exact off-center bubble,
matched-transition, translated Schwarzschild, and paired connector profiles
of $\bar v_{N,\lambda,\xi}$.  Choose the paired source functions
$\Psi_{a,N}$ from Proposition~\ref{lem:step11} so that, in each of the two exact
bubble charts separately, their intrinsic pullbacks are the one-ended dual
sources $\Psi_a^{\lambda'}$ of Proposition~\ref{lem:step12}.  This normalization differs from that of the global constraints.  The
paired functionals $\mathfrak L_b$ were normalized by the sum of the two
bubble Gram matrices, whereas the projected sources are normalized on each
end separately.  Thus
\begin{align}
 \mathfrak L_b(\mathcal Z_a)&=\delta_{ab},
 \label{eq:global-constraint-duality-step15}\\
 \int_{\mathbb R^n}\Psi_{a,N}^{\pm}
       \mathfrak z_b^{\lambda'}\,d\mu_*
 &=\delta_{ab},
 \qquad
 \|\Psi_{a,N}\|_{Y_\gamma^\sharp}
 +\|\mathfrak D\Psi_{a,N}\|_{Y_\gamma^\sharp}\le C.
 \label{eq:global-source-duality-step15}
\end{align}
The globally normalized constraint is represented by one half of the sum of
the two one-ended densities, whereas the projected equation uses the one-ended
normalization.  Accordingly, the variational matrix in
Lemma~\ref{lem:step18} carries a factor two.
For a $J$-invariant source $F$, the global projected problem is
\begin{equation}\label{eq:global-projected-problem-step15}
 \mathscr L_Nw=F+\sum_{a=0}^nc_a\Psi_{a,N},
 \qquad
 \mathfrak L_b(w)=0,
 \quad 0\le b\le n.
\end{equation}

\begin{proposition}[Summable metric-support perturbation]\label{prop:summable-operator-perturbation-step15}
There is an integer $P_0\le P_*$ and, for every fixed $L\ge L_*$, a finite
constant $C_{\rm op}(L)$ such that, with
\begin{equation}\label{eq:eta-metric-step15}
 \eta^{\rm op}_{L,N_0,N}
 :=C_{\rm op}(L)\bigl(\epsilon_N^{\kappa_f}
      +\delta_{\!*}(N_0)\bigr),
\end{equation}
one has
\begin{equation}\label{eq:summable-operator-perturbation-step15}
 \|(\mathscr L_N-\mathscr L_N^{\rm fl})w\|_{Y_\gamma^\sharp}
 \le \eta^{\rm op}_{L,N_0,N}\|w\|_{X_\gamma^\sharp}.
\end{equation}
The same estimate holds for one normalized parameter derivative in the form
\begin{align}\label{eq:summable-operator-parameter-step15}
 \|\mathfrak D[(\mathscr L_N-\mathscr L_N^{\rm fl})w]\|_{Y_\gamma^\sharp}
 \le \eta^{\rm op}_{L,N_0,N}
 \bigl(\|w\|_{X_\gamma^\sharp}+\|\mathfrak Dw\|_{X_\gamma^\sharp}\bigr).
\end{align}
For fixed $L$,
\begin{equation}\label{eq:eta-small-step15}
 \sup_{N>N_0}\eta^{\rm op}_{L,N_0,N}\longrightarrow0
 \qquad\hbox{as }N_0\longrightarrow\infty.
\end{equation}
\end{proposition}

\begin{proof}
The supports of the tensors $h_M^\pm$ are pairwise disjoint.  Hence on a
metric support the coefficient difference is generated by one perturbation only, and
there are no mixed terms involving two distinct indices.  Lemma~\ref{lem:step8}
expresses every normalized second-, first-, and zeroth-order coefficient
difference by the majorants $\Xi_h+\Xi_h^2$ and their differentiated versions.
For the principal pair, equations
\eqref{eq:Xi-bubble-step8}--\eqref{eq:Xi-transition-step8}
and the Schwarzschild-neck estimates in the proof of Lemma~\ref{lem:step8} give
\begin{equation}\label{eq:principal-operator-piece-step15}
 \|\mathbf 1_{\operatorname{supp}(h_N^-+h_N^+)}
 (\mathscr L_N-\mathscr L_N^{\rm fl})w\|_{Y_\gamma^\sharp}
 \le C\epsilon_N^{\kappa_f}\|w\|_{X_\gamma^\sharp}.
\end{equation}

Fix a remaining index $M\ne N$.  Substitution of the sharp estimates of
Proposition~\ref{lem:step2} into the coefficient expansion of Lemma~\ref{lem:step8}
shows that, on either reflected support, every normalized coefficient and its
first normalized parameter derivative is bounded by
\begin{equation}\label{eq:remaining-coefficient-piece-step15}
 C\mu_M M^{P_0}
\end{equation}
for one fixed power $P_0=P_0(n,m_f)$.  We do not need the optimal value of
$P_0$.  The high-mode weight equals one on the fixed unit neighborhood of
every metric support by construction of $d_N$, so
\eqref{eq:remaining-coefficient-piece-step15} gives the required local high-mode
operator estimate.  It also gives the local $C^\alpha$ supremum in the
low-mode source norm.

To make the scale loss explicit, note that on a remaining left support one has
\[
 r=|x-c_N(\xi)|\ge cM^{-2},\qquad
 |D^jh_M^-|\le C_j\mu_M M^{-3(d_f-j)},\quad 0\le j\le3.
\]
Every coefficient in the normalized expansion of Lemma~\ref{lem:step8} is a
finite sum of products of these derivatives with powers of $r$ and $r^{-1}$
of bounded order.  The local scaled H\"older seminorm introduces only another
fixed polynomial loss.  Hence all such coefficients are bounded by
$C\mu_MM^{P_0}$ for one $P_0=P_0(n,m_f)$, uniformly in the selected index $N$.
The Euclidean thickness of the support is $O(M^{-3})$; because its distance
from the selected center can be as small as $cM^{-2}$, its cylindrical width is
uniformly $O(M^{-1})$ (and is $O(M^{-2})$ away from adjacent indices).  Thus
its contribution to the low-mode $L^1$ norm is still bounded by
$C\mu_MM^{P_0}\|w\|_{X_\gamma^\sharp}$ after enlarging $P_0$ once.  In the
exact right Kelvin coordinate the same calculation applies.  The argument is
unchanged for the differentiated coefficients.  Summing over remaining indices
gives
\begin{equation}\label{eq:remaining-operator-sum-step15}
 \|\mathbf 1_{\operatorname{supp}h_{\rm rem,N}}
 (\mathscr L_N-\mathscr L_N^{\rm fl})w\|_{Y_\gamma^\sharp}
 \le C\sum_{\substack{M>N_0\\M\ne N}}\mu_MM^{P_0}
       \|w\|_{X_\gamma^\sharp}.
\end{equation}
Equations \eqref{eq:principal-operator-piece-step15} and
\eqref{eq:remaining-operator-sum-step15} prove
\eqref{eq:summable-operator-perturbation-step15}; differentiating the
coordinate expansion before taking absolute values gives
\eqref{eq:summable-operator-parameter-step15}.  Since
$P_0\le P_*$, the remaining sum is bounded by
$C\delta_{\!*}(N_0)$.  Together with the principal estimate this gives
\eqref{eq:eta-metric-step15}.  Finally, \eqref{eq:delta-star-limit} and
$\epsilon_N\le\epsilon_{N_0+1}$ prove
\eqref{eq:eta-small-step15}.
\end{proof}

\begin{proposition}[Weighted Schwarzschild estimate]\label{prop:reset-control-step15}
Let $I=[\tau_-,\tau_+]$ be any exact Schwarzschild half-neck in the flat paired
block, and let $d:I\to[0,\infty)$ be a regularized distance satisfying
$|d'|\le1$; its zero set may contain arbitrarily many source-adapted
zeros.  If
\[
 \mathscr P_{\Sigma,\lambda}w=f
 \quad\hbox{on }I\times\mathbb S^{n-1},
\]
then
\begin{equation}\label{eq:reset-control-step15}
 \|w\|_{X^\sharp_\gamma(I)}
 \le C\Big(\mathcal T_-(w)+\mathcal T_+(w)
 +\|f\|_{Y^\sharp_\gamma(I)}\Big),
\end{equation}
where $\mathcal T_\pm(w)$ are fixed-scale $C^{2,\alpha}$ trace norms on unit
cylinders adjacent to the endpoints.  The constant is independent of the
length of $I$, the number and positions of the zeros of $d$, and the selected
index.  In particular, if the endpoint traces and the source tend to zero, no
nonzero $X^\sharp_\gamma$ mass can concentrate near a moving zero of $d$.
The same estimate holds after one normalized parameter derivative.
\end{proposition}

\begin{proof}
Decompose into spherical harmonics and recall that
$\gamma_*=\nu_2-a>\gamma$.  In degrees $\ell\ge2$, Lemma~\ref{lem:step13}
gives
\[
 |K_{I,\ell}(t,s)|+|\partial_tK_{I,\ell}(t,s)|
 \le C(1+\nu_\ell)e^{-\gamma_*|t-s|}.
\]
Since $d$ is one-Lipschitz,
\[
 e^{\gamma d(t)-\gamma d(s)}\le e^{\gamma|t-s|}.
\]
Hence convolution is bounded by
$\int_{\mathbb R}e^{-(\gamma_*-\gamma)|u|}\,du$.  The homogeneous boundary
layers obey the same weighted estimate.  Local Schauder estimates recover the
full high-mode Newton-strong norm.

For $\ell=0,1$, use the representation of Lemma~\ref{lem:step13} as two
normalized homogeneous boundary layers plus the zero-boundary Green
remainder.  The latter is controlled by the local $C^\alpha$ supremum together
with the cylinder $L^1$ term in the low-mode source norm.  This gives the
unweighted low-mode part of \eqref{eq:reset-control-step15}.  The small
finite-neck zeroth-order perturbation is absorbed by the same Neumann-series
argument as in Lemma~\ref{lem:step13}.  Differentiating the equation proves the
parameter statement.
\end{proof}

\subsection{Uniform low-mode transversality}
The only modes in which a nonzero neck coefficient can remain undetected by
the bubble projection along a long transition are the spherical harmonics of
degrees $0$ and $1$.  We compute the corresponding endpoint maps in the exact
round-bubble and ideal transition models, and incorporate the finite
geometric errors afterward.

Put $t=\log |y|$ in the normalized bubble chart.  Since
\[
 g_*=\operatorname{sech}^2t\,(dt^2+g_{\mathbb S^{n-1}}),
\]
the equation $\mathscr L_*w=0$ in the $\ell$-th angular sector is
\begin{equation}\label{eq:round-low-ode-rigorous}
 w''-(n-2)\tanh t\,w'-\Lambda_\ell w
       +n\operatorname{sech}^2t\,w=0.
\end{equation}
For $\ell=0,1$ the regular Jacobi solutions are
\begin{equation}\label{eq:round-regular-low-modes-rigorous}
 z_0(t)=\tanh t,
 \qquad z_{1,m}(t,\theta)=\operatorname{sech}t\,Y_{1,m}(\theta).
\end{equation}
We suppress the fixed angular factor below.  Fix $t_0>2$.  A second scalar
solution on $[t_0,\infty)$ is
\begin{equation}\label{eq:round-singular-low-modes-rigorous}
 \widehat z_\ell(t)
 =z_\ell(t)\int_{t_0}^t
   \frac{\cosh^{n-2}s}{z_\ell(s)^2}\,ds.
\end{equation}
Thus, as $t\to\infty$,
\begin{align}
 z_0(t)&=1+O(e^{-2t}),&
 \widehat z_0(t)&=c_0e^{(n-2)t}(1+O(e^{-2t})),\label{eq:low-mode-asymptotics-zero-rigorous}\\
 z_1(t)&=2e^{-t}(1+O(e^{-2t})),&
 \widehat z_1(t)&=c_1e^{(n-1)t}(1+O(e^{-2t})).
 \label{eq:low-mode-asymptotics-one-rigorous}
\end{align}
The Liouville variable
\begin{equation}\label{eq:round-low-liouville-rigorous}
 Y(t)=\cosh(t)^{-(n-2)/2}w(t)
\end{equation}
therefore has regular and singular branches with exponents
\begin{equation}\label{eq:kappa-low-rigorous}
 \kappa_0=\frac{n-2}{2},\qquad \kappa_1=\frac n2:
 \qquad
 Y_\ell^{\rm reg}=O(e^{-\kappa_\ell t}),\quad
 Y_\ell^{\rm sing}=\widehat c_\ell e^{\kappa_\ell t}(1+O(e^{-2t})).
\end{equation}
We normalize the transition Liouville variable using the exact identity
\begin{equation}\label{eq:transition-integrating-factor-step15}
 A_u=\partial_s\log(q\mathcal V^4).
\end{equation}
We therefore use
\begin{equation}\label{eq:transition-liouville-normalization-step15}
 Y_{\rm tr}=2(q\mathcal V^4)^{1/2}w.
\end{equation}
In the exact round-bubble coordinate and the exact Schwarzschild coordinate,
respectively,
\begin{align}
 (q_B\mathcal V_B^4)^{1/2}
 &=\frac12\cosh(t)^{-\frac{n-2}{2}},
 \label{eq:bubble-integrating-factor-step15}\\
 (q_\Sigma\mathcal V_\Sigma^4)^{1/2}
 &=2(2\lambda)^{\frac{n-4}{2}}\cosh(at).
 \label{eq:Schwarzschild-integrating-factor-step15}
\end{align}
Thus \eqref{eq:transition-liouville-normalization-step15} agrees exactly
with \eqref{eq:round-low-liouville-rigorous} at the bubble endpoint.  At the
Schwarzschild endpoint it satisfies
\begin{equation}\label{eq:transition-Schwarzschild-gauge-factor-step15}
 Y_{\rm tr}=\Theta_\lambda Y_\Sigma,
 \qquad
 Y_\Sigma=\cosh(at)w,
 \qquad
 \Theta_\lambda=4(2\lambda)^{\frac{n-4}{2}}.
\end{equation}
The factor $\Theta_\lambda$ is constant along each exact Schwarzschild
block, so it does not alter its homogeneous Dirichlet-to-Neumann
coefficient, but it does rescale boundary values and inhomogeneous flux
remainders.

\begin{lemma}[Augmented round exterior estimate]
\label{aux:round-exterior-estimate-rigorous}
Let $\Psi_\ell$ and $\ell_\ell^*$ denote the localized projection source and
constraint in one fixed $\ell=0$ or $\ell=1$ sector.  Their supports are
contained in $\{t\le t_0\}$.  Let $T>t_0+2$, assume
$\operatorname{supp}f\subset\{t\le t_0\}$, and suppose that
\begin{equation}\label{eq:round-augmented-boundary-problem-rigorous}
 \mathscr L_{*,\ell}w=f+c\Psi_\ell\quad\hbox{on }\{t<T\},
 \qquad \ell_\ell^*(w)=g.
\end{equation}
Extend $f$ by zero to the complete round sphere and subtract the
complete-sphere projected solution associated with this extension and the
constraint datum $g$.
On $[t_0,T]$ the remainder has the form
\[
 \widetilde w=A z_\ell+B\widehat z_\ell.
\]
There is a constant independent of $T$ such that
\begin{equation}\label{eq:fixed-core-augmented-rigorous}
 |A|+|\widetilde c|
 +\|\widetilde w\|_{C^{2,\alpha}(\{t\le t_0\},g_*)}
 \le C|B|.
\end{equation}
The complete-sphere particular solution is bounded by
$C(\|f\|_{C^\alpha(g_*)}+|g|)$, and its Liouville value and derivative at
$t=T$ are bounded by
$Ce^{-\kappa_\ell T}(\|f\|_{C^\alpha}+|g|)$.
\end{lemma}

\begin{proof}
The complete-sphere projected problem is uniquely solvable by
Proposition~\ref{lem:step12}.  Its solution is smooth at the second stereographic
pole, hence has only the regular branch there; the final trace estimate
follows from \eqref{eq:kappa-low-rigorous}.

It remains to prove \eqref{eq:fixed-core-augmented-rigorous}.  If it failed,
there would be a sequence for which the left-hand side equals one and
$B\to0$.  Interior Schauder estimates give convergence on compact subsets
of $\mathbb S^n\setminus\{q_\infty\}$.  Since the singular coefficient tends
to zero, the limit has only the regular Jacobi branch and extends smoothly
across $q_\infty$.  The limiting pair solves the complete round augmented
homogeneous problem with zero localized constraint.  By the uniqueness part
of Proposition~\ref{lem:step12}, both the function and the projected coefficient
vanish, contradicting the normalization.  This proves the estimate.
\end{proof}

For zero interior data and prescribed Liouville boundary value $Y(T)=b$, let
$\mathcal D^{\rm bub}_{\ell,T}$ denote the outward bubble flux,
\[
 \partial_tY(T)=\mathcal D^{\rm bub}_{\ell,T}b.
\]

\begin{lemma}[Round-bubble Dirichlet-to-Neumann limit]
\label{aux:bubble-dtn-rigorous}
There are $T_*>0$, $\eta>0$, and $C<\infty$ such that, for $\ell=0,1$,
\begin{equation}\label{eq:bubble-dtn-limit-rigorous}
 \mathcal D^{\rm bub}_{\ell,T}
 =\kappa_\ell+O(e^{-\eta T}),
 \qquad T\ge T_*.
\end{equation}
In particular,
\begin{equation}\label{eq:bubble-dtn-positive-rigorous}
 \mathcal D^{\rm bub}_{\ell,T}\ge\frac12\kappa_\ell>0.
\end{equation}
For nonzero bubble data, the additional flux remainder is bounded by
$C(\|f\|_{C^\alpha(g_*)}+|g|)$.  The same estimates hold after one
normalized parameter derivative.
\end{lemma}

\begin{proof}
For the boundary-induced homogeneous solution, Lemma~\ref{aux:round-exterior-estimate-rigorous} gives $|A|\le C|B|$.
Equations \eqref{eq:low-mode-asymptotics-zero-rigorous}--
\eqref{eq:kappa-low-rigorous} then give
\[
 b=B Y_\ell^{\rm sing}(T)
   \bigl(1+O(e^{-2\kappa_\ell T})\bigr).
\]
Differentiating the same expansion yields
$\partial_tY(T)=\kappa_\ell b+O(e^{-\eta T})b$ for some $\eta>0$.
The inhomogeneous remainder is the trace of the complete-sphere particular
solution from the preceding lemma.  Smooth dependence of the projected
round problem on the normalized parameters gives the differentiated
statement.
\end{proof}

In a fixed low angular sector the ideal transition equation, after the
Liouville transformation of Lemma~\ref{lem:step7}, is
\begin{equation}\label{eq:transition-schrodinger-rigorous}
 -u''+Q_{\ell,L}(s)u=h(s),\qquad -L<s<L,
 \qquad Q_{\ell,L}\ge\kappa_*^2>0,
\end{equation}
where $\kappa_*$ is independent of $L$ and of $\ell=0,1$.  For the
homogeneous equation let $u_-$ and $u_+$ have endpoint values $(1,0)$ and
$(0,1)$, respectively, and define the outward Dirichlet-to-Neumann matrix
\begin{equation}\label{eq:transition-dtn-definition-rigorous}
 \mathbf D^{\rm tr}_{\ell,L}
 =\begin{pmatrix}
 -u_-'(-L)&-u_+'(-L)\\
 u_-'(L)&u_+'(L)
 \end{pmatrix}
 =\begin{pmatrix}d_{--}&d_{-+}\\d_{+-}&d_{++}\end{pmatrix}.
\end{equation}

\begin{lemma}[Transition Dirichlet-to-Neumann bounds]
\label{aux:transition-dtn-rigorous}
For all $L\ge L_*$,
\begin{equation}\label{eq:transition-dtn-diagonal-rigorous}
 d_{--},d_{++}\ge\kappa_*,
\end{equation}
and
\begin{equation}\label{eq:transition-dtn-offdiag-rigorous}
 d_{-+}=d_{+-},\qquad
 |d_{-+}|\le\frac{\kappa_*}{\sinh(2\kappa_*L)}
 \le Ce^{-2\kappa_*L}.
\end{equation}
For the inhomogeneous equation the endpoint flux remainder is bounded by
$C\|h\|_{Y_{\rm tr}(L)}$.  The same conclusions hold after one normalized
parameter derivative.
\end{lemma}

\begin{proof}
Integration by parts gives
\[
 d_{--}=\int_{-L}^L\bigl(|u_-'|^2+Q_{\ell,L}u_-^2\bigr)\,ds.
\]
Among functions with endpoint values $(1,0)$, the minimum of
\[
 \int_{-L}^L(|v'|^2+\kappa_*^2v^2)\,ds
 =\kappa_*\coth(2\kappa_*L).
\]
This proves the first diagonal estimate; the second is identical.  Symmetry
follows from the Wronskian identity.

Let
\[
 h_+(s)=\frac{\sinh(\kappa_*(s+L))}{\sinh(2\kappa_*L)}.
\]
Then
$(-\partial_s^2+Q_{\ell,L})h_+
 =(Q_{\ell,L}-\kappa_*^2)h_+\ge0$ and $h_+$ has the same endpoint values as
$u_+$.  The maximum principle and the boundary Hopf inequality at $-L$ give
$0\le u_+'(-L)\le h_+'(-L)$, which is
\eqref{eq:transition-dtn-offdiag-rigorous}.  The zero-boundary Green kernel
and the low-mode $L^1$ term in $Y_{\rm tr}(L)$ control the inhomogeneous
flux.  Differentiating the equation gives the parameter estimate.
\end{proof}

Let $b_{\ell,m}$ be the bubble-side transition value, which equals the
round-bubble trace by
\eqref{eq:transition-liouville-normalization-step15}, and let
$x_{\ell,m}^{\rm tr}$ be the neck-side value of $Y_{\rm tr}$.  Matching the
two outward conormal fluxes at the bubble interface gives
\begin{equation}\label{eq:bubble-transition-flux-rigorous}
 (\mathcal D^{\rm bub}_{\ell,T}+d_{--})b_{\ell,m}
   +d_{-+}x_{\ell,m}^{\rm tr}=r^{\rm bub}_{\ell,m},
\end{equation}
where the remainder is controlled by the bubble and transition data.

\begin{proposition}[Uniform low-mode matching matrix]
\label{aux:low-mode-matching-rigorous}
For $T\ge T_*$, $L\ge L_*$, and $\ell=0,1$,
\begin{equation}\label{eq:bubble-endpoint-map-rigorous}
 b_{\ell,m}=k^{\rm bub}_{\ell,L,T}x_{\ell,m}^{\rm tr}
   +r^{\rm bub}_{\ell,m},
 \qquad
 |k^{\rm bub}_{\ell,L,T}|\le Ce^{-2\kappa_*L}.
\end{equation}
On an exact Schwarzschild half-neck $[-T_\Sigma,0]$, put
$Y_\Sigma^{\rm tr}=\Theta_\lambda Y_\Sigma$.  Then
$Y_\Sigma^{\rm tr}(-T_\Sigma)=x_{\ell,m}^{\rm tr}$, and the parity condition
at the centered connector gives
\begin{equation}\label{eq:schwarzschild-dtn-rigorous}
 \partial_tY_\Sigma^{\rm tr}(-T_\Sigma)
 =d^\Sigma_{\ell,T_\Sigma}x_{\ell,m}^{\rm tr}
   +r^{\Sigma,{\rm tr}}_{\ell,m},
 \qquad
 d^\Sigma_{\ell,T_\Sigma}=
 \begin{cases}
 -\nu_\ell\tanh(\nu_\ell T_\Sigma),&\ell\ \mathrm{even},\\
 -\nu_\ell\coth(\nu_\ell T_\Sigma),&\ell\ \mathrm{odd},
 \end{cases}
 \le0.
\end{equation}
Here
\begin{equation}\label{eq:Schwarzschild-remainder-gauge-factor-step15}
 r^{\Sigma,{\rm tr}}_{\ell,m}
 =\Theta_\lambda r^\Sigma_{\ell,m},
\end{equation}
where $r^\Sigma_{\ell,m}$ is the remainder in the canonical Schwarzschild
gauge $Y_\Sigma$.
Consequently, the complete low-mode interface system is
\begin{equation}\label{eq:complete-low-matching-rigorous}
 \mathbf M_{\ell,L,T}
 \binom{b_{\ell,m}}{x_{\ell,m}^{\rm tr}}
 =\binom{r^{\rm bub}_{\ell,m}}
 {r^{\Sigma,{\rm tr}}_{\ell,m}-r^{\rm tr}_{+,\ell,m}},
 \qquad
 \mathbf M_{\ell,L,T}=
 \begin{pmatrix}
 1&-k^{\rm bub}_{\ell,L,T}\\
 d_{+-}&d_{++}-d^\Sigma_{\ell,T_\Sigma}
 \end{pmatrix},
\end{equation}
and
\begin{equation}\label{eq:complete-low-matching-inverse-rigorous}
 \|\mathbf M_{\ell,L,T}^{-1}\|\le C
\end{equation}
uniformly in $L,T,T_\Sigma$, the selected index, and the zero set of the weight.  The
inhomogeneous remainders are bounded by the bubble source norm, the local
transition source norm, and the low-mode local-supremum plus $L^1_t$ neck
source norm.  All statements persist under one normalized parameter
derivative.
\end{proposition}

\begin{proof}
Solving \eqref{eq:bubble-transition-flux-rigorous} gives
\[
 k^{\rm bub}_{\ell,L,T}
 =-(\mathcal D^{\rm bub}_{\ell,T}+d_{--})^{-1}d_{-+}.
\]
The denominator is at least $\kappa_\ell/2+\kappa_*$ by
Lemmas~\ref{aux:bubble-dtn-rigorous} and~\ref{aux:transition-dtn-rigorous}; the off-diagonal estimate therefore proves
\eqref{eq:bubble-endpoint-map-rigorous}.

Formula \eqref{eq:schwarzschild-dtn-rigorous} follows by solving
$-Y''+\nu_\ell^2Y=0$ on $[-T_\Sigma,0]$ with Neumann data at $0$ for even
$\ell$ and Dirichlet data at $0$ for odd $\ell$.  The inhomogeneous
remainder is the Green integral from Lemma~\ref{lem:step13}.  Flux matching
at the neck-side transition endpoint gives
\[
 (d_{++}-d^\Sigma_{\ell,T_\Sigma})x_{\ell,m}^{\rm tr}
 =-d_{+-}b_{\ell,m}+r^{\Sigma,{\rm tr}}_{\ell,m}
   -r^{\rm tr}_{+,\ell,m},
\]
which is \eqref{eq:complete-low-matching-rigorous}.  Since
$d_{++}-d^\Sigma\ge\kappa_*$ and
$|k^{\rm bub}d_{+-}|\le Ce^{-4\kappa_*L}$, the determinant is bounded below
by $\kappa_*/2$ after increasing $L_*$.  The entries are uniformly bounded,
so \eqref{eq:complete-low-matching-inverse-rigorous} follows.  Differentiation
of the displayed formulas proves the parameter statement.
\end{proof}

\begin{proposition}[Uniform inverse for the flat paired block]\label{prop:flat-paired-inverse-step15}
There exist $L_*\ge4$, $N_*\ge1$, $C_{\rm fl}>0$, and a fixed
$\vartheta_{\rm geo}>0$ such that the following holds.  For every
$L\ge L_*$ and $N_0\ge N_*$, choose $N_*(L,N_0)$ so large that all
bubble--transition scale separations hold and
\begin{equation}\label{eq:flat-ordered-transition-smallness-step15}
 C_L\delta_\lambda+e^L\lambda^{1-4/n}+N^{-1}
 \le \vartheta_{\rm geo}e^{-L}
\end{equation}
for every $N\ge N_*(L,N_0)$.  Here $C_L$ is enlarged once to dominate the
finite list of fixed-$L$ transition coefficient and matching errors.  Then
the flat augmented problem
\begin{equation}\label{eq:flat-augmented-problem-step15}
 \mathscr L_N^{\rm fl}w=F+\sum_{a=0}^nc_a\Psi_{a,N},
 \qquad \mathfrak L_b(w)=g_b,
\end{equation}
has a unique $J$-invariant solution and
\begin{equation}\label{eq:flat-augmented-estimate-step15}
 \|w\|_{X_\gamma^\sharp}+\sum_{a=0}^n|c_a|
 \le C_{\rm fl}\left(\|F\|_{Y_\gamma^\sharp}
              +\sum_{b=0}^n|g_b|\right).
\end{equation}
The constant $C_{\rm fl}$ is independent of $L,N_0,N$ and of the normalized
finite-dimensional parameters in this ordered regime.  The same estimate
holds after one normalized parameter derivative.
\end{proposition}

\begin{proof}
Consider first the ideal augmented block operator formed from the two
truncated round bubbles, the leading transitions, the exact Schwarzschild
half-necks, and the centered connector, with value and conormal matching at
the interfaces and the paired localized constraints.  The local
zero-boundary particular solutions are controlled by
Proposition~\ref{lem:step12}, Proposition~\ref{lem:step14}, and
Lemma~\ref{lem:step13}; the connector is a fixed elliptic block.

After subtracting these particular solutions, the remaining unknowns are the
homogeneous interface data.  In degrees zero and one, Proposition~\ref{aux:low-mode-matching-rigorous} gives a uniform bound for all interface
coefficients.  In degrees at least two, the transition gap from
Lemma~\ref{lem:step7}, the round spectral gap, and the Schwarzschild
exponential dichotomy give the same bound in the weighted norm; Proposition~\ref{prop:reset-control-step15} allows additional zeros of the weight.  Thus
the ideal system satisfies
\begin{equation}\label{eq:ideal-augmented-bound-rigorous}
 \|w\|_{X_\gamma^\sharp}+|c|
 \le C_0\bigl(\|F\|_{Y_\gamma^\sharp}+|g|\bigr),
\end{equation}
with $C_0$ independent of all block lengths.  The homogeneous system is
trivial, and the Fredholm alternative gives existence and uniqueness.

The finite-$N$ flat block is written on the same augmented spaces.  Its
transition, bubble-interface, and connector errors are bounded by
\[
 C\bigl(C_L\delta_\lambda+e^L\lambda^{1-4/n}+N^{-1}\bigr).
\]
The same bound holds for sources, constraints, conormal traces, and one
normalized parameter derivative by
Propositions~\ref{lem:step11}--\ref{lem:step14}.
Hence \eqref{eq:flat-ordered-transition-smallness-step15} makes the
finite-$N$ operator a small perturbation of the ideal one.  The Neumann
series gives \eqref{eq:flat-augmented-estimate-step15}; differentiation of
the augmented system gives the parameter estimate.
\end{proof}

\begin{lemma}[Global uniform projected inverse]\label{lem:step15}
There is a constant $C_{\rm inv}>0$, depending only on the fixed analytic
data and not on $L,N_0,N$, with the following property.  First choose any
$L\ge L_*$; then choose $N_0$ and a lower bound for $N$ so large that
\eqref{eq:flat-ordered-transition-smallness-step15} holds and
\begin{equation}\label{eq:actual-inverse-ordered-smallness-step15}
 C_{\rm fl}\eta^{\rm op}_{L,N_0,N}\le\frac12.
\end{equation}
For every selected index above this lower bound, problem
\eqref{eq:global-projected-problem-step15} has a unique $J$-invariant solution
and
\begin{equation}\label{eq:global-inverse-estimate-step15}
 \|w\|_{X_\gamma^\sharp}+\sum_{a=0}^n|c_a|
 \le C_{\rm inv}\|F\|_{Y_\gamma^\sharp}.
\end{equation}
If the data have one normalized parameter derivative, then
\begin{align}\label{eq:global-inverse-parameter-step15}
 \|\mathfrak Dw\|_{X_\gamma^\sharp}
 +\sum_{a=0}^n|\mathfrak Dc_a|
 \le C_{\rm inv}\bigl(
  \|F\|_{Y_\gamma^\sharp}
  +\|\mathfrak DF\|_{Y_\gamma^\sharp}\bigr).
\end{align}
\end{lemma}

\begin{proof}
Let $\mathcal A_N^{\rm fl}$ and $\mathcal A_N$ be the augmented maps for the
flat and actual operators with the same paired sources and constraints.
Proposition~\ref{prop:flat-paired-inverse-step15} gives
\begin{equation}\label{eq:flat-augmented-inverse-bound-step15}
 \|(\mathcal A_N^{\rm fl})^{-1}\|\le C_{\rm fl}
\end{equation}
throughout the ordered family, while Proposition~\ref{prop:summable-operator-perturbation-step15} gives
\begin{equation}\label{eq:augmented-map-difference-step15}
 \|\mathcal A_N-\mathcal A_N^{\rm fl}\|
 \le\eta^{\rm op}_{L,N_0,N}.
\end{equation}
For a fixed $L$, equations \eqref{eq:delta-star-limit} and
\eqref{eq:eta-small-step15} allow us to choose $N_0$ and then the index
threshold so that \eqref{eq:actual-inverse-ordered-smallness-step15} holds.
The factorization
\[
 \mathcal A_N=\mathcal A_N^{\rm fl}
 \left[I+(\mathcal A_N^{\rm fl})^{-1}
 (\mathcal A_N-\mathcal A_N^{\rm fl})\right]
\]
and the Neumann series yield
\[
 \|\mathcal A_N^{-1}\|\le2C_{\rm fl}.
\]
Thus \eqref{eq:global-inverse-estimate-step15} holds with a constant
independent of $L$, $N_0$, and $N$.

For one normalized parameter derivative, differentiate the equation and the
constraints in the natural block coordinates.  The differentiated unknown
solves the same augmented problem with source
\[
 \mathfrak DF-(\mathfrak D\mathscr L_N)w
 +\sum_ac_a\mathfrak D\Psi_{a,N}
\]
and the differentiated constraint data.  Lemmas~\ref{lem:step8},~\ref{lem:step11},~\ref{lem:step13}, and the uniform local-coefficient and mode-split transfer estimates of
Proposition~\ref{lem:step14}, together with
\eqref{eq:summable-operator-parameter-step15}, give a
parameter-derivative coefficient bound with a constant independent of
$L,N_0,N$ once \eqref{eq:actual-inverse-ordered-smallness-step15} is imposed.
Apply the just-proved inverse first to $(w,c)$ and then to the differentiated
problem.  Enlarging $2C_{\rm fl}$ by one fixed factor proves
\eqref{eq:global-inverse-parameter-step15}.  This completes the proof.
\end{proof}

\section{Nonlinear projected problem}

The mode-split space $X_\gamma^\sharp$ is the correct space for the coarse
linear inverse, but it is not by itself a nonlinear algebra.  Indeed, products
of first spherical harmonics contain degree-two components, while products of
high modes can contribute to the low-mode $L^1$ source norm.  The projected
inverse has additional transfer and integrability properties, which we now
state using norms adapted to the nonlinear problem.

Choose once and for all a number $\beta$ such that
\begin{equation}\label{eq:beta-choice-step16}
 \frac\gamma2<\beta<\min\{1,2\gamma\}.
\end{equation}
Such a choice is possible because $\gamma>0$ and
$\gamma<\gamma_*=\nu_2-a<2$; the last inequality follows from
$\nu_2^2-a^2=n(n-2)/(n-1)<n=(a+2)^2-a^2$.
For a unit cylinder $Q_t=[t-1,t+1]\times\mathbb S^{n-1}$, put
\[
 n_F(t):=\mathfrak Y_{Q_t}(F).
\]
Let $\Pi_0$ and $\Pi_1$ denote the degree-zero and degree-one spherical
projections.  On either Schwarzschild end write
\[
 b_{\mathfrak e}(t):=\Pi_0w(t,\cdot),\qquad
 w_{\mathfrak e}^{\perp}:=(I-\Pi_0)w,
 \qquad \mathfrak e\in\{-,+\}.
\]
For the exact Schwarzschild factor,
\begin{equation}\label{eq:neutral-q-step16}
 q_\Sigma(t)=\frac1{4\cosh^2(at)},\qquad
 P_{\Sigma,0}^{(0)}b
 :=b_{tt}+2a\tanh(at)b_t.
\end{equation}
The second operator in \eqref{eq:neutral-q-step16} is the radial part of
the limiting relative operator.  In particular,
$P_{\Sigma,0}^{(0)}1=0$; this constant transfer branch is the reason that a
full-neck $L_t^2$ norm of $\Pi_0w$ cannot be uniform.

For a radial function $b$ on $\mathcal S_{\mathfrak e}$ define its neutral
transfer density by
\begin{equation}\label{eq:neutral-density-step16}
 \mathfrak d_0[b](t)
 :=\bigl\|q_\Sigma^{-1/2}b_t\bigr\|_{C^\alpha([t-1,t+1])}
  +\bigl\|q_\Sigma^{-1}P_{\Sigma,0}^{(0)}b
    \bigr\|_{C^\alpha([t-1,t+1])},
\end{equation}
and set
\begin{align}\label{eq:neutral-transfer-norm-step16}
 \|b\|_{\mathfrak T_{\mathfrak e}}
 :={}&\|b\|_{L^\infty(\mathcal S_{\mathfrak e})}
 +\int_{\mathcal S_{\mathfrak e}}\mathfrak d_0[b](t)\,dt\nonumber\\
 &+\int_{\mathcal T_{\mathfrak e}}m_b(t)\,dt.
\end{align}
The final integral is over a transition of length $2L$.  It is an ordinary
local Newton-strong term after $L$ is fixed, but its global bound may contain
a factor $1+L$.

Define the corrected refined norms
\begin{align}\label{eq:refined-X-step16}
 \|w\|_{\widehat X_{\gamma,\beta}}
 :={}&\|w\|_{X_\gamma^\sharp}
 +\sum_{\mathfrak e=\pm}
   \left(\|b_{\mathfrak e}\|_{\mathfrak T_{\mathfrak e}}
   +\int_{\mathcal T_{\mathfrak e}\cup\mathcal S_{\mathfrak e}}
       m_{w_{\mathfrak e}^{\perp}}(t)\,dt\right)\nonumber\\
 &+\sup_{Q\subset\mathcal T\cup\mathcal S}
   e^{\beta d_N(Q)}\mathfrak X_Q(\Pi_1w),
\end{align}
\begin{align}\label{eq:refined-Y-step16}
 \|F\|_{\widehat Y_{\gamma,\beta}}
 :={}&\|F\|_{Y_\gamma^\sharp}
 +\sup_{Q\subset\mathcal T\cup\mathcal S}
   e^{\beta d_N(Q)}\mathfrak Y_Q(\Pi_1F)\nonumber\\
 &+\sum_{\mathfrak e=\pm}
   \int_{\mathcal T_{\mathfrak e}\cup\mathcal S_{\mathfrak e}}
   n_F(t)\,dt.
\end{align}
The bubble and connector pieces are already included in the coarse norms.
Thus the radial average is kept as a neutral transfer variable, while
only the nonradial remainder is required to be integrable along the full
neck.

\subsection{Radial neutral-transfer estimate}
Suppose on a metric-free Schwarzschild component $I=(\tau_-,\tau_+)$ that
\begin{equation}\label{eq:radial-transfer-equation-step16}
 P_{\Sigma,0}^{(0)}b=q_\Sigma f.
\end{equation}
Since $q_\Sigma=(4\cosh^2(at))^{-1}$,
\begin{equation}\label{eq:radial-flux-identity-step16}
 \bigl(\cosh^2(at)b_t\bigr)_t=\frac14f.
\end{equation}
Consequently
\begin{align}\label{eq:radial-transfer-estimate-step16}
 \|b\|_{\mathfrak T(I)}
 \le C\biggl(&|b(\tau_-)|+|b(\tau_+)|
 +|\mathfrak f(\tau_-)|+|\mathfrak f(\tau_+)|\nonumber\\
 &+\int_I|f(t)|\,dt\biggr),
 \qquad \mathfrak f=\cosh^2(at)b_t,
\end{align}
where $C$ is independent of $|I|$.  Indeed,
\eqref{eq:radial-flux-identity-step16} controls the flux, while
\[
 q_\Sigma^{-1/2}|b_t|
 =2|\mathfrak f|\operatorname{sech}(at),
 \qquad
 \int_{\mathbb R}\operatorname{sech}(at)\,dt<\infty.
\]
Integration of $b_t=\mathfrak f\operatorname{sech}^2(at)$ gives the
$L^\infty$ estimate, and local Schauder estimates give the stated
$C^\alpha$ version.  The finite-$\lambda$ metric-free neck operator has only the
target zeroth-order perturbation, whose normalized coefficient satisfies
\begin{equation}\label{eq:neutral-target-envelope-step16}
 \mathfrak v_\lambda(t)\le C\lambda^4(r^{-n}+r^n),
 \qquad
 \int\mathfrak v_\lambda(t)\,dt
 \le C(e^{-nL}+\lambda^4).
\end{equation}
This zeroth-order perturbation is absorbed once $L$ is fixed.  The same
estimate holds after one normalized parameter derivative.

\begin{lemma}[H\"older-valued nonradial transfer on a Schwarzschild cylinder]
\label{aux:holder-nonradial-transfer-step16}
Let $I=(\tau_-,\tau_+)$ be contained in one Schwarzschild end, and suppose
\begin{equation}\label{eq:nonradial-transfer-equation-step16}
 \mathscr P_{\Sigma,0}z=q_\Sigma f,\qquad \Pi_0z=0.
\end{equation}
Let $d:I\to[0,\infty)$ be any regularized distance with $|d'|\le1$ which
vanishes on the endpoint collars; additional zeros are allowed.
Choose
\[
 0<\sigma_2<\gamma_*=\nu_2-a,\qquad
 \beta<\sigma_1<1.
\]
If $Q_t=[t-1,t+1]\times\mathbb S^{n-1}$ and
\[
 N_j(t):=\|\Pi_jf\|_{C^\alpha(Q_t)},\qquad
 N_{\ge2}(t):=\|\Pi_{\ge2}f\|_{C^\alpha(Q_t)},
\]
then, after enlarging endpoint cylinders by a fixed amount,
\begin{align}
 m_{\Pi_{\ge2}z}(t)
 \le{}& C\int_Ie^{-\sigma_2|t-s|}N_{\ge2}(s)\,ds
 {}+C e^{-\sigma_2(t-\tau_-)}\mathcal B_-(z)
 +C e^{-\sigma_2(\tau_+-t)}\mathcal B_+(z),
 \label{eq:holder-high-offdiagonal-step16}\\
 m_{\Pi_1z}(t)
 \le{}& C\int_Ie^{-\sigma_1|t-s|}N_1(s)\,ds
 {}+C e^{-\sigma_1(t-\tau_-)}\mathcal B_-(z)
 +C e^{-\sigma_1(\tau_+-t)}\mathcal B_+(z),
 \label{eq:holder-first-offdiagonal-step16}
\end{align}
where $\mathcal B_\pm(z)$ are the Newton-strong norms on the two endpoint
unit collars.  Consequently,
\begin{align}
 \int_I m_z(t)\,dt
 &\le C\left(\int_I\|f(t,\cdot)\|_{C^\alpha_\theta}\,dt
              +\mathcal B_-(z)+\mathcal B_+(z)\right),
 \label{eq:holder-nonradial-L1-step16}\\
 \sup_{Q_t\subset I}e^{\beta d(t)}m_{\Pi_1z}(t)
 &\le C\left(
 \sup_{Q_t\subset I}e^{\beta d(t)}N_1(t)
 +\int_I N_1(t)\,dt+\mathcal B_-(z)+\mathcal B_+(z)\right).
 \label{eq:holder-first-weighted-step16}
\end{align}
All constants are independent of $|I|$ and of the selected index.  The same
estimates hold after one normalized parameter derivative.
\end{lemma}

\begin{proof}
Put
\[
 \mathcal A_\theta
 :=\left(a^2-b_\Sigma\Delta_{\mathbb S^{n-1}}\right)^{1/2}.
\]
Under $Y=\cosh(at)z$, equation
\eqref{eq:nonradial-transfer-equation-step16} becomes
\begin{equation}\label{eq:operator-valued-Schrodinger-step16}
 (-\partial_t^2+\mathcal A_\theta^2)Y
 =-\frac14\operatorname{sech}(at)f.
\end{equation}
The Dirichlet Green operator on $I$ is obtained from the scalar formula in
Lemma~\ref{lem:step13} by the spectral functional calculus of
$\mathcal A_\theta$.  On the compact sphere its Poisson semigroup satisfies
\begin{align}\label{eq:spherical-Poisson-holder-step16}
 \|e^{-\rho\mathcal A_\theta}\Pi_{\ge2}\|_{C^\alpha\to C^\alpha}
 &\le C e^{-\nu_2\rho},\\
 \|e^{-\rho\mathcal A_\theta}\Pi_1\|_{C^\alpha\to C^\alpha}
 &\le C e^{-\nu_1\rho}.\nonumber
\end{align}
These are operator estimates on the full H\"older spaces, not estimates of
individual harmonic coefficients.  They follow, for example, by applying
the analytic semigroup estimate for the positive elliptic operator
$\mathcal A_\theta$ on $C^\alpha(\mathbb S^{n-1})$ and using the spectral
gaps $\nu_2$ and $\nu_1=a+1$ on the indicated invariant subspaces.

Decompose the right-hand side of
\eqref{eq:operator-valued-Schrodinger-step16} into unit $t$-slabs.  The
Green formula on a finite interval, interpreted by functional calculus,
together with
\eqref{eq:spherical-Poisson-holder-step16} show that the response to a slab
centered at $s$ has $C^\alpha$ size at most
\[
 C e^{-(\nu_2-a)|t-s|}N_{\ge2}(s)
 \quad\hbox{or}\quad
 C e^{-(\nu_1-a)|t-s|}N_1(s)
\]
on a slab centered at $t$ after returning from $Y$ to $z$.  The loss of
$a$ is exactly the bound for the ratio of the two hyperbolic-cosine
factors on one end.  The two homogeneous Poisson operators satisfy the
same estimates with $|t-s|$ replaced by the distance to the corresponding
endpoint.

For neighboring slabs, interior and boundary Schauder estimates for
\eqref{eq:nonradial-transfer-equation-step16} recover the complete
Newton-strong quantity $m_z(t)$ from the just-established $C^\alpha$ bound
and the local source norm.  For separated slabs the same Schauder estimate,
applied on a fixed enlargement, preserves the exponential factor.
Decreasing the two exponents to $\sigma_2$ and $\sigma_1$ gives
\eqref{eq:holder-high-offdiagonal-step16} and
\eqref{eq:holder-first-offdiagonal-step16}.

Integration in $t$ and Fubini prove
\eqref{eq:holder-nonradial-L1-step16}.  Since $d$ is one-Lipschitz,
\[
 e^{\beta d(t)-\beta d(s)}\le e^{\beta|t-s|},
\]
and $\beta<\sigma_1$; multiplying
\eqref{eq:holder-first-offdiagonal-step16} by $e^{\beta d_I(t)}$ proves
\eqref{eq:holder-first-weighted-step16}.  Differentiating
\eqref{eq:operator-valued-Schrodinger-step16} in fixed natural coordinates
produces the same operator and a source bounded by the differentiated and
undifferentiated data.  This proves the parameter statement.
\end{proof}

\begin{lemma}[Liouville matching and the physical radial flux]
\label{aux:liouville-physical-flux-step16}
Orient an exact degree-zero Schwarzschild half-neck as $I=[-T,0]$, with
the centered end at $0$, and suppose
\[
 P_{\Sigma,0}^{(0)}b=q_\Sigma f.
\]
Set
\[
 Y_\Sigma=\cosh(at)b,\qquad
 x_\Sigma=Y_\Sigma(-T),\qquad
 g_0=\partial_tY_\Sigma(0),\qquad
 \mathfrak f=\cosh^2(at)b_t.
\]
Then
\begin{align}
 b(-T)&=\operatorname{sech}(aT)x_\Sigma,
 \label{eq:liouville-value-conversion-step16}\\
 \mathfrak f(-T)
 &=\cosh(aT)\bigl[\partial_tY_\Sigma(-T)
                    +a\tanh(aT)x_\Sigma\bigr],
 \label{eq:liouville-flux-conversion-step16}\\
 \mathfrak f(t)
 &=g_0-\frac14\int_t^0f(s)\,ds.
 \label{eq:physical-flux-propagation-step16}
\end{align}
In particular, if the Schwarzschild Dirichlet-to-Neumann relation is
written as
\begin{equation}\label{eq:radial-dtn-remainder-step16}
 \partial_tY_\Sigma(-T)
 =-a\tanh(aT)x_\Sigma+r^\Sigma,
\end{equation}
then its inhomogeneous remainder satisfies the length-uniform scaled bound
\begin{equation}\label{eq:scaled-radial-dtn-remainder-step16}
 \cosh(aT)|r^\Sigma|
 \le |g_0|+\frac14\int_{-T}^0|f(s)|\,ds.
\end{equation}
Thus the exponentially large homogeneous Liouville trace cancels exactly
from the physical flux.  The same identities and estimates hold on the
oppositely oriented end and after one normalized parameter derivative.
\end{lemma}

\begin{proof}
A direct differentiation gives
\begin{equation}\label{eq:radial-Liouville-conjugation-step16}
 P_{\Sigma,0}^{(0)}b
 =\operatorname{sech}(at)
  (\partial_t^2Y_\Sigma-a^2Y_\Sigma)
\end{equation}
and
\begin{equation}\label{eq:radial-flux-Liouville-formula-step16}
 \mathfrak f
 =\cosh(at)\bigl(\partial_tY_\Sigma
                  -a\tanh(at)Y_\Sigma\bigr).
\end{equation}
At $t=-T$ these are
\eqref{eq:liouville-value-conversion-step16}--
\eqref{eq:liouville-flux-conversion-step16}.  At the centered end,
\eqref{eq:radial-flux-Liouville-formula-step16} gives
$\mathfrak f(0)=\partial_tY_\Sigma(0)=g_0$.  Integrating
\eqref{eq:radial-flux-identity-step16} from $t$ to $0$ proves
\eqref{eq:physical-flux-propagation-step16}.  Substitution of
\eqref{eq:radial-dtn-remainder-step16} into
\eqref{eq:liouville-flux-conversion-step16} gives
$\mathfrak f(-T)=\cosh(aT)r^\Sigma$, and
\eqref{eq:scaled-radial-dtn-remainder-step16} follows.

For the centered even radial sector one has $g_0=0$.  With inhomogeneous
data on the fixed connector, its elliptic estimate bounds $g_0$ by the
connector source and trace norms.  If $x^{\rm tr}$ and
$r^{\Sigma,{\rm tr}}$ denote the neck-side variables in the matching system,
then \eqref{eq:transition-Schwarzschild-gauge-factor-step15} and
\eqref{eq:Schwarzschild-remainder-gauge-factor-step15} give
\begin{equation}\label{eq:matching-physical-gauge-conversion-step16}
 x^{\rm tr}=\Theta_\lambda x_\Sigma,
 \qquad
 r^{\Sigma,{\rm tr}}=\Theta_\lambda r^\Sigma.
\end{equation}
Thus the bound for $x^{\rm tr}$ given by the matching matrix is not
inserted directly into \eqref{eq:liouville-value-conversion-step16}.  In the
application to the flat paired inverse, the physical endpoint values of
$b$ are instead controlled by the $C^0$ part of
\eqref{eq:flat-augmented-estimate-step15}.  On the whole half-neck,
\eqref{eq:physical-flux-propagation-step16} gives the length-independent
estimate
\begin{equation}\label{eq:physical-flux-uniform-step16}
 \sup_{-T\le t\le0}|\mathfrak f(t)|
 \le |g_0|+\frac14\int_{-T}^0|f(s)|\,ds.
\end{equation}
At $t=-T$, its canonical Schwarzschild form is
\eqref{eq:scaled-radial-dtn-remainder-step16}.  Combining these bounds with
\eqref{eq:radial-transfer-estimate-step16} gives a constant independent of
$T$.  Differentiation in the natural half-neck coordinate proves the
parameter statement.
\end{proof}

\begin{proposition}[Refined neutral-transfer estimates]
\label{prop:refined-transfer-step16}
The residual of Lemma~\ref{lem:step9} obeys
\begin{equation}\label{eq:refined-residual-step16}
 \|\mathcal N_g(\bar v_{N,\lambda,\xi})
   \|_{\widehat Y_{\gamma,\beta}}
 \le \eta_{L,N_0}+C_L\lambda^{2-8/n}
      +C\epsilon_N^{\kappa_f},
\end{equation}
with one normalized parameter derivative.  Moreover, if
$(w,c_0,\ldots,c_n)$ is the solution of the projected problem
\eqref{eq:global-projected-problem-step15}, then there is a fixed constant
$C_*>0$ such that, for every admissible transition
length $L$, the refined inverse satisfies
\begin{equation}\label{eq:refined-inverse-step16}
 \|w\|_{\widehat X_{\gamma,\beta}}+
 \sum_{a=0}^n|c_a|
 \le C_{\rm ref}(L)\|F\|_{\widehat Y_{\gamma,\beta}},
 \qquad
 C_{\rm ref}(L)\le C_*(1+L),
\end{equation}
again with one normalized parameter derivative.  Once $L$ is fixed,
$C_{\rm ref}(L)$ is independent of $N_0,N$ and of the normalized
finite-dimensional parameters in the ordered regime.
\end{proposition}

\begin{proof}
For the residual, all metric, transition, and connector sources lie where
$d_N=0$.  Their full cylindrical $L^1$ bounds follow from the regionwise
estimates in Lemmas~\ref{lem:step8}--\ref{lem:step9}: for remaining perturbations one
sums $C\mu_MM^P$, and the principal source gives
$C\epsilon_N^{\kappa_f}$.  On an exact Schwarzschild piece the source is
radial and bounded by $C\lambda^4(r^{-n}+r^n)$, whose full cylindrical
$L^1$ norm is controlled by
\eqref{eq:neutral-target-envelope-step16}.  The weighted first-harmonic
term vanishes there.  This proves \eqref{eq:refined-residual-step16}.

We first set the metric-free neck target coefficient to zero in the flat paired
inverse.  Apply Lemma~\ref{aux:holder-nonradial-transfer-step16} on every metric-free Schwarzschild
component.  Its operator-valued estimate controls the full harmonic sum and gives
\begin{align}\label{eq:flat-nonradial-transfer-bound-step16}
 &\sum_{\mathfrak e=\pm}
 \int_{\mathcal S_{\mathfrak e}}
 m_{w_{\mathfrak e}^{\perp}}(t)\,dt
 +\sup_{Q\subset\mathcal S}
 e^{\beta d_N(Q)}\mathfrak X_Q(\Pi_1w)\nonumber\\
 &\qquad\le C\left(
 \sum_{\mathfrak e=\pm}
 \int_{\mathcal S_{\mathfrak e}}n_F(t)\,dt
 +\mathcal B_{\rm int}(w)\right),
\end{align}
where $\mathcal B_{\rm int}(w)$ is the sum of the Newton-strong endpoint
collar norms.  The matching estimate
\eqref{eq:complete-low-matching-inverse-rigorous}, the high-mode
exponential dichotomy in the proof of
Proposition~\ref{prop:flat-paired-inverse-step15}, and the fixed-overlap
transition and connector estimates bound
\[
 \mathcal B_{\rm int}(w)
 \le C\|F\|_{\widehat Y_{\gamma,\beta}}.
\]
Thus \eqref{eq:flat-nonradial-transfer-bound-step16} gives both the full
nonradial $L_t^1$ term and the weighted degree-one term with a constant
independent of all neck lengths.

It remains to treat degree zero.  The endpoint values and fluxes are
most naturally expressed in different gauges.  The matching matrix
\eqref{eq:complete-low-matching-rigorous} controls the transition-normalized
variable $x^{\rm tr}$ and hence the interface response in its native gauge.
The physical endpoint values of $b$ are controlled by the coarse inverse
\eqref{eq:flat-augmented-estimate-step15}, while the fixed connector estimate
controls $g_0$.  Lemma~\ref{aux:liouville-physical-flux-step16}, specifically
\eqref{eq:physical-flux-uniform-step16}, controls the physical flux by the
connector data and the full $L_t^1$ source norm.  Apply
\eqref{eq:radial-transfer-estimate-step16} on each metric-free component.  The
transition contribution is controlled by the fixed-$L$ local estimate of
Proposition~\ref{lem:step14}, and the centered connector is a fixed elliptic
block with the $J$-even radial condition.  Propagating
\eqref{eq:physical-flux-propagation-step16} across the components produces
only the full $L_t^1$ source norm; the neighborhoods of the zero set of the weight and of the metric supports
have fixed overlap.  Hence
\begin{equation}\label{eq:flat-neutral-transfer-bound-step16}
 \sum_{\mathfrak e=\pm}
 \left(\|b_{\mathfrak e}\|_{\mathfrak T_{\mathfrak e}}
 +\int m_{w_{\mathfrak e}^{\perp}}(t)\,dt\right)
 +\sup_Qe^{\beta d_N(Q)}\mathfrak X_Q(\Pi_1w)
 \le C_0(1+L)\|F\|_{\widehat Y_{\gamma,\beta}}.
\end{equation}
Together with the coarse inverse, this proves
the analogue of \eqref{eq:refined-inverse-step16} for the flat augmented
operator with the
metric-free neck target term removed.  Denote the resulting uniform constant by
$C_{\rm tr}^0(L)$, where
\begin{equation}\label{eq:flat-clean-transfer-growth-step16}
 C_{\rm tr}^0(L)\le C_0(1+L).
\end{equation}
The factor $1+L$ comes only from integrating the transition terms already controlled in the coarse local norm.

We now restore the finite-$\lambda$ target term on the metric-free Schwarzschild pieces.  Multiplication by its normalized coefficient maps
$\widehat X_{\gamma,\beta}$ to $\widehat Y_{\gamma,\beta}$ with norm at
most
\begin{equation}\label{eq:target-transfer-operator-size-step16}
 \eta_\Sigma^{\rm tr}
 :=C_\Sigma(e^{-nL}+\lambda^4)
\end{equation}
by \eqref{eq:neutral-target-envelope-step16}; the same bound holds after
one normalized parameter derivative.  Enlarge $L_*$ once so that
\begin{equation}\label{eq:uniform-L-target-transfer-step16}
 C_0(1+L)C_\Sigma e^{-nL}\le\frac14
 \qquad(L\ge L_*).
\end{equation}
For each fixed $L\ge L_*$, increase the index threshold so that the
$\lambda^4$ part is also at most $1/4$.  Then
$C_{\rm tr}^0(L)\eta_\Sigma^{\rm tr}\le\frac12$, a Neumann series gives the
finite-flat refined inverse with
\begin{equation}\label{eq:finite-flat-transfer-constant-step16}
 C_{\rm tr}^{\rm fl}(L):=2C_{\rm tr}^0(L)
 \le2C_0(1+L).
\end{equation}
The constant $C_\Sigma$ depends only on the model, while the displayed
transfer constants grow at most linearly in $L$ and enter the subsequent
parameter selection.

For the actual operator, treat
$K_Nw=(\mathscr L_N-\mathscr L_N^{\rm fl})w$ as an additional source.
For $w^\perp$ the regionwise proof of Proposition~\ref{prop:summable-operator-perturbation-step15} also controls the full
cylindrical $L^1$ term.  For the radial transfer component, Lemma~\ref{lem:step8} gives normalized size $C\epsilon_N^{\kappa_f}$ on the
column corresponding to the principal pair, whose cylindrical length is
$O(1+|\log\epsilon_N|)$.  The remaining columns give the summable bound
$C\delta_{\!*}(N_0)$.  Derivatives of $b_{\mathfrak e}$ are controlled by
\eqref{eq:neutral-transfer-norm-step16}.  Thus, after enlarging the
fixed-$L$ constant once,
\begin{equation}\label{eq:refined-operator-perturbation-step16}
 \|K_Nw\|_{\widehat Y_{\gamma,\beta}}
 \le \eta^{\rm tr}_{L,N_0,N}
       \|w\|_{\widehat X_{\gamma,\beta}},
\end{equation}
where
\begin{equation}\label{eq:neutral-perturbation-size-step16}
 \eta^{\rm tr}_{L,N_0,N}
 \le C_{\rm op}(L)\left[
 \delta_{\!*}(N_0)
 +(1+|\log\epsilon_N|)\epsilon_N^{\kappa_f}\right].
\end{equation}
After first choosing $N_0$ and then increasing the index threshold,
$C_{\rm tr}^{\rm fl}(L)\eta^{\rm tr}_{L,N_0,N}\le\frac12$.  The Neumann series
therefore transfers \eqref{eq:flat-neutral-transfer-bound-step16} to the
actual operator, with
\begin{equation}\label{eq:actual-refined-transfer-growth-step16}
 C_{\rm ref}(L):=2C_{\rm tr}^{\rm fl}(L)
 \le4C_0(1+L).
\end{equation}
Thus the proposition holds after enlarging the fixed constant $C_*$.
Differentiating the flux identity, the matching
matrix, and the two coefficient perturbations in the natural block
coordinates gives the same estimate for one normalized parameter
derivative.
\end{proof}

\subsection{Radial identity for the nonlinear estimate}
The long-neck radial part cannot be treated by an
$L_t^2\times L_t^2$ argument.  Instead, for a radial relative correction
$b$ one has the exact identity
\begin{equation}\label{eq:radial-null-identity-step16}
 \widetilde p_t+2a\widetilde q
 =-\frac12P_{\Sigma,0}^{(0)}b-\frac a2b_t^2,
 \qquad
 \widetilde p=p-\frac12b_t,\quad
 \widetilde q=\widetilde p(1-\widetilde p).
\end{equation}
Thus the normalized derivative-quadratic term is
$q_\Sigma^{-1}b_t^2=(q_\Sigma^{-1/2}b_t)^2$, which is integrable in the
transfer norm.  This is the mechanism used in Lemma~\ref{lem:step16}.

For a relative correction $w$, define the nonlinear remainder
\begin{equation}\label{eq:nonlinear-remainder-step16}
 Q_N(w):=\mathcal N_g\bigl(
   \bar v_{N,\lambda,\xi}e^{aw}\bigr)
 -\mathcal N_g\bigl(\bar v_{N,\lambda,\xi}\bigr)
 -\mathscr L_Nw.
\end{equation}

\begin{lemma}[Quadratic estimate in Newton-strong spaces]\label{lem:step16}
There are constants $r_{\rm nl}>0$ and $C_Q>0$, independent of
$L,N_0,N$ in the ordered regime, such that, whenever
$\|w_i\|_{\widehat X_{\gamma,\beta}}\le r_{\rm nl}$,
\begin{equation}\label{eq:quadratic-Lipschitz-step16}
 \|Q_N(w_1)-Q_N(w_2)\|_{\widehat Y_{\gamma,\beta}}
 \le C_Q\bigl(\|w_1\|_{\widehat X_{\gamma,\beta}}
            +\|w_2\|_{\widehat X_{\gamma,\beta}}\bigr)
       \|w_1-w_2\|_{\widehat X_{\gamma,\beta}}.
\end{equation}
In particular,
\begin{equation}\label{eq:quadratic-self-step16}
 \|Q_N(w)\|_{\widehat Y_{\gamma,\beta}}
 \le C_Q\|w\|_{\widehat X_{\gamma,\beta}}^2.
\end{equation}
For one normalized parameter derivative $\mathfrak D$, with the weights
frozen in the natural coordinates,
\begin{align}\label{eq:quadratic-parameter-step16}
 &\|\mathfrak D[Q_N(w_1)-Q_N(w_2)]
   \|_{\widehat Y_{\gamma,\beta}}\nonumber\\
 &\quad\le C_Q\bigl(\|w_1\|_{\widehat X}+\|w_2\|_{\widehat X}\bigr)
       \|\mathfrak D(w_1-w_2)\|_{\widehat X}\nonumber\\
 &\qquad+C_Q\bigl(\|\mathfrak Dw_1\|_{\widehat X}
                 +\|\mathfrak Dw_2\|_{\widehat X}
                 +\|w_1\|_{\widehat X}+\|w_2\|_{\widehat X}\bigr)
       \|w_1-w_2\|_{\widehat X},
\end{align}
where $\widehat X=\widehat X_{\gamma,\beta}$.
\end{lemma}

\begin{proof}
Put $\bar g=\bar v_{N,\lambda,\xi}^{8/(n-4)}g$ and
$\kappa=\frac14\binom n2$.  Since $\bar v e^{aw}$ induces the metric
$e^{2w}\bar g$, the conformal Schouten formula gives
\begin{equation}\label{eq:relative-Schouten-step16}
 A_{e^{2w}\bar g}=A_{\bar g}-\nabla_{\bar g}^2w
 +dw\otimes dw-\frac12|dw|_{\bar g}^2\bar g.
\end{equation}
Consequently
\begin{equation}\label{eq:proper-residual-relative-step16}
 \mathcal N_g(\bar v e^{aw})=\bar v^{p_*}\left[
 e^{(n-4)w}\sigma_2(B_0+\mathcal A_1(w)+\mathcal A_2(w))
 -\kappa e^{nw}\right],
\end{equation}
with $|B_0|\le Cq$, $|\mathcal A_1(w)|\le Cq\mathfrak X_Q(w)$, and
$|\mathcal A_2(w)|\le Cq\mathfrak X_Q(w)^2$ on each unit model block.
Since $\sigma_2$ is quadratic and
$\mathcal V^{16/(n-4)}/q^2$ is uniformly bounded, Taylor's formula gives
\begin{equation}\label{eq:local-quadratic-majorant-step16}
 \mathfrak Y_Q(Q_N(w_1)-Q_N(w_2))
 \le C(\mathfrak X_Q(w_1)+\mathfrak X_Q(w_2))
       \mathfrak X_Q(w_1-w_2).
\end{equation}
This proves all fixed-block and local-supremum parts of the estimate.

The only issue is summation along the long Schwarzschild pieces.  Write
$w_i=b_i+z_i$, where $b_i=\Pi_0w_i$, and set
\[
 R_i=\|w_i\|_{\widehat X_{\gamma,\beta}},\qquad
 R_{12}=\|w_1-w_2\|_{\widehat X_{\gamma,\beta}},
\]
\[
 \rho_i=\mathfrak d_0[b_i]+m_{z_i},\qquad
 \rho_{12}=\mathfrak d_0[b_1-b_2]+m_{z_1-z_2}.
\]
By the definition of the refined norm,
\begin{equation}\label{eq:neutral-density-control-step16}
 \|b_i\|_{L^\infty}+\|\rho_i\|_{L^\infty_t}+\|\rho_i\|_{L^1_t}
 \le CR_i,
 \qquad
 \|b_1-b_2\|_{L^\infty}+\|\rho_{12}\|_{L^\infty_t}
 +\|\rho_{12}\|_{L^1_t}\le CR_{12}.
\end{equation}
The target and metric-support coefficient errors are bounded by a
nonnegative function $\mathfrak e_N$ with uniformly bounded $L^1_t$ norm,
by Lemma~\ref{lem:step8} and Proposition~\ref{prop:summable-operator-perturbation-step15}.

For the pure radial part, the identity
\eqref{eq:radial-null-identity-step16} and Lemma~\ref{lem:step4} give
\begin{align}\label{eq:exact-relative-radial-residual-step16}
 \mathscr R_\Sigma(b)
 ={}&4(n-1)e^{4ab}\frac{\widetilde q}{q_\Sigma}
 \left[-\frac12q_\Sigma^{-1}P_{\Sigma,0}^{(0)}b
       -\frac a2(q_\Sigma^{-1/2}b_t)^2\right]
 -\kappa\mathfrak v_\lambda e^{nb},
\end{align}
where
$\widetilde q/q_\Sigma
 =1-\frac12\tanh(at)q_\Sigma^{-1}b_t
   -\frac14(q_\Sigma^{-1/2}b_t)^2$.
The factor $q_\Sigma^{-1}b_t$ is controlled by the coarse Newton-strong
norm, using the tangential Schouten component away from the center and the
ordinary first-derivative term on the central cylinder.  Thus the radial
nonlinear difference is bounded by one $L^\infty_t$ factor and one transfer
density.  Every term containing a nonradial component is treated by the
unit-block Banach algebra estimate.  We obtain
\[
 n_{Q_N(w_1)-Q_N(w_2)}(t)
 \le C\bigl[(R_1+R_2)\rho_{12}
 +R_{12}(\rho_1+\rho_2)
 +\mathfrak e_N(R_1+R_2)R_{12}\bigr].
\]
Integrating and using \eqref{eq:neutral-density-control-step16} gives
\begin{equation}\label{eq:full-neck-L1-bilinear-step16}
 \sum_{\mathfrak e=\pm}\int_{\mathcal S_{\mathfrak e}}
 n_{Q_N(w_1)-Q_N(w_2)}(t)\,dt
 \le C(R_1+R_2)R_{12}.
\end{equation}
The transition terms are summed by the $L^\infty_t\times L^1_t$ structure
of the transfer norm.

For the weighted supremum terms, products containing a high mode are
controlled by the high-mode weight.  The only high harmonic produced by two
low factors comes from two first harmonics, and $2\beta>\gamma$ controls it;
products of two high factors in the weighted low-mode source are controlled
by $\beta<2\gamma$.  Together with
\eqref{eq:local-quadratic-majorant-step16} and
\eqref{eq:full-neck-L1-bilinear-step16}, this proves
\eqref{eq:quadratic-Lipschitz-step16} and
\eqref{eq:quadratic-self-step16}.

Differentiating \eqref{eq:proper-residual-relative-step16} places the
parameter derivative on one correction factor or one uniformly controlled
coefficient.  The same local and integrated estimates therefore give
\eqref{eq:quadratic-parameter-step16}.

Finally, pairing the invariant expansion with a test function and integrating
once by parts against the divergence-free first Newton tensor gives
\begin{equation}\label{eq:quadratic-energy-form-step16}
 |\langle Q_N(w_1)-Q_N(w_2),\phi\rangle|
 \le C_{Q,E}(\|w_1\|_{\widehat X_{\gamma,\beta}}
 +\|w_2\|_{\widehat X_{\gamma,\beta}})
 \|w_1-w_2\|_{\mathcal H}\|\phi\|_{\mathcal H}
\end{equation}
for zero-trace localizations, with the analogous differentiated estimate.
The same pointwise calculation in the one-sided bubble norm gives a constant
$C_{Q,b}$.  We choose $r_{\rm nl}$ so that
\begin{equation}\label{eq:final-energy-radius-step16}
 8\max\{C_{Q,E},C_{Q,b}\}r_{\rm nl}\le1.
\end{equation}
\end{proof}

Let $\mathcal X$ and $\mathcal H$ be Hilbert spaces, let
$\mathcal T:\mathcal X\to\mathcal H$ be a bounded surjective trace map,
and choose a bounded right inverse
$\mathcal E:\mathcal H\to\mathcal X$.  Put
$\mathcal X_0=\ker\mathcal T$.  For $j=0,1$, let
$\mathfrak a_j$ be a bounded symmetric bilinear form on $\mathcal X$, let
$\boldsymbol\Psi_j:\mathbb R^m\to\mathcal X'$ and
$\boldsymbol\ell_j:\mathcal X\to\mathbb R^m$ be bounded, and define the
zero-trace augmented operator
\begin{equation}\label{eq:abstract-augmented-dirichlet-map}
 \mathbb A_j^D(v,\mathbf d)
 :=\bigl(\mathfrak a_j(v,\cdot)-\boldsymbol\Psi_j\mathbf d,
          \boldsymbol\ell_j(v)\bigr),
 \qquad
 \mathbb A_j^D:\mathcal X_0\times\mathbb R^m
 \longrightarrow\mathcal X_0'\times\mathbb R^m.
\end{equation}

\begin{proposition}[Stability of augmented Dirichlet problems and Dirichlet-to-Neumann maps]
\label{aux:abstract-energy-dtn-stability}
Assume
\begin{equation}\label{eq:abstract-base-inverse}
 \|(\mathbb A_0^D)^{-1}\|\le C_0
\end{equation}
and, for some $0<\eta\le1$,
\begin{align}
 |(\mathfrak a_1-\mathfrak a_0)(v,w)|
 &\le\eta\|v\|_{\mathcal X}\|w\|_{\mathcal X},
 \label{eq:abstract-form-closeness}\\
 \|\boldsymbol\Psi_1-\boldsymbol\Psi_0\|_{\mathbb R^m\to\mathcal X'}
 +\|\boldsymbol\ell_1-\boldsymbol\ell_0\|_{\mathcal X\to\mathbb R^m}
 &\le\eta.
 \label{eq:abstract-finite-rank-closeness}
\end{align}
There is a constant $c_{\rm stab}>0$, depending only on $C_0$, the
uniform bounds of the displayed maps, and $\|\mathcal E\|$, such that if
$c_{\rm stab}\eta\le\frac12$, then $\mathbb A_1^D$ is invertible and
\begin{equation}\label{eq:abstract-perturbed-inverse}
 \|(\mathbb A_1^D)^{-1}\|\le2C_0.
\end{equation}

For $\phi\in\mathcal H$, let
$(H_j\phi,\mathbf d_j(\phi))$ be the unique augmented Poisson pair
satisfying
\begin{equation}\label{eq:abstract-poisson-pair}
 \mathcal T H_j\phi=\phi,
 \qquad \boldsymbol\ell_j(H_j\phi)=0,
 \qquad
 \mathfrak a_j(H_j\phi,\zeta)
 =\langle\boldsymbol\Psi_j\mathbf d_j(\phi),\zeta\rangle
 \quad(\zeta\in\mathcal X_0).
\end{equation}
Define its weak Dirichlet-to-Neumann map by
\begin{equation}\label{eq:abstract-dtn-definition}
 \langle\mathcal D_j\phi,\psi\rangle
 :=\mathfrak a_j(H_j\phi,\mathcal E\psi)
 -\langle\boldsymbol\Psi_j\mathbf d_j(\phi),\mathcal E\psi\rangle.
\end{equation}
This definition is independent of the chosen right inverse of the trace.
Moreover,
\begin{align}
 \|H_j\phi\|_{\mathcal X}+|\mathbf d_j(\phi)|
 +\|\mathcal D_j\phi\|_{\mathcal H'}
 &\le C\|\phi\|_{\mathcal H},
 \label{eq:abstract-poisson-dtn-bound}\\
 \|H_1-H_0\|_{\mathcal H\to\mathcal X}
 +\|\mathbf d_1-\mathbf d_0\|_{\mathcal H\to\mathbb R^m}
 +\|\mathcal D_1-\mathcal D_0\|_{\mathcal H\to\mathcal H'}
 &\le C\eta.
 \label{eq:abstract-dtn-stability}
\end{align}

If $F\in\mathcal X'$ and
$(u_j^0,\mathbf c_j^0)=(\mathbb A_j^D)^{-1}
(F|_{\mathcal X_0},\mathbf g)$, define the weak outward flux
$h_{j,F,\mathbf g}\in\mathcal H'$ by
\begin{equation}\label{eq:abstract-inhomogeneous-flux}
 \langle h_{j,F,\mathbf g},\psi\rangle
 :=\mathfrak a_j(u_j^0,\mathcal E\psi)
 -\langle F+\boldsymbol\Psi_j\mathbf c_j^0,\mathcal E\psi\rangle,
\end{equation}
then
\begin{equation}\label{eq:abstract-inhomogeneous-flux-bound}
 \|u_j^0\|_{\mathcal X}+|\mathbf c_j^0|
 +\|h_{j,F,\mathbf g}\|_{\mathcal H'}
 \le C(\|F\|_{\mathcal X'}+|\mathbf g|).
\end{equation}

Finally, suppose all domains have been identified with fixed spaces so that
$\mathcal T$, $\mathcal X_0$, and the chosen trace extension
$\mathcal E$ are parameter independent, while the forms, source maps, and
constraints depend $C^1$ on a parameter $q$, with their first derivatives
uniformly bounded.  Then the
inverse, Poisson, flux, and Dirichlet-to-Neumann maps are $C^1$ and
\begin{align}
 \|\partial_qH_j\|_{\mathcal H\to\mathcal X}
 +\|\partial_q\mathbf d_j\|_{\mathcal H\to\mathbb R^m}
 +\|\partial_q\mathcal D_j\|_{\mathcal H\to\mathcal H'}&\le C,
 \label{eq:abstract-dtn-parameter-bound}\\
 \|\partial_qu_j^0\|_{\mathcal X}+|\partial_q\mathbf c_j^0|
 +\|\partial_qh_{j,F,\mathbf g}\|_{\mathcal H'}
 &\le C\bigl(\|\partial_qF\|_{\mathcal X'}
              +\|F\|_{\mathcal X'}
              +|\partial_q\mathbf g|+|\mathbf g|\bigr).
 \label{eq:abstract-flux-parameter-bound}
\end{align}
\end{proposition}

\begin{proof}
By \eqref{eq:abstract-form-closeness}--
\eqref{eq:abstract-finite-rank-closeness},
$\|\mathbb A_1^D-\mathbb A_0^D\|\le C\eta$.  The resolvent identity and a
Neumann series give \eqref{eq:abstract-perturbed-inverse}.

Write $H_j\phi=\mathcal E\phi+v_j$.  The pair
$(v_j,\mathbf d_j(\phi))\in\mathcal X_0\times\mathbb R^m$ solves the
zero-trace augmented problem with data
$(-\mathfrak a_j(\mathcal E\phi,\cdot),
 -\boldsymbol\ell_j(\mathcal E\phi))$.  The inverse bounds give
\eqref{eq:abstract-poisson-dtn-bound}; subtracting the two equations gives
the $O(\eta)$ bound for the Poisson pairs, and substitution into
\eqref{eq:abstract-dtn-definition} gives
\eqref{eq:abstract-dtn-stability}.  Independence of the trace extension
follows because two extensions differ by an element of $\mathcal X_0$.
The same inverse estimate and the weak flux definition yield
\eqref{eq:abstract-inhomogeneous-flux-bound}.

After identifying the domains, differentiation of the inverse equation gives
\[
 \mathbb A_j^D\partial_q(v,\mathbf d)
 =\partial_q\mathbf F-(\partial_q\mathbb A_j^D)(v,\mathbf d).
\]
Applying the preceding bounds and differentiating the weak flux formulas
proves \eqref{eq:abstract-dtn-parameter-bound}--
\eqref{eq:abstract-flux-parameter-bound}.
\end{proof}

\subsection{Uniform energy estimates}
We now collect the energy estimates needed for the bubble localization.  Write
$\mathcal B_T=\{t<T\}$ for the truncated round-bubble block and equip it
with the Liouville trace
\[
 \mathcal T_Tu=\cosh(T)^{-\frac{n-2}{2}}u|_{t=T}
 \in H^{1/2}(\mathbb S^{n-1}).
\]
We first fix two canonical low-mode singular fields.  Their construction
depends only on the round model and the normalized parameters.  For each scalar sector $\ell=0,1$ and each normalized parameter value,
choose a cutoff $\chi$ which vanishes on $\{t\le t_0\}$ and equals one on
$\{t\ge t_0+1\}$.  With $\widehat z_\ell$ from
\eqref{eq:round-singular-low-modes-rigorous}, the classical commutator
$-[\mathscr L_*,\chi]\widehat z_\ell$ is smooth and compactly supported.
Proposition~\ref{lem:step12} gives a smooth complete-sphere correction
$p_\ell^{\rm s}$ and a coefficient $e_\ell^{\rm s}$ such that
\[
 S_\ell:=\chi\widehat z_\ell+p_\ell^{\rm s},\qquad
 \mathscr L_*S_\ell=e_\ell^{\rm s}\Psi_\ell,
 \qquad \ell_\ell^*(S_\ell)=0.
\]
On $\{t\ge t_0+1\}$ one has
$S_\ell=\widehat z_\ell+A_\ell^{\rm s}z_\ell$.  Hence, by
\eqref{eq:kappa-low-rigorous}, we may choose $T_E\ge T_*$ so large that
\begin{equation}\label{eq:large-bubble-truncation-threshold-step17}
 |Y_{S_\ell}(T)|\ge c e^{\kappa_\ell T},\qquad
 \|S_\ell\|_{H^1(\mathcal B_T,g_*)}+|e_\ell^{\rm s}|
 \le C e^{\kappa_\ell T}
 \quad(\ell=0,1,\ T\ge T_E).
\end{equation}
The complete-sphere inverse and the localized Gram matrix depend smoothly
on the normalized parameters.  The fields $S_\ell$ can therefore be chosen
smoothly over the compact normalized parameter set, and the displayed
bounds, as well as their first normalized parameter derivatives, are
uniform.  All constants here are fixed by the round model, the cutoff, and
the localized Gram matrix.

Set
\[
 \mathcal H_\partial=
 \bigl(H^{1/2}(\mathbb S^{n-1})
       \oplus H^{1/2}(\mathbb S^{n-1})\bigr)^J.
\]
Let $\mathcal X_{\rm ext}$ be the energy completion of blockwise fields on
the two ideal transitions, the two ideal Schwarzschild half-necks, and the
centered connector, subject to internal value matching and the centered
parity condition.  Its norm is the sum of the $H^1$ norms in the
corresponding Liouville gauges and the fixed $H^1$ connector norm; write
$\mathcal X_{\rm ext}^0=\ker\mathcal T_{\rm ext}$ for the zero bubble-side
trace space.

\begin{proposition}[Energy estimates for the ideal problem and coefficient comparisons]
\label{aux:preparameter-energy-package}
For every $T\ge T_E$, the zero-trace round augmented map
\[
 (u,\mathbf d)\longmapsto
 \bigl(\mathscr L_*u-\textstyle\sum_a d_a\Psi_a,
       (\ell_b^*(u))_{b=0}^n\bigr)
\]
takes
\[
 \bigl(\ker\mathcal T_T\cap H^1(\mathcal B_T,g_*)\bigr)
 \times\mathbb R^{n+1}
\]
isomorphically onto the corresponding energy dual space times
$\mathbb R^{n+1}$, with inverse norm independent of $T$ and of the
normalized bubble parameters.  The bubble-side trace
$\mathcal T_{\rm ext}:\mathcal X_{\rm ext}\to\mathcal H_\partial$ is
surjective and has a right inverse of uniform norm.  On
$\mathcal X_{\rm ext}^0$ the ideal exterior operator has a uniform inverse.
If $\mathcal D_T^{\rm bub}$ and
$\mathcal D^{\rm ext}$ are the weak Dirichlet-to-Neumann maps defined with
the common Liouville trace and the common outward-flux normalization, then
\begin{equation}\label{eq:preparameter-ideal-schur-bound}
 \mathcal S^{\rm id}:=\mathcal D_T^{\rm bub}+\mathcal D^{\rm ext}
 :\mathcal H_\partial\longrightarrow\mathcal H_\partial'
 \quad\hbox{is an isomorphism},
 \qquad
 \|(\mathcal S^{\rm id})^{-1}\|\le C_E^{\rm id},
\end{equation}
where $C_E^{\rm id}$ is independent of $T$, $L\ge L_*$, the
Schwarzschild half-neck lengths, and the normalized parameters.

There are constants $C_{E,{\rm geo}}<\infty$ and, for every fixed
$L\ge L_*$, $C_{E,{\rm op}}(L)<\infty$, fixed before the choice of
$L,N_0,N$, with the following property.  After pulling all blocks to the
common Liouville coordinates, the finite-flat versus ideal exterior forms
satisfy
\begin{equation}\label{eq:preparameter-finite-ideal-form-bound}
 |\mathfrak a_{\rm ext}^{\rm fl}(v,w)
   -\mathfrak a_{\rm ext}^{\rm id}(v,w)|
 \le C_{E,{\rm geo}}
 \bigl(C_L\delta_\lambda+e^L\lambda^{1-4/n}+N^{-1}\bigr)
 \|v\|_{\mathcal X_{\rm ext}}\|w\|_{\mathcal X_{\rm ext}}.
\end{equation}
On every bubble and exterior block, the actual versus finite-flat forms,
localized source maps, and localized constraints satisfy the hypotheses of
Proposition~\ref{aux:abstract-energy-dtn-stability} with perturbation
size
\begin{equation}\label{eq:preparameter-actual-flat-form-bound}
 C_{E,{\rm op}}(L)
 \bigl(\epsilon_N^{\kappa_f}+\delta_{\!*}(N_0)\bigr).
\end{equation}
The same statements hold after one normalized parameter derivative under
the fixed-domain identifications used in Section~7.
\end{proposition}

\begin{proof}
We first prove the uniform trace estimate needed in all three block
problems.  On a terminal unit cylinder write
$u=\cosh(t)^{(n-2)/2}Y$.  The round $H^1$ norm controls
$\partial_tY+\frac{n-2}{2}\tanh(t)Y$ and $\nabla_\theta Y$ there.  The
one-dimensional Hardy inequality at the second stereographic pole controls
the remaining $L^2$ term.  Consequently
\[
 \|Y(T,\cdot)\|_{H^{1/2}}
 \le C\|u\|_{H^1(\mathcal B_T,g_*)},
\]
with $C$ independent of $T$.  Extending $Y(T,\cdot)$ through a terminal
cylinder of fixed width gives a right inverse with the same uniform bound.
The same argument applies to every transition-side trace.  This also
fixes the common flux normalization: the physical conormal is multiplied
on the two sides of an exact interface by the same positive interface
density, so division by that density gives the Liouville outward flux used
in the Dirichlet-to-Neumann maps.

For the round zero-trace bubble problem, decompose in angular harmonics.
If $\ell\ge2$, extension by zero through the omitted cap belongs to
$H^1(\mathbb S^n,g_*)$ and contains only full spherical degrees
$k\ge\ell\ge2$.  The round spectral gap therefore gives
\[
 \int_{\mathbb S^n}(|\nabla u|^2-nu^2)\,dv_{g_*}
 \ge c\|u\|_{H^1(\mathbb S^n,g_*)}^2.
\]
For $\ell=0,1$, extend the source by zero to the complete sphere, solve the
complete augmented problem by Proposition~\ref{lem:step12}, and add the canonical
singular field $S_\ell$ so as to impose the zero boundary trace.  The bounds
\eqref{eq:large-bubble-truncation-threshold-step17} and the complete-sphere
Hardy trace estimate give a correction coefficient of size
$O(e^{-\kappa_\ell T})$ times the data, while the energy of $S_\ell$ is
$O(e^{\kappa_\ell T})$.  Their product is uniformly bounded.  If the data
vanish, subtracting the singular multiple leaves a smooth complete-sphere
augmented homogeneous solution, which vanishes by Proposition~\ref{lem:step12};
the zero boundary trace then removes the singular multiple.  This proves
the first assertion, including uniqueness.

For the ideal exterior zero-trace problem, use the common Liouville gauges.
In every angular sector the transition form has potential bounded below by
\eqref{eq:transition-dtn-diagonal-rigorous}; on each Schwarzschild
half-neck the conjugated form is
$\int(|Y'|^2+\nu_\ell^2Y^2)$; and the centered connector is a fixed block.
The centered parity condition eliminates its unused endpoint trace.  Thus
the high sectors are uniformly coercive.  In the two low sectors the
zero-boundary Green operators and
$d_{++}-d^\Sigma_{\ell,T_\Sigma}\ge\kappa_*$ give the same uniform bound.
This proves the exterior zero-trace inverse.

For the ideal Schur complement, angular degrees at least two are controlled
by the sum of the positive bubble and exterior energies.  In degrees zero
and one, eliminating the neck-side trace gives
\[
 \mathcal D_{\ell,T}^{\rm bub}+d_{--}
 -d_{-+}(d_{++}-d^\Sigma_{\ell,T_\Sigma})^{-1}d_{+-}.
\]
By Lemmas~\ref{aux:bubble-dtn-rigorous} and~\ref{aux:transition-dtn-rigorous}, the first two positive terms are bounded
below independently of all lengths, whereas the absolute value of the last
term is $O(e^{-4\kappa_*L})$.  Increasing the already fixed $L_*$ gives a
uniform positive lower bound.  Lax--Milgram in the high sectors and
finite-dimensional inversion in the low sectors prove
\eqref{eq:preparameter-ideal-schur-bound}.

We next record the perturbation bounds used in the parameter choice.  The finite-flat coefficient comparisons in
the proof of Proposition~\ref{prop:flat-paired-inverse-step15} give
$O(C_L\delta_\lambda)$ on a transition,
$O(e^L\lambda^{1-4/n})$ in the bubble--transition coordinate and flux
identifications, and $O(N^{-1})$ on the off-center connector.  Writing the
principal parts in the common Liouville gauges, multiplying the coefficient
difference by the two first derivatives, and applying Cauchy--Schwarz gives
\eqref{eq:preparameter-finite-ideal-form-bound}; the zeroth-order and fixed
finite-rank terms are bounded in the same or a stronger norm.
Lemma~\ref{lem:step8} gives, relative to the model Newton form, the
actual--flat coefficient error
$O(\epsilon_N^{\kappa_f}+\delta_{\!*}(N_0))$ on every bubble and exterior
block, together with the same estimate for one normalized parameter
derivative.  The localized sources and constraints are supported in fixed
bubble subcores and have uniformly bounded pullbacks by
Propositions~\ref{lem:step11}--\ref{lem:step12}.  Another Cauchy--Schwarz
estimate therefore gives \eqref{eq:preparameter-actual-flat-form-bound} and
the finite-rank closeness required in
\eqref{eq:abstract-finite-rank-closeness}.  Proposition~\ref{aux:abstract-energy-dtn-stability} then gives the associated
zero-trace, Poisson, weak-flux, and Dirichlet-to-Neumann perturbation
constants.  All constants here depend only on the preceding estimates.
\end{proof}

Let $C_E^{\rm id}$, $C_{E,{\rm geo}}$, and $C_{E,{\rm op}}(L)$ be the
constants from Proposition~\ref{aux:preparameter-energy-package}.  Replacing
$C_E^{\rm id}$ by $\max\{1,C_E^{\rm id}\}$ and enlarging
$C_{E,{\rm geo}}$ if necessary, we may assume that it also incorporates the
stability, zero-trace inverse, Poisson, weak-flux, and
Dirichlet-to-Neumann constants from
Proposition~\ref{aux:abstract-energy-dtn-stability} for the finite-flat and
ideal exterior problems.  Thus every corresponding inverse or
Dirichlet-to-Neumann difference is bounded by
\[
 C_{E,{\rm geo}}
 \bigl(C_L\delta_\lambda+e^L\lambda^{1-4/n}+N^{-1}\bigr)
\]
whenever this quantity is within the abstract stability threshold.  For
each fixed $L$, enlarge $C_{E,{\rm op}}(L)$ similarly to include the
corresponding constants for the actual--flat bubble and exterior problems.
After these choices, decrease $\vartheta_{\rm geo}$
from Proposition~\ref{prop:flat-paired-inverse-step15} so that
\begin{equation}\label{eq:energy-geometric-margin-step17}
 C_E^{\rm id}C_{E,{\rm geo}}\vartheta_{\rm geo}\le\frac14.
\end{equation}
Finally, let $C_{\rm pr}$ be the fixed constant in the following
consequence of the estimate for the principal pair in Lemma~\ref{lem:step8}:
the paired first metric variation generated by any principal inner core has
$\widehat Y_{\gamma,\beta}$ norm at most
$C_{\rm pr}\epsilon_N^{\kappa_f}$.  These constants depend only on the
preceding estimates.

For the remainder of the reduction set
\begin{equation}\label{eq:nonlinear-residual-size-step17}
 \varepsilon_N
 :=\eta_{L,N_0}+C_L\lambda^{2-8/n}+C\epsilon_N^{\kappa_f}.
\end{equation}
Choose $C_R\ge1$, independent of $L,N_0,N$, so that the residual
estimate \eqref{eq:refined-residual-step16} and its differentiated
form are bounded by $C_R\varepsilon_N$.  Enlarge the fixed constant $C_*$ in
\eqref{eq:refined-inverse-step16}, if necessary, and set
\begin{equation}\label{eq:refined-majorant-step17}
 K_{\rm ref}(L):=C_*(1+L).
\end{equation}
Enlarge $C_*$, if necessary, so that $K_{\rm ref}(L)$ dominates
$C_{\rm ref}(L)$, $C_{\rm tr}^0(L)$, and
$C_{\rm tr}^{\rm fl}(L)$ for every $L\ge L_*$.  Thus
$K_{\rm ref}(L)$ is fixed once $L$ is selected.

\begin{proposition}[Choice of parameters]
\label{prop:ordered-parameters-step17}
There is a constant $r_{\rm cone}>0$, independent of $L,N_0,N$, for which
one can first choose $L$, then define a positive number
$\varepsilon_\#(L)$, and subsequently select the remaining parameters in
the order
\[
 L\longrightarrow N_0\longrightarrow N
\]
so that all hypotheses of
Lemmas~\ref{lem:step10},~\ref{lem:step14},~\ref{lem:step15},
and~\ref{lem:step16} hold and
\begin{equation}\label{eq:ordered-residual-threshold-step17}
 K_{\rm ref}(L)\|\mathcal N_g(\bar v_{N,\lambda,\xi})
       \|_{\widehat Y_{\gamma,\beta}}
 \le K_{\rm ref}(L)C_R\varepsilon_N
 \le\varepsilon_\#(L),
\end{equation}
with the same bound after one normalized parameter derivative.  Moreover,
\begin{equation}\label{eq:ordered-contraction-threshold-step17}
 2K_{\rm ref}(L)C_Q(2\varepsilon_\#(L))\le\frac12,
 \qquad
 2\varepsilon_\#(L)\le\min\{r_{\rm nl},r_{\rm cone}\}.
\end{equation}
Here $r_{\rm cone}$ is a uniform model-relative Newton-cone radius: if
$\|w\|_{\widehat X_{\gamma,\beta}}\le r_{\rm cone}$, then
$e^{2w}(\bar v)^{8/(n-4)}g$ remains Newton elliptic with fixed
model-relative comparison constants.
The radius $r_{\rm nl}$ is the one fixed in Lemma~\ref{lem:step16} and satisfies the energy-form absorption condition \eqref{eq:final-energy-radius-step16}.
\end{proposition}

\begin{proof}
Lemma~\ref{lem:step10}, the conformal Schouten formula, and the definition of
$\widehat X_{\gamma,\beta}$ give a uniform $r_{\rm cone}>0$: on each unit
model block the change of the Newton tensor is bounded by
$C(\|w\|_{\widehat X}+\|w\|_{\widehat X}^2)$ times the model Newton tensor.

We now choose the parameters successively.  Since
$K_{\rm ref}(L)=O(1+L)$, first take $L\ge L_*$ so large that all terms of the
form $K_{\rm ref}(L)e^{-a_\Sigma L}$ and
$K_{\rm ref}(L)e^{-nL}$ occurring in the residual, refined inverse, and
energy comparisons are below their required thresholds.  With this $L$
fixed, set
\begin{equation}\label{eq:epsilon-sharp-definition-step17}
 \varepsilon_\#(L)
 :=\min\left\{\frac{r_{\rm nl}}8,\frac{r_{\rm cone}}8,
 \frac1{32K_{\rm ref}(L)C_Q},1\right\}.
\end{equation}
Increasing $L$ if necessary, we may also arrange
$K_{\rm ref}(L)C_Re^{-a_\Sigma L}\le\varepsilon_\#(L)/4$ and the target
absorption required in Proposition~\ref{prop:refined-transfer-step16}.

Next choose $N_0$ large.  Because
$\delta_{\!*}(N_0)+N_0^{-1}\to0$, all $N_0$-dependent smallness conditions
in the global inverse, transition estimate, refined inverse, and finite-rank
energy comparison can be imposed simultaneously.  In particular, we require
\begin{equation}\label{eq:choose-N0-energy-step17}
 C_E^{\rm id}C_{E,{\rm op}}(L)\delta_{\!*}(N_0)\le\frac1{16},
\end{equation}
The remaining $N_0$-dependent terms are chosen so that
\[
 K_{\rm ref}(L)C_R\,(\text{their contribution to }\varepsilon_N)
 \le \frac{\varepsilon_\#(L)}4.
\]

Finally, with $L$ and $N_0$ fixed, choose the selected-index threshold so
large that all remaining scale-separation and smallness conditions hold.
This is possible because $\lambda\asymp\epsilon_N\to0$ and
$(1+|\log\epsilon_N|)\epsilon_N^{\kappa_f}\to0$.  We record explicitly the three conditions needed in Section~7:
\begin{align}
 C_L\delta_\lambda+e^L\lambda^{1-4/n}+N^{-1}
 &\le\vartheta_{\rm geo}e^{-L},
 \label{eq:choose-N-flat-geometry-step17}\\
 C_E^{\rm id}C_{E,{\rm op}}(L)\epsilon_N^{\kappa_f}
 &\le\frac1{16},
 \label{eq:choose-N-energy-step17}\\
 K_{\rm ref}(L)C_{\rm pr}\epsilon_N^{\kappa_f}
 &\le r_{\rm nl}.
 \label{eq:choose-N-principal-radius-step17}
\end{align}
At the same stage we impose the $N$-dependent parts of
\eqref{eq:flat-ordered-transition-smallness-step15},
\eqref{eq:actual-inverse-ordered-smallness-step15},
\eqref{eq:transition-admissible-smallness-step14}, and the two Neumann-series
absorptions in Proposition~\ref{prop:refined-transfer-step16}, and require
all bubble truncation parameters to be at least $T_E$.  These form a finite
collection of conditions, each tending to zero with $N$.

The residual estimate now splits into its $e^{-a_\Sigma L}$,
$\delta_{\!*}(N_0)+N_0^{-1}$, and positive-power-of-$\epsilon_N$ parts;
the three choices above give \eqref{eq:ordered-residual-threshold-step17}.
Equation \eqref{eq:ordered-contraction-threshold-step17} follows from
\eqref{eq:epsilon-sharp-definition-step17}.  Thus all preceding analytic
hypotheses hold with the stated order of choices.

Increasing the selected-index threshold does not affect the scale
relations $\lambda^{n-4}=o(s_N)$ and
$A_N\lambda^{a_\Sigma}=o(s_N)$ or any of the local response limits.
\end{proof}

\begin{lemma}[Projected nonlinear solution]\label{lem:step17}
For the parameters selected in
Proposition~\ref{prop:ordered-parameters-step17}, and every subsequent selected
index, there are unique
\[
 w_{N,\lambda,\xi}\in\widehat X_{\gamma,\beta},
 \qquad c_0(N,\lambda,\xi),\ldots,c_n(N,\lambda,\xi)\in\mathbb R,
\]
in the ball selected below such that
\begin{equation}\label{eq:projected-nonlinear-equation-step17}
 \mathcal N_g\!\left(
   \bar v_{N,\lambda,\xi}e^{aw_{N,\lambda,\xi}}\right)
 =\sum_{a=0}^nc_a(N,\lambda,\xi)\Psi_{a,N},
 \qquad
 \mathfrak L_b(w_{N,\lambda,\xi})=0.
\end{equation}
They satisfy the uniform estimate
\begin{equation}\label{eq:nonlinear-solution-bound-step17}
 \|w_{N,\lambda,\xi}\|_{\widehat X_{\gamma,\beta}}
 +\sum_{a=0}^n|c_a(N,\lambda,\xi)|
 \le C\varepsilon_N.
\end{equation}
The correction and coefficients are $C^1$ in $(\lambda',\xi')$, and every
normalized parameter derivative obeys
\begin{equation}\label{eq:nonlinear-parameter-bound-step17}
 \|\mathfrak Dw_{N,\lambda,\xi}\|_{\widehat X_{\gamma,\beta}}
 +\sum_{a=0}^n|\mathfrak Dc_a(N,\lambda,\xi)|
 \le C\varepsilon_N.
\end{equation}
The constant may depend on the selected transition length $L$, but is
independent of $N_0,N$ and of the normalized finite-dimensional parameters
after $L$ is fixed.  The
metric
\[
 v_{N,\lambda,\xi}^{8/(n-4)}g,
 \qquad
 v_{N,\lambda,\xi}
 =\bar v_{N,\lambda,\xi}e^{aw_{N,\lambda,\xi}},
\]
is positive and Newton elliptic with uniform model-relative constants.  If
all projected coefficients vanish, it solves the normalized $\sigma_2$
equation and lies in $\Gamma_2^+$.
\end{lemma}

\begin{proof}
Let $R_N=\mathcal N_g(\bar v_{N,\lambda,\xi})$ and write
$\mathbb G_NF=(G_NF,\mathbf C_NF)$ for the refined projected inverse.
Proposition~\ref{prop:refined-transfer-step16} gives
\[
 \|R_N\|_{\widehat Y_{\gamma,\beta}}
 +\|\mathfrak DR_N\|_{\widehat Y_{\gamma,\beta}}
 \le C_R\varepsilon_N,
 \qquad
 \|G_NF\|_{\widehat X_{\gamma,\beta}}+|\mathbf C_NF|
 \le C_{\rm ref}(L)\|F\|_{\widehat Y_{\gamma,\beta}}.
\]
Set
\[
 \mathcal T_N(w)=G_N[-R_N-Q_N(w)],
 \qquad
 \varrho_N=2C_{\rm ref}(L)C_R\varepsilon_N.
\]
By Proposition~\ref{prop:ordered-parameters-step17},
\begin{equation}\label{eq:contraction-smallness-step17}
 \varrho_N\le\min\{r_{\rm nl},r_{\rm cone}\},
 \qquad
 2C_{\rm ref}(L)C_Q\varrho_N\le\frac12.
\end{equation}
Lemma~\ref{lem:step16} therefore shows that $\mathcal T_N$ maps the ball
$\{\|w\|_{\widehat X}\le\varrho_N\}$ into itself and has Lipschitz constant
at most $1/2$.  Its unique fixed point, together with
$\mathbf c_N=\mathbf C_N[-R_N-Q_N(w)]$, satisfies
\eqref{eq:projected-nonlinear-equation-step17} and
\eqref{eq:nonlinear-solution-bound-step17}.

The fixed-point equation is $C^1$ in the normalized parameters.
Differentiating it and using \eqref{eq:quadratic-parameter-step16} gives
\[
 \|\mathfrak Dw\|_{\widehat X}+|\mathfrak D\mathbf c_N|
 \le C\bigl(\varepsilon_N+\varrho_N^2
       +\varrho_N\|\mathfrak Dw\|_{\widehat X}\bigr).
\]
The last term is absorbed by \eqref{eq:contraction-smallness-step17}, proving
\eqref{eq:nonlinear-parameter-bound-step17}.

Finally, $\|w\|_{\widehat X}\le r_{\rm cone}$ preserves the Newton gap and
positivity.  If all $c_a$ vanish, then the normalized equation holds; the
positive first Newton tensor has trace $(n-1)\sigma_1>0$, so the solution
lies in $\Gamma_2^+$.
\end{proof}

We have now obtained a $C^1$ projected solution with uniform control in the
refined spaces.  It remains to identify the projected coefficients with the
derivatives of the reduced energy and then to derive its asymptotic
expansion.  We first record the variational identity used in the
finite-dimensional reduction.

For a positive conformal factor $v$ on $\mathbb S^n$, define the normalized
$\sigma_2$ energy
\begin{equation}\label{eq:global-energy-step18}
 \mathscr E_g(v)
 :=\int_{\mathbb S^n}
   \sigma_2(g_v^{-1}A_{g_v})\,dv_{g_v}
 -\frac{n-4}{4n}\binom n2
  \int_{\mathbb S^n}dv_{g_v},
 \qquad g_v=v^{8/(n-4)}g.
\end{equation}

\begin{lemma}[Variational reduction identity]\label{lem:step18}
Let $v_{N,\lambda,\xi}$ be the projected solution from
Lemma~\ref{lem:step17}, and define
\[
 \mathscr J_N(\lambda',\xi')
 =\mathscr E_g\bigl(v_{N,\epsilon_N\lambda',\epsilon_N\xi'}\bigr),
 \qquad
 a_0=\log\lambda',\quad a_j=\xi_j',\ 1\le j\le n.
\]
For the paired source functions $\Psi_{a,N}$ used in the projected equation,
put
\begin{equation}\label{eq:variational-matrix-definition-step18}
 M^N_{ab}
 :=\int_{\mathbb S^n}
   \Psi_{a,N}\,v_{N,\lambda,\xi}^{-1}
   \partial_{a_b}v_{N,\lambda,\xi}\,dv_g,
 \qquad 0\le a,b\le n.
\end{equation}
Then the exact identity
\begin{equation}\label{eq:variational-reduction-step18}
 \partial_{a_b}\mathscr J_N
 =4\sum_{a=0}^nc_aM^N_{ab}
\end{equation}
holds.  Moreover, uniformly on the normalized parameter set,
\begin{equation}\label{eq:variational-matrix-estimate-step18}
 \bigl\|M^N-2aI_{n+1}\bigr\|
 \le C\varepsilon_N,
\end{equation}
where $\varepsilon_N$ is defined in
\eqref{eq:nonlinear-residual-size-step17}.  After the ordered choice of
$L,N_0$ and for all sufficiently large selected indices, $M^N$ is uniformly
invertible.  Consequently every interior critical point of $\mathscr J_N$
solves the full normalized equation and gives a positive
$\Gamma_2^+$-admissible metric.
\end{lemma}

\begin{proof}
For a conformal variation $\widehat g_s=e^{2s\varphi}\widehat g$, the
variation of the Schouten endomorphism is
$-2\varphi\widehat g^{-1}A_{\widehat g}
 -\widehat g^{-1}\nabla_{\widehat g}^2\varphi$.  Since the first Newton
tensor is divergence free, integration by parts gives
\[
 \left.\frac d{ds}\right|_{0}
 \int_{\mathbb S^n}\sigma_2(\widehat g_s^{-1}A_{\widehat g_s})
 \,dv_{\widehat g_s}
 =(n-4)\int_{\mathbb S^n}\varphi\sigma_2\,dv_{\widehat g}.
\]
Applying this to $g_v$ and the additive variation $v+s\psi$, and including
the volume term in \eqref{eq:global-energy-step18}, yields
\begin{equation}\label{eq:energy-first-variation-step18}
 D\mathscr E_g(v)[\psi]
 =4\int_{\mathbb S^n}\mathcal N_g(v)v^{-1}\psi\,dv_g.
\end{equation}
The chain rule and \eqref{eq:projected-nonlinear-equation-step17} now give
\eqref{eq:variational-reduction-step18}.

Write $v=\bar ve^{aw}$.  Since
$\partial_{a_0}=\lambda\partial_\lambda$ and
$\partial_{a_j}=\epsilon_N\partial_{\xi_j}$,
\[
 v^{-1}\partial_{a_b}v
 =a(\mathcal Z_b+\partial_{a_b}w).
\]
The paired sources are supported in fixed subcores of the two exact bubble
charts.  Proposition~\ref{lem:step11}, Proposition~\ref{lem:step12}, and
\eqref{eq:bubble-source-duality-step12} give
\[
 \sum_{\pm}\int\Psi_{a,N}\mathcal Z_b\,dv_g=2\delta_{ab}.
\]
On these subcores, Lemma~\ref{lem:step17} gives
$\|\partial_{a_b}w\|_{C^0}\le C\varepsilon_N$, while the intrinsic
$L^1$ norms of the sources are uniformly bounded.  This proves
\eqref{eq:variational-matrix-estimate-step18}.  Hence $M^N$ is uniformly
invertible for large selected indices, and at a critical point
\eqref{eq:variational-reduction-step18} implies $c_0=\cdots=c_n=0$.
\end{proof}

\section{Energy localization and expansion}

Put
\begin{equation}\label{eq:AN-sN}
 A_N=\mu_N\epsilon_N^{d_f},
 \qquad
 s_N=A_N^2=\mu_N^2\epsilon_N^{2d_f}
 =\epsilon_N^{n/2+2m_f}.
\end{equation}

\subsection{The one-ended principal model}
Fix a radial function $\eta_\sharp\in C_c^\infty([0,\infty))$ which is one
on $[0,1]$ and zero on $[2,\infty)$.  For the fixed transition length $L$,
choose $c_L>0$ so small that, uniformly for
$\lambda=\epsilon_N\lambda'$ and $1/2\le\lambda'\le3/2$,
\[
 B_{2\sigma_N^\sharp}(x_N^-)
 \subset\{\,|x-c_N(\xi)|<R_\lambda^-e^{-L}\,\},
 \qquad
 \sigma_N^\sharp=c_L\epsilon_N^{4/n},
\]
for all sufficiently large $N$.  Since
$\sigma_N^\sharp=o(\rho_N)$, the cutoff in $h_N^-$ is identically one on
this ball.  Define the localized inner tensor
\begin{equation}\label{eq:inner-principal-tensor-step19}
 h_{N,-}^\sharp(x)
 :=\eta_\sharp\!\left((\sigma_N^\sharp)^{-1}|x-x_N^-|\right)h_N^-(x).
\end{equation}
In the metric-centered variables $x=x_N^-+\epsilon_Nz$ put
\begin{equation}\label{eq:scaled-principal-tensor-step19}
 \widehat h_N(z):=A_N^{-1}h_{N,-}^\sharp(x_N^-+\epsilon_Nz),
 \qquad
 \widehat h(z):=f(|z|^2)H(z).
\end{equation}
Then
\begin{equation}\label{eq:scaled-principal-tensor-convergence-step19}
 \widehat h_N(z)
 =\eta_\sharp\!\left(\frac{\epsilon_N|z|}{\sigma_N^\sharp}\right)
   \widehat h(z),
 \qquad
 \widehat h_N\longrightarrow\widehat h
 \quad\hbox{in }C^3_{\rm loc}(\mathbb R^n),
\end{equation}
with the same conclusion after one normalized parameter derivative.

For $U=U_{\lambda',\xi'}$, define the finite-$N$ and limiting first
metric sources by
\begin{align}
 \mathcal R_{N,\xi',\lambda'}
 &:=\left.\frac d{dt}\right|_{t=0}
   \mathcal N_{e^{t\widehat h_N}\delta}(U),
 \label{eq:finite-principal-metric-source-step19}\\
 \mathcal R_{\xi',\lambda'}
 &:=\left.\frac d{dt}\right|_{t=0}
   \mathcal N_{e^{t\widehat h}\delta}(U).
 \label{eq:principal-metric-source-step19}
\end{align}
Let
\begin{equation}\label{eq:relative-bubble-operator-step19}
 \mathscr L_{\lambda',\xi'}W
 :=U^{-p_*}\left.\frac d{dt}\right|_{t=0}
   \mathcal N_\delta(Ue^{atW}).
\end{equation}
Transport the localized functionals and source functions of
Proposition~\ref{lem:step12} to the round bubble $U_{\lambda',\xi'}$.  Since
$\widehat h_N$ is compactly supported, Proposition~\ref{lem:step12} gives a unique
smooth projected response $W_{N,\xi',\lambda'}^\sharp$ satisfying
\begin{equation}\label{eq:finite-principal-response-problem-step19}
 \mathscr L_{\lambda',\xi'}W_{N,\xi',\lambda'}^\sharp
 =-U^{-p_*}\mathcal R_{N,\xi',\lambda'}
   +\sum_{a=0}^n\beta_{N,a}\Psi_a^{\lambda',\xi'},
 \qquad
 \ell_b^{\lambda',\xi'}(W_{N,\xi',\lambda'}^\sharp)=0.
\end{equation}
The \emph{limiting principal projected response}
$W_{\xi',\lambda'}$ is the limit of these solutions, equivalently the limit
obtained after replacing $\widehat h_N$ by a standard cutoff
$\chi_R\widehat h$ and sending $R\to\infty$.  It satisfies
\begin{equation}\label{eq:principal-response-problem-step19}
 \mathscr L_{\lambda',\xi'}W_{\xi',\lambda'}
 =-U^{-p_*}\mathcal R_{\xi',\lambda'}
   +\sum_{a=0}^n\beta_a\Psi_a^{\lambda',\xi'},
 \qquad
 \ell_b^{\lambda',\xi'}(W_{\xi',\lambda'})=0.
\end{equation}
The strict inequality $n>4m_f+8$ guarantees that the limit and all
pairings below exist.  The limiting source, projected response, coefficient
vector, and functional depend $C^2$ on $(\xi',\lambda')$ on the normalized
parameter set.  The finite-$N$ convergence and energy expansion will only be
used with one normalized parameter derivative.  Indeed, after up to two
parameter derivatives the largest radial density in the quadratic metric
terms remains $O(r^{4m_f-n+7})\,dr$ at infinity.

For a compactly supported tensor $k$, set
\[
 \mathscr A_{\rm geom}(k;\xi',\lambda')
 :=\frac12\left.\frac{d^2}{dt^2}\right|_{t=0}
   \mathscr E_{e^{tk}\delta}(U_{\lambda',\xi'}).
\]
The preceding integrability gives the cutoff-independent limit
\begin{equation}\label{eq:Ageom-step19}
 \mathscr A_{\rm geom}(\xi',\lambda')
 :=\lim_{R\to\infty}
   \mathscr A_{\rm geom}(\chi_R\widehat h;\xi',\lambda'),
\end{equation}
where $\chi_R$ is any standard radial cutoff.  Define the relative bubble
bilinear form
\begin{equation}\label{eq:bubble-bilinear-form-step19}
 \mathscr B_{\lambda',\xi'}(W_1,W_2)
 :=4a\int_{\mathbb R^n}
 U^{p_*}(\mathscr L_{\lambda',\xi'}W_1)W_2\,dz.
\end{equation}
It is symmetric.  Define
\begin{align}
 \mathcal F(\xi',\lambda')
 :={}&\mathscr A_{\rm geom}(\xi',\lambda')
 +4a\int_{\mathbb R^n}
       \mathcal R_{\xi',\lambda'}W_{\xi',\lambda'}\,dz
 \nonumber\\
 &\quad+\frac12\mathscr B_{\lambda',\xi'}
       (W_{\xi',\lambda'},W_{\xi',\lambda'}).
 \label{eq:local-functional-step19}
\end{align}
Since the localized projection sources represent the functionals
$\ell_b^{\lambda',\xi'}$, equation
\eqref{eq:principal-response-problem-step19} gives the equivalent formula
\begin{equation}\label{eq:local-functional-simplified-step19}
 \mathcal F(\xi',\lambda')
 =\mathscr A_{\rm geom}(\xi',\lambda')
  +2a\int_{\mathbb R^n}
       \mathcal R_{\xi',\lambda'}W_{\xi',\lambda'}\,dz.
\end{equation}

Let $\mathscr E_{N,-}^{\rm loc}$ be the normalized energy of the local
metric
\begin{equation}\label{eq:left-local-total-metric-step19}
 \left[
 U_{\epsilon_N\lambda',\,x_N^-+\epsilon_N\xi'}(x)
 \exp\!\left(
   aA_NW_{N,\xi',\lambda'}^\sharp
   \!\left(\frac{x-x_N^-}{\epsilon_N}\right)
 \right)
 \right]^{\frac8{n-4}}
 e^{h_{N,-}^\sharp(x)}\delta.
\end{equation}
Let $\mathscr E_{N,+}^{\rm loc}$ be the normalized energy of its exact
antipodal image.

\begin{lemma}[One-ended local energy expansion]\label{lem:step19}
Uniformly for
\[
 \frac12\le\lambda'\le\frac32,
 \qquad |\xi'|\le1,
\]
one has
\begin{equation}\label{eq:left-local-energy-expansion-step19}
 \mathscr E_{N,-}^{\rm loc}(\xi',\lambda')
 =\mathscr E_2(\mathbb S^n)
  +s_N\mathcal F(\xi',\lambda')+o_{C^1}(s_N).
\end{equation}
The right local energy satisfies the same expansion
\begin{equation}\label{eq:right-local-energy-expansion-step19}
 \mathscr E_{N,+}^{\rm loc}(\xi',\lambda')
 =\mathscr E_2(\mathbb S^n)
  +s_N\mathcal F(\xi',\lambda')+o_{C^1}(s_N).
\end{equation}
In the right Kelvin chart the model coincides with the left model.  In
physical right bubble coordinates this identification conjugates the Weyl
tensor and the translation variables by the orthogonal differential of the
antipodal map.
\end{lemma}

\begin{proof}
We divide the proof into four steps.

\emph{Step 1: well-posedness of the limiting one-ended model.}
Let $\widehat h_R=\chi_R\widehat h$.  For every fixed $R$ the source obtained
from \eqref{eq:principal-metric-source-step19} with $\widehat h_R$ is smooth
on the compactified round bubble, so Proposition~\ref{lem:step12} gives a unique
projected response $W_R$.  The coordinate expression of the first metric
variation contains only $\widehat h_R$, $D\widehat h_R$, and
$D^2\widehat h_R$, multiplied by derivatives and powers of the round bubble.
The second metric variation is quadratic in the same quantities.  Since
\[
 |D^j\widehat h(z)|\le C(1+|z|)^{2m_f+2-j},
 \qquad 0\le j\le3,
\]
the worst radial density in the quadratic terms is
\[
 C r^{4m_f-n+7}\,dr.
\]
It is integrable precisely under $n>4m_f+8$.  The round Green kernel on the
orthogonal complement of the first harmonics then shows that $W_R$ is Cauchy
in $C^{2,\alpha}_{\rm loc}$ and in the associated quadratic-form norm.
For the limiting cutoff family the same estimates apply after up to two
derivatives in $(\xi',\lambda')$, uniformly on the compact normalized
parameter set: differentiating the projected equation twice produces the
same round augmented operator, while the differentiated sources have the
same or better far-field powers.  Thus the limiting response and its
coefficient vector are $C^2$.  For the finite-$N$ family we keep the
uniform bounds with one normalized parameter derivative, which are exactly
what is needed for the $C^1$ reduced-energy expansion.  The particular
cutoffs in \eqref{eq:scaled-principal-tensor-step19} have radii
$\sigma_N^\sharp/\epsilon_N\to\infty$, so the same argument gives
\begin{equation}\label{eq:finite-response-convergence-step19}
 W_{N,\xi',\lambda'}^\sharp\longrightarrow W_{\xi',\lambda'}
 \quad\hbox{in }C^{2,\alpha}_{\rm loc}
 \quad\hbox{and in the quadratic-form norm},
\end{equation}
with one normalized parameter derivative.  Hence the limits in
\eqref{eq:Ageom-step19} and \eqref{eq:local-functional-step19} exist.  The
preceding twice-differentiated round estimates and dominated convergence
show that $\mathcal F\in C^2$ on the normalized parameter set.  The
finite-$N$ convergence used in the expansion is $C^1$.

By \eqref{eq:principal-response-problem-step19}, symmetry of
$\mathscr B_{\lambda',\xi'}$, and the identities
\[
 \int U^{p_*}\Psi_a^{\lambda',\xi'}
       W_{\xi',\lambda'}\,dz
 =\ell_a^{\lambda',\xi'}(W_{\xi',\lambda'})=0,
\]
we have
\[
 \mathscr B_{\lambda',\xi'}(W,W)
 =-4a\int\mathcal R_{\xi',\lambda'}W\,dz.
\]
This proves \eqref{eq:local-functional-simplified-step19}.

\emph{Step 2: exact reduction to the metric-scale model.}
Use the affine change of variables
$x=x_N^-+\epsilon_Nz$.  The bubble scaling gives
\[
 U_{\epsilon_N\lambda',\,x_N^-+\epsilon_N\xi'}
      (x_N^-+\epsilon_Nz)
 =\epsilon_N^{-a}U_{\lambda',\xi'}(z),
\]
while $\delta_x=\epsilon_N^2\delta_z$.  Since
$8a/(n-4)=2$, the two powers of $\epsilon_N$ cancel in the total metric.
Consequently, diffeomorphism invariance of the energy gives the exact identity
\begin{equation}\label{eq:scaled-local-energy-identity-step19}
 \mathscr E_{N,-}^{\rm loc}
 =\mathscr E_{e^{A_N\widehat h_N}\delta}
   \left(U_{\lambda',\xi'}e^{aA_NW_{N,\xi',\lambda'}^\sharp}\right).
\end{equation}
Thus the small parameter in the local problem is exactly $A_N$, and its
square is $s_N$.

\emph{Step 3: the second-order Taylor expansion.}
The round total metric
$U_{\lambda',\xi'}^{8/(n-4)}\delta$ is a critical point of the normalized
metric energy.  If $k$ is an arbitrary variation of the round metric $g_*$, naturality and the Einstein symmetry imply
\[
 \left.\frac d{dt}\right|_{0}
 \int\sigma_2(g_t^{-1}A_{g_t})\,dv_{g_t}
 =\frac{(n-4)}{8n}\binom n2
   \int\operatorname{tr}_{g_*}k\,dv_{g_*}.
\]
The derivative of the volume term in
\eqref{eq:global-energy-step18} is the same quantity, so the first variation
of the normalized energy vanishes.

Apply Taylor's formula to the right-hand side of
\eqref{eq:scaled-local-energy-identity-step19} at $A_N=0$.  The pure metric
part of the quadratic coefficient is
$\mathscr A_{\rm geom}(\widehat h_N;\xi',\lambda')$.  The mixed metric--
conformal term is
\[
 4a\int\mathcal R_{N,\xi',\lambda'}
          W_{N,\xi',\lambda'}^\sharp\,dz,
\]
where $\mathcal R_{N,\xi',\lambda'}$ is the first metric source generated by
$\widehat h_N$.  For brevity in this step write
$W_N^\sharp=W_{N,\xi',\lambda'}^\sharp$ and
$W=W_{\xi',\lambda'}$.  The pure conformal quadratic term is
$\frac12\mathscr B(W_N^\sharp,W_N^\sharp)$.  Hence
\begin{align}
 \mathscr E_{N,-}^{\rm loc}
 ={}&\mathscr E_2(\mathbb S^n)
 +A_N^2\Bigg[
   \mathscr A_{\rm geom}(\widehat h_N;\xi',\lambda')
   +4a\int\mathcal R_{N,\xi',\lambda'}W_N^\sharp\,dz
 \nonumber\\
 &\hspace{38mm}
   +\frac12\mathscr B_{\lambda',\xi'}(W_N^\sharp,W_N^\sharp)
 \Bigg]+o_{C^1}(A_N^2).
 \label{eq:finite-N-local-Taylor-step19}
\end{align}

We justify the uniform remainder and the passage to the limiting coefficient.
The projected Green estimate used in Step~1 also gives, on the support of
$\widehat h_N$,
\[
 A_N\sum_{j=0}^2(1+|z|)^j
 \bigl(|D^j\widehat h_N(z)|+|D^jW_N^\sharp(z)|\bigr)=o(1).
\]
Outside that support $W_N^\sharp$ is smooth across the compactified north-pole
cap, and the same smallness holds in the round intrinsic norm.  Therefore
every cubic Taylor term is bounded by this $o(1)$ factor times the integrable
quadratic density from Step~1.  This gives the
$o_{C^1}(A_N^2)$ remainder in
\eqref{eq:finite-N-local-Taylor-step19}.  Moreover,
\eqref{eq:scaled-principal-tensor-convergence-step19} and the same integrable
dominating density imply
\begin{align*}
 \mathscr A_{\rm geom}(\widehat h_N;\xi',\lambda')
 &\longrightarrow\mathscr A_{\rm geom}(\xi',\lambda'),\\
 \int\mathcal R_{N,\xi',\lambda'}W_N^\sharp\,dz
 &\longrightarrow\int\mathcal R_{\xi',\lambda'}W\,dz,\\
 \mathscr B(W_N^\sharp,W_N^\sharp)
 &\longrightarrow\mathscr B(W,W).
\end{align*}
in $C^1$ on the normalized parameter set.  Combining these limits with
\eqref{eq:local-functional-step19},
$A_N^2=s_N$, and \eqref{eq:finite-N-local-Taylor-step19} proves
\eqref{eq:left-local-energy-expansion-step19}.

\emph{Step 4: the right end.}
The right local metric is the exact $J$-image of the left local metric.  The
normalized energy is invariant under diffeomorphisms, so
\[
 \mathscr E_{N,+}^{\rm loc}=\mathscr E_{N,-}^{\rm loc}
\]
exactly.  In the coordinate $z=\iota_A(x)$, both the scalar factor and the
metric endomorphism are the left objects from
\eqref{eq:left-local-total-metric-step19}.  If instead one uses the physical
right Euclidean bubble coordinate, the differential of $\iota_A$ contributes
an orthogonal reflection to the translation variables and conjugates the
Weyl tensor by the same orthogonal map.  All contractions in
\eqref{eq:local-functional-step19} are invariant under this simultaneous
conjugation.  Thus the coefficient is the same $\mathcal F$, proving
\eqref{eq:right-local-energy-expansion-step19} and completing the proof.
\end{proof}

\medskip
\noindent\textbf{Local extraction of the projected coefficients.}
Choose a compact set $K_\Psi\Subset\mathbb R^n$ containing, in the
metric-centered left bubble coordinate, the supports of all intrinsic
pullbacks of the paired sources $\Psi_{a,N}$.  If
$(u_N,\mathbf d_N)=\mathbb G_NF_N$ and
\[
 f_{N,-}(z)=U_{\lambda',\xi'}(z)^{-p_*}\epsilon_N^n
 F_{E,N}(x_N^-+\epsilon_Nz),
\]
then
\begin{equation}\label{eq:local-coefficient-extraction-step20}
 \sum_{a=0}^n|d_{N,a}|
 \le C\left(
 \|u_N(x_N^-+\epsilon_N\,\cdot)\|_{C^{2,\alpha}(K_\Psi)}
 +\|f_{N,-}\|_{C^\alpha(K_\Psi)}\right).
\end{equation}
The same estimate holds in the right Kelvin chart and after one normalized
parameter derivative, with the corresponding differentiated and
undifferentiated norms on the right-hand side.  More generally, if
\begin{equation}\label{eq:comparison-local-equation-step20}
 \mathscr P_{N,-}U_N
 =H_N+\sum_{a=0}^n\gamma_{N,a}\Psi_{a,N}^{-}
 \quad\hbox{on }K_\Psi,
\end{equation}
then
\begin{equation}\label{eq:comparison-coefficient-extraction-step20}
 \sum_{a=0}^n|\gamma_{N,a}|
 \le C\left(\|U_N\|_{C^{2,\alpha}(K_\Psi)}
 +\|H_N\|_{C^\alpha(K_\Psi)}\right).
\end{equation}
Indeed, the coefficients of $\mathscr P_{N,-}$ are uniformly bounded on
$K_\Psi$, while the finite family $\{\Psi_{a,N}^{-}\}$ is uniformly
linearly independent by the bubble Gram matrix.  Thus
$\sum_a|\gamma_{N,a}|$ is controlled by the $C^\alpha$ norm of
$\sum_a\gamma_{N,a}\Psi_{a,N}^{-}$, and the displayed estimates follow by
substituting the local equation.  Differentiation and Kelvin conjugation give
the remaining assertions.

Let $R_{b,N}\asymp_L\lambda^{4/n-1}$ be the normalized radius of the exact
left bubble portion, put $T_{b,N}=\log(1+R_{b,N})$, and set
\begin{equation}\label{eq:bubble-transfer-weight-rigorous}
 \omega_{b,N}(r)
 :=\lambda^{a_\Sigma}
   +\min\left\{1,\left(\frac{1+r}{R_{b,N}}\right)^{n/2}\right\}.
\end{equation}
For a function on the bubble chart define
\[
 \|u\|_{\mathcal X_{b,N}}
 :=\sup_{0\le \tau\le T_{b,N}}
  \omega_{b,N}(\tau)^{-1}X_u(\tau),
\]
where $\tau=\log(1+r)$, where by abuse of notation
$\omega_{b,N}(\tau):=\omega_{b,N}(e^\tau-1)$, and where $X_u(\tau)$ is the
intrinsic round $C^{2,\alpha}$ annular norm on the unit annulus centered at
$\tau$.  For an intrinsic bubble source $H$ define
\begin{equation}\label{eq:weighted-bubble-source-norm-step20}
 \|H\|_{\mathcal Y_{b,N}}
 :=\sup_{0\le \tau\le T_{b,N}}
  \omega_{b,N}(\tau)^{-1}Y_H(\tau),
\end{equation}
where $Y_H(\tau)$ is the corresponding intrinsic $C^\alpha$ source norm.
When $H$ is given only on the truncated bubble portion, fix a uniformly
bounded extension through a collar of fixed round width and let $G_*H$
denote the projected round response to that extension.  Different choices
of extension differ by an outer-collar source and are included in the
boundary-field estimate below.

\begin{proposition}[One-sided weighted stability of the principal bubble resolvent]
\label{aux:weighted-bubble-resolvent-rigorous}
The projected round inverse satisfies
\begin{equation}\label{eq:weighted-round-bubble-inverse-step20}
 \|G_*H\|_{\mathcal X_{b,N}}
 \le C\|H\|_{\mathcal Y_{b,N}}.
\end{equation}

Let $\mathscr K_{b,N}$ be the actual-minus-round coefficient perturbation in
the principal bubble.  Then
\begin{equation}\label{eq:weighted-bubble-perturbation-rigorous}
 \|G_*[\mathscr K_{b,N}u]\|_{\mathcal X_{b,N}}
 \le \eta_{b,N}\|u\|_{\mathcal X_{b,N}},
 \qquad \eta_{b,N}\longrightarrow0,
\end{equation}
uniformly in the normalized parameters and after one normalized parameter
derivative.

Finally, write a global source as $F=F_{b,-}+F_{\rm out}$, where $F_{b,-}$
is its intrinsic restriction to the exact left bubble portion and
$F_{\rm out}$ is supported in its complement.  If
$(u,\mathbf d)=\mathbb G_NF$ is the actual global augmented response, then
\begin{equation}\label{eq:weighted-global-bubble-inverse-step20}
 \|u(x_N^-+\epsilon_N\,\cdot)\|_{\mathcal X_{b,N}}
 \le C\left(
   \|F_{b,-}\|_{\mathcal Y_{b,N}}
   +\|F_{\rm out}\|_{\widehat Y_{\gamma,\beta}}
 \right).
\end{equation}
The same estimate holds in the exact right Kelvin chart and after one
normalized parameter derivative.  In particular, a source supported outside
the exact bubble portion produces $O(\lambda^{a_\Sigma})$ on every fixed
compact bubble set.
\end{proposition}

\begin{proof}
Put $\tau=\log(1+r)$.  The projected round Green operator has the one-sided
annular estimate
\begin{align}\label{eq:one-sided-round-green-rigorous}
 X_{G_*H}(\tau)
 \le C\biggl(&\int_0^\tau Y_H(s)\,ds
 +\int_\tau^{T_{b,N}}e^{-(n-2)(s-\tau)}Y_H(s)\,ds\biggr).
\end{align}
The first kernel is only bounded: an interior source may generate a regular
component which approaches a nonzero value at the second round pole.  The
second kernel has the radial singular-to-regular decay $n-2$, and all
nonradial sectors decay faster.  Formula
\eqref{eq:one-sided-round-green-rigorous} follows mode by mode from the
regular and singular round solutions in
\eqref{eq:round-regular-low-modes-rigorous}--
\eqref{eq:round-singular-low-modes-rigorous}; the degree-one regular
components are removed by the localized constraints.  Local Schauder
estimates restore two derivatives and absorb the diagonal singularity.

Since $R_{b,N}\asymp_L\lambda^{4/n-1}$,
\[
 \lambda^{a_\Sigma}\asymp_L e^{-nT_{b,N}/2}.
\]
Consequently, with $\omega(\tau)=\omega_{b,N}(e^\tau-1)$,
\begin{equation}\label{eq:bubble-weight-convolution-step20}
 \int_0^\tau\omega(s)\,ds\le C\omega(\tau),
 \qquad
 \int_\tau^{T_{b,N}}e^{-(n-2)(s-\tau)}\omega(s)\,ds
 \le C\omega(\tau).
\end{equation}
For the first inequality, factor out $e^{-nT_{b,N}/2}$ and use the boundedness
of $\tau/(1+e^{n\tau/2})$; the second follows from
$n-2-n/2=a_\Sigma>0$.  Substitution of
$Y_H(s)\le\omega(s)\|H\|_{\mathcal Y_{b,N}}$ into
\eqref{eq:one-sided-round-green-rigorous} proves
\eqref{eq:weighted-round-bubble-inverse-step20}.

In normalized bubble coordinates Lemma~\ref{lem:step8} gives
\begin{equation}\label{eq:Xi-weighted-rigorous}
 |\mathscr K_{b,N}|_{\rm coef}
 \le k_N(\tau):=C A_N(1+e^\tau)^{2m_f+4},
 \qquad 0\le \tau\le T_{b,N},
\end{equation}
and the same majorant for one normalized parameter derivative.  Since
$R_{b,N}\asymp_L\epsilon_N^{4/n-1}$,
\[
 A_NR_{b,N}^{2m_f+4}
 \asymp_L\epsilon_N^{\beta_f-2+8(m_f+2)/n}
 =\epsilon_N^{\kappa_f}.
\]
Thus
\begin{equation}\label{eq:integrated-principal-coefficient-rigorous}
 \int_0^{T_{b,N}}k_N(s)\,ds+\sup_sk_N(s)
 \le C\bigl(A_NT_{b,N}+\epsilon_N^{\kappa_f}\bigr)
 =:\eta_{b,N}\longrightarrow0.
\end{equation}
Apply \eqref{eq:one-sided-round-green-rigorous} to
$H=\mathscr K_{b,N}u$.  In the first integral use monotonicity of the
weight and \eqref{eq:integrated-principal-coefficient-rigorous}; in the second
use
$\omega(s)/\omega(\tau)\le Ce^{n(s-\tau)/2}$.  The remaining kernel is
$e^{-a_\Sigma(s-\tau)}$.  This proves
\eqref{eq:weighted-bubble-perturbation-rigorous}; differentiating gives the
same estimate because $|D\omega_{b,N}|\le C\omega_{b,N}$ in the natural
parameter coordinates.

It remains to pass from the complete round inverse to the global augmented
inverse.  For the flat paired operator, a bubble-supported source is treated
by \eqref{eq:weighted-round-bubble-inverse-step20}.  The difference between
that complete-sphere particular solution and the bubble component of the
global flat solution is a boundary field.  Its low-mode Cauchy coefficients
are bounded by the uniformly invertible matching matrix
\eqref{eq:complete-low-matching-rigorous}, and its high modes by the
transition spectral gap.  The one-sided Green estimate then bounds this
boundary field in $\mathcal X_{b,N}$.  A source supported outside the bubble
enters only through the outer Cauchy data; the same matching matrix and the
transition--Schwarzschild Green operators give an
$O(\|F_{\rm out}\|_{\widehat Y_{\gamma,\beta}})$ bound in
$\mathcal X_{b,N}$.  Finally, the actual-minus-flat coefficient perturbation
is absorbed by \eqref{eq:weighted-bubble-perturbation-rigorous} together
with Proposition~\ref{prop:summable-operator-perturbation-step15}.  This
proves \eqref{eq:weighted-global-bubble-inverse-step20}.  Exact Kelvin
conjugation gives the right-end statement, and differentiation uses the same
matching system and the differentiated perturbation estimates.
\end{proof}

\begin{proposition}[Localized transfer from the complement of a principal core]
\label{prop:localized-bubble-transfer-step20}
Let $\mathbb G_N$ denote the projected inverse of
Lemma~\ref{lem:step15}, in the refined form of
Proposition~\ref{prop:refined-transfer-step16}.  Let $F_N$ be a
$J$-invariant source which vanishes on
$B_{\sigma_N^\sharp}(x_N^-)$, and let
$(u_N,\mathbf d_N)=\mathbb G_NF_N$.  Then, for every compact
$K\Subset\mathbb R^n$,
\begin{equation}\label{eq:localized-transfer-step20}
 \left\|u_N(x_N^-+\epsilon_Nz)\right\|_{C^{2,\alpha}(K)}
 \le C_K\lambda^{a_\Sigma}
       \|F_N\|_{\widehat Y_{\gamma,\beta}},
\end{equation}
provided the chosen left core is omitted from the source.  The same estimate
holds in the exact right Kelvin chart, and after one normalized parameter
derivative.  If $K$ contains the fixed supports of the localized projection
sources, then in addition
\begin{equation}\label{eq:localized-coefficient-gain-repair}
 \sum_{a=0}^n|d_{N,a}|
 \le C\lambda^{a_\Sigma}\|F_N\|_{\widehat Y_{\gamma,\beta}}.
\end{equation}
More generally, if a source is supported in the left bubble block but
outside $|z|\le R$ in metric-centered variables, its contribution on a
fixed compact $K$ is bounded by the round-bubble Green tail from $|z|>R$;
this tail tends to zero as $R\to\infty$, uniformly for the sources arising
from the principal metric tensor under $n>4m_f+8$.
\end{proposition}

\begin{proof}
We first prove the estimate for the flat augmented inverse.  Let
$r_b\asymp_L\lambda^{4/n}$ be the bubble-side radius of the left transition,
and write $t_b=\log r_b$ on the adjoining Schwarzschild half-neck.  After
subtracting zero-boundary particular solutions, the low-mode interface
coefficients are determined by the uniformly invertible matrix
\eqref{eq:complete-low-matching-rigorous}.  Hence no neutral radial
coefficient remains.

Consider a normalized source block centered at $t_s>t_b$ on the left
Schwarzschild half-neck.  The kernel in Lemma~\ref{lem:step13} gives, at the
bubble-side endpoint,
\begin{equation}\label{eq:flat-offdiag-kernel-rigorous}
 |w_{\ell,m}(t_b)|+|\partial_tw_{\ell,m}(t_b)|
 \le C\int_{t_b}^0e^{-\sigma_\ell(s-t_b)}
       q(s)|f_{\ell,m}(s)|\,ds,
\end{equation}
where
$\sigma_0=\nu_0+a=a_\Sigma$,
$\sigma_1=\nu_1+a=a_\Sigma+1$, and
$\sigma_\ell>a_\Sigma$ for $\ell\ge2$.  The matching matrix changes the
constant in this estimate but not the exponent.  If the source is supported
on a unit block at radius $r_s$, then
\begin{equation}\label{eq:neck-to-interface-rigorous}
 |w(t_b)|+|\partial_tw(t_b)|
 \le Cq(r_s)\left(\frac{r_b}{r_s}\right)^{a_\Sigma}
       \mathfrak Y_{Q_s}(F_N).
\end{equation}

The boundary-induced solution on the exact bubble is controlled by the
singular round branch.  Lemma~\ref{aux:bubble-dtn-rigorous} and
\eqref{eq:kappa-low-rigorous} imply, for every fixed compact bubble set $K$,
\begin{equation}\label{eq:interface-to-bubble-rigorous}
 \|w\|_{C^{2,\alpha}(K)}
 \le C_K\left(\frac{\lambda}{r_b}\right)^{a_\Sigma}
   \bigl(|w(t_b)|+|\partial_tw(t_b)|\bigr).
\end{equation}
Indeed, the exact low-mode decay is
$(\lambda/r_b)^{n-2}$ in degree zero and
$(\lambda/r_b)^{n-1}$ in degree one, both stronger than the displayed bound;
the high modes are stronger still.  Combining
\eqref{eq:neck-to-interface-rigorous} and
\eqref{eq:interface-to-bubble-rigorous} gives
\[
 Cq(r_s)\left(\frac{\lambda}{r_s}\right)^{a_\Sigma}
 \mathfrak Y_{Q_s}(F_N)
 \le C\lambda^{a_\Sigma}\mathfrak Y_{Q_s}(F_N),
\]
because $q(r_s)\le Cr_s^{a_\Sigma}$ on the left half-neck.  The low modes
sum by the $L_t^1$ term in $\widehat Y_{\gamma,\beta}$, and the high modes
sum by the weighted supremum.

A source in the transition has zero-boundary particular response and endpoint
flux bounded by $C_L\delta_\lambda\|f\|_{C^\alpha}$ by
Proposition~\ref{lem:step14} and Lemma~\ref{aux:transition-dtn-rigorous}.  Equation
\eqref{eq:interface-to-bubble-rigorous} therefore gives
\[
 C_L\delta_\lambda
 \left(\frac{\lambda}{\lambda^{4/n}}\right)^{a_\Sigma}
 \|f\|_{C^\alpha}
 =C_L\lambda^{a_\Sigma}\|f\|_{C^\alpha}.
\]
It remains to cover sources lying in the exact bubble portion but outside
the omitted metric-centered core.  Put
\[
 R_{\sharp,N}:=\frac{\sigma_N^\sharp}{\epsilon_N},
 \qquad
 R_{b,N}:=\frac{r_b}{\epsilon_N}.
\]
Because $\lambda/\epsilon_N$ stays in a fixed compact subset of
$(0,\infty)$ and both radii are fixed multiples, depending only on $L$, of
$\epsilon_N^{4/n-1}$, one has
$R_{\sharp,N}\asymp_L R_{b,N}$.  Hence this source is confined to a
cylindrical interval of $L$-dependent but $N$-independent length adjacent to
the outer bubble interface.  If $\tau_K$ is a fixed annular coordinate
containing a given compact set $K$, the one-sided round Green estimate
\eqref{eq:one-sided-round-green-rigorous} gives
\begin{align}
 \|\mathbf 1_KG_*\mathbf 1_{\{R_{\sharp,N}<|z|<R_{b,N}\}}F_N
   \|_{C^{2,\alpha}}
 &\le C_{K,L}R_{\sharp,N}^{-(n-2)}
       \|F_N\|_{\widehat Y_{\gamma,\beta}}.\label{eq:outer-bubble-annulus-transfer}
\end{align}
Since $n-2>n/2$ and
$R_{b,N}^{-n/2}\asymp_L\lambda^{a_\Sigma}$, the right-hand side of
\eqref{eq:outer-bubble-annulus-transfer} is bounded by
$C_{K,L}\lambda^{a_\Sigma}\|F_N\|_{\widehat Y_{\gamma,\beta}}$.
Thus the flat localized estimate covers every source allowed by the
hypothesis, including the exact-bubble outer annulus.

Sources on the fixed connector or on the opposite end first enter a
Schwarzschild half-neck through a bounded fixed-block Dirichlet-to-Neumann
map; applying the same half-neck kernel and the same matching matrix gives
the same factor $\lambda^{a_\Sigma}$.  We have proved
\begin{equation}\label{eq:flat-localized-resolvent-rigorous}
 \|\mathbf1_K\mathbb G_N^{\rm fl}\mathbf1_{\rm remote}F_N\|_{C^{2,\alpha}}
 \le C_K\lambda^{a_\Sigma}
       \|F_N\|_{\widehat Y_{\gamma,\beta}}.
\end{equation}
The coefficient vector obeys the same estimate by the local coefficient estimate above.  Differentiating the Green
kernels and the matching matrix changes only the constant, since
$\lambda\partial_\lambda\lambda^{a_\Sigma}
 =a_\Sigma\lambda^{a_\Sigma}$.

We now pass to the actual operator.  Put
$\mathscr K_N=\mathscr L_N-\mathscr L_N^{\rm fl}$ and use the augmented
resolvent identity
\[
 (u_N,\mathbf d_N)
 =\mathbb G_N^{\rm fl}\bigl(F_N-\mathscr K_Nu_N\bigr).
\]
Pull the bubble component back to the normalized round chart and measure it
in the weighted norm of Proposition~\ref{aux:weighted-bubble-resolvent-rigorous}.  The flat remote response obeys
\eqref{eq:weighted-global-bubble-inverse-step20}, while
\eqref{eq:weighted-bubble-perturbation-rigorous} gives
\[
 \|u_N\|_{\mathcal X_{b,N}}
 \le C\|F_N\|_{\widehat Y_{\gamma,\beta}}
   +o_N(1)\|u_N\|_{\mathcal X_{b,N}}.
\]
The perturbations outside the principal bubble are remote and are already
included in the first term by the flat off-diagonal estimate and Proposition~\ref{prop:summable-operator-perturbation-step15}.  Absorbing the last term
shows that
$\|u_N\|_{\mathcal X_{b,N}}\le C\|F_N\|_{\widehat Y_{\gamma,\beta}}$.
On every fixed compact set the weight
\eqref{eq:bubble-transfer-weight-rigorous} is bounded by
$C_K\lambda^{a_\Sigma}$, proving
\eqref{eq:localized-transfer-step20}.  Since the intrinsic source is zero on $K_\Psi$, the local coefficient estimate above proves
\eqref{eq:localized-coefficient-gain-repair}.  The differentiated weighted
resolvent estimate gives the parameter version, and exact Kelvin conjugation
gives the right-end statement.

Finally, for a source contained in the bubble block but outside $|z|\le R$,
subtract the projected round responses to the full source and to its
truncation at radius $R$.  The weighted round Green estimate from
Lemma~\ref{lem:step19} shows that the difference tends to zero on each fixed
compact set; the common absolute density is
$r^{4m_f-n+7}\,dr$.  This proves the final assertion, including one
normalized parameter derivative.
\end{proof}

\begin{lemma}[Localization of the principal response]
\label{lem:step20}
Use the metric-centered normalized coordinates
\begin{equation}\label{eq:metric-centered-charts-step20}
 x=x_N^-+\epsilon_Nz
 \quad\hbox{on the left},
 \qquad
 \iota_A(x)=x_N^-+\epsilon_Nz
 \quad\hbox{on the right}.
\end{equation}
Then, uniformly for
$1/2\le\lambda'\le3/2$ and $|\xi'|\le1$,
\begin{align}
 A_N^{-1}w_{N,\lambda,\xi}(x_N^-+\epsilon_Nz)
 &\longrightarrow W_{\xi',\lambda'}(z),
 \label{eq:left-localization-step20}\\
 A_N^{-1}w_{N,\lambda,\xi}(x),\quad
 \iota_A(x)=x_N^-+\epsilon_Nz
 &\longrightarrow W_{\xi',\lambda'}(z),
 \label{eq:right-localization-step20}
\end{align}
in $C^{2,\alpha}_{\rm loc}(\mathbb R^n)$.  In the bubble-centered variable
$y$, where $z=\xi'+\lambda'y$, the left limit is equivalently
$W_{\xi',\lambda'}(\xi'+\lambda'y)$; in physical right Euclidean variables
it is the exact Kelvin image.  Moreover,
\begin{equation}\label{eq:coefficient-localization-step20}
 A_N^{-1}c_a(N,\lambda,\xi)\longrightarrow\beta_a(\xi',\lambda'),
 \qquad 0\le a\le n,
\end{equation}
where the $\beta_a$ are the coefficients in
\eqref{eq:principal-response-problem-step19}.  The convergences
\eqref{eq:left-localization-step20}--\eqref{eq:coefficient-localization-step20}
hold with one normalized parameter derivative.

After subtracting the paired finite principal response, the contribution on
either fixed bubble chart of the opposite principal core, the principal outer tail,
the remaining metric sources, the matched transitions, the exact Schwarzschild
source, the compact connector, and the nonlinear remainder is $o(A_N)$.
\end{lemma}

\begin{proof}
We give the proof on the left; the right statement then follows exactly from
$J$-invariance.

\emph{Step 1: the paired principal linear source.}
Let $P_{N,-}^\sharp$ be the first metric variation generated by
$h_{N,-}^\sharp$ at the exact left bubble, and let
$P_{N,+}^\sharp$ be its exact $J$-image.  In the coordinate
$x=x_N^-+\epsilon_Nz$, conformal covariance and
\eqref{eq:scaled-principal-tensor-step19} give
\begin{equation}\label{eq:principal-source-scaling-step20}
 A_N^{-1}U_{\lambda',\xi'}^{-p_*}
   \bigl(P_{N,-}^\sharp\bigr)^{\rm pull}
 =U_{\lambda',\xi'}^{-p_*}\mathcal R_{N,\xi',\lambda'}.
\end{equation}
Here $\bigl(P_{N,-}^\sharp\bigr)^{\rm pull}$ denotes the intrinsic bubble
pullback, namely $\epsilon_N^nP_{N,-}^\sharp(x_N^-+\epsilon_Nz)$.
Let $(w_N^\sharp,\mathbf b_N)$ be the global projected linear solution with
source $-(P_{N,-}^\sharp+P_{N,+}^\sharp)$.
By Proposition~\ref{lem:step12},
Proposition~\ref{prop:localized-bubble-transfer-step20}, and the one-sided
bubble boundary estimate in Proposition~\ref{aux:weighted-bubble-resolvent-rigorous},
\begin{equation}\label{eq:finite-principal-global-localization-step20}
 A_N^{-1}w_N^\sharp(x_N^-+\epsilon_Nz)
 -W_{N,\xi',\lambda'}^\sharp(z)
 \longrightarrow0
 \quad\text{in }C^{2,\alpha}_{\rm loc}.
\end{equation}
Indeed, the source at the opposite bubble is outside the chosen principal core
and hence contributes $O(A_N\lambda^{a_\Sigma})$ on the left.  The only
remaining difference is a boundary-induced bubble field at the shrinking
bubble--transition interface.  The one-sided weighted bubble estimate shows
that this field tends to zero on every fixed compact subset of the bubble
chart.  The argument uses only this weighted decay into the bubble core,
not decay of the unweighted interface trace.

The coefficient convergence is now extracted from the local equation rather
than inferred from the function estimate.  Put
\[
 E_N^\sharp
 :=A_N^{-1}w_N^\sharp(x_N^-+\epsilon_N\,\cdot)
   -W_{N,\xi',\lambda'}^\sharp.
\]
On $K_\Psi$ this function satisfies
\begin{equation}\label{eq:principal-coefficient-comparison-equation-step20}
 \mathscr P_{N,-}E_N^\sharp
 =H_N^\sharp
 +\sum_{a=0}^n
  \bigl(A_N^{-1}b_{N,a}-\beta_{N,a}\bigr)\Psi_{a,N}^{-},
\end{equation}
where $\|H_N^\sharp\|_{C^\alpha(K_\Psi)}\to0$.  This follows from
\eqref{eq:principal-source-scaling-step20}, the $C^1$ convergence of the
intrinsic bubble operator to $\mathscr L_*$ on fixed compact sets, and the
finite-response equation in Lemma~\ref{lem:step19}.  Applying the local coefficient estimate above to
\eqref{eq:principal-coefficient-comparison-equation-step20} and using
\eqref{eq:finite-principal-global-localization-step20} gives
\begin{equation}\label{eq:finite-principal-coefficients-step20}
 A_N^{-1}b_{N,a}-\beta_{N,a}\longrightarrow0.
\end{equation}
Equations
\eqref{eq:finite-response-convergence-step19},
\eqref{eq:finite-principal-global-localization-step20}, and
\eqref{eq:finite-principal-coefficients-step20} therefore imply
\begin{equation}\label{eq:principal-global-limit-step20}
 A_N^{-1}w_N^\sharp(x_N^-+\epsilon_Nz)
 \longrightarrow W_{\xi',\lambda'}(z),
 \qquad
 A_N^{-1}b_{N,a}\longrightarrow\beta_a.
\end{equation}

\emph{Step 2: all linear nonprincipal sources are negligible locally.}
Write
\[
 R_N=P_{N,-}^\sharp+P_{N,+}^\sharp+R_N^{\rm rem}.
\]
On every fixed metric-centered ball, Taylor expansion of the metric residual
gives
\[
 R_N-P_{N,-}^\sharp=O(A_N^2)
\]
in the intrinsic bubble source norm.  The part of the principal source outside
$h_{N,-}^\sharp$ begins at radius comparable with
$\epsilon_N^{4/n}$; Lemma~\ref{lem:step8} and
Proposition~\ref{prop:localized-bubble-transfer-step20} give a contribution
$o(A_N)$.  The same proposition and Lemma~\ref{lem:step9} give
\begin{equation}\label{eq:remote-linear-localization-step20}
 \left\|\mathbb G_N[-R_N^{\rm rem}]
       \right\|_{C^{2,\alpha}(K_-)}
 \le C_K\lambda^{a_\Sigma}
 \left(\eta_{L,N_0}+C_L\lambda^{2-8/n}
       +C\epsilon_N^{\kappa_f}\right)+o(A_N).
\end{equation}
Here $K_-$ denotes the pullback of a fixed compact set in the left chart.
The exponent comparison
\begin{equation}\label{eq:transfer-versus-principal-step20}
 \frac{\lambda^{a_\Sigma}}{A_N}
 \asymp\epsilon_N^{a_\Sigma-(\beta_f+d_f)}
 =\epsilon_N^{\beta_f}\longrightarrow0
\end{equation}
shows that the right-hand side of
\eqref{eq:remote-linear-localization-step20} is $o(A_N)$.  This includes the
opposite principal core, all remaining perturbations, both transitions, the exact
Schwarzschild target source, and the compact connector.

\emph{Step 3: the nonlinear remainder.}
Set
\[
 d_N=w_{N,\lambda,\xi}-w_N^\sharp,
 \qquad
 \mathbf e_N=(c_0-b_{N,0},\ldots,c_n-b_{N,n}).
\]
Use the complete augmented inverse, with the projected source coefficients
included in the unknown vector.  Thus
\begin{equation}\label{eq:difference-augmented-equation-step20}
 (d_N,\mathbf e_N)
 =\mathbb G_N\!\left[
   -R_N^{\rm rem}-Q_N(w_N^\sharp)
   -\bigl(Q_N(w_{N,\lambda,\xi})-Q_N(w_N^\sharp)\bigr)
 \right].
\end{equation}
We use the one-sided bubble norm of Proposition~\ref{aux:weighted-bubble-resolvent-rigorous}.  Let $F_N^{(0)}$ denote the
sum of the first two sources in brackets in
\eqref{eq:difference-augmented-equation-step20}.  The regionwise estimates
of Step~2, the finite principal-response bounds of Lemma~\ref{lem:step19},
and the local quadratic estimate give
\begin{equation}\label{eq:nonprincipal-weighted-source-step20}
 \|(F_N^{(0)})_{b,-}\|_{\mathcal Y_{b,N}}
 +\|(F_N^{(0)})_{\rm out}\|_{\widehat Y_{\gamma,\beta}}
 \le C\left(\varepsilon_N
       +\lambda^{-a_\Sigma}A_N^2\right).
\end{equation}
We spell out the weighted estimate.  Decompose the principal remainder in the
left bubble chart as
\[
 R_N^{\rm rem}=R_{N,\mathrm{core}}^{(2)}
   +R_{N,\mathrm{cut}}^{(1)}+R_{N,\mathrm{out}}.
\]
On every fixed subcore on which the actual principal tensor and the finite
principal tensor agree, Taylor expansion gives
\[
 Y_{R_{N,\mathrm{core}}^{(2)}}(r)
 +Y_{Q_N(w_N^\sharp)}(r)
 \le C A_N^2(1+r)^{4m_f+8}.
\]
The first-order term $R_{N,\mathrm{cut}}^{(1)}$ is supported in the cutoff
annulus $r\asymp R_{b,N}$, while $R_{N,\mathrm{out}}$ is supported farther
out.  Lemma~\ref{lem:step8} gives their normalized source norm the bound
$C_L\epsilon_N^{\kappa_f}$; on these annuli the weight in
\eqref{eq:bubble-transfer-weight-rigorous} is bounded below by a positive
constant depending only on the already fixed $L$.

For the quadratic core term, the crossover of the two terms in
\eqref{eq:bubble-transfer-weight-rigorous} occurs at $r\asymp1$, because
$R_{b,N}^{-n/2}\asymp_L\lambda^{a_\Sigma}$.  Hence
\begin{align*}
 \sup_{0\le r\le c_LR_{b,N}}
 \frac{A_N^2(1+r)^{4m_f+8}}
      {\lambda^{a_\Sigma}+((1+r)/R_{b,N})^{n/2}}
 \le C_L\left(
   \lambda^{-a_\Sigma}A_N^2
   +A_N^2R_{b,N}^{4m_f+8}
 \right).
\end{align*}
The second term is
$(A_NR_{b,N}^{2m_f+4})^2=O(\epsilon_N^{2\kappa_f})$ and is absorbed in
$C\varepsilon_N$.  Combining the core, cutoff, and outer pieces proves
\eqref{eq:nonprincipal-weighted-source-step20}.  The same calculation holds
after one normalized parameter derivative.

For the nonlinear difference, the blockwise estimate of
Lemma~\ref{lem:step16} gives on the exact left bubble
\begin{equation}\label{eq:nonlinear-difference-weighted-source-step20}
 \|[Q_N(w_{N,\lambda,\xi})-Q_N(w_N^\sharp)]_{b,-}
   \|_{\mathcal Y_{b,N}}
 \le Cr_{\rm nl}\|d_N\|_{\mathcal X_{b,N}}.
\end{equation}
Indeed, on each bubble annulus one factor is bounded by the fixed nonlinear
radius, and division by the same weight leaves the
$\mathcal X_{b,N}$ norm of $d_N$.  On the complement of the left exact
bubble,
\begin{equation}\label{eq:nonlinear-difference-outer-source-step20}
 \|[Q_N(w_{N,\lambda,\xi})-Q_N(w_N^\sharp)]_{\rm out}
   \|_{\widehat Y_{\gamma,\beta}}
 \le Cr_{\rm nl}\|d_N\|_{\widehat X_{\gamma,\beta}}
 \le C r_{\rm nl}\varepsilon_N.
\end{equation}
The last inequality follows from Lemma~\ref{lem:step17} and the global
linear estimate for $w_N^\sharp$.

Apply \eqref{eq:weighted-global-bubble-inverse-step20} to
\eqref{eq:difference-augmented-equation-step20}.  Equations
\eqref{eq:nonprincipal-weighted-source-step20}--
\eqref{eq:nonlinear-difference-outer-source-step20} yield
\[
 \|d_N(x_N^-+\epsilon_N\,\cdot)\|_{\mathcal X_{b,N}}
 \le C\left(\varepsilon_N+\lambda^{-a_\Sigma}A_N^2\right)
   +Cr_{\rm nl}\|d_N(x_N^-+\epsilon_N\,\cdot)
        \|_{\mathcal X_{b,N}}.
\]
The constant in this inequality is $C_{Q,b}$ from
Lemma~\ref{lem:step16}.  The radius condition
\eqref{eq:final-energy-radius-step16} therefore absorbs the last term.
We obtain
\begin{equation}\label{eq:weighted-difference-step20}
 \|d_N(x_N^-+\epsilon_N\,\cdot)\|_{\mathcal X_{b,N}}
 \le C\left(\varepsilon_N+\lambda^{-a_\Sigma}A_N^2\right).
\end{equation}
This estimate allows an $O(\varepsilon_N)$ interface field; it asserts only
that such a field acquires the one-sided bubble weight when propagated into
the bubble core.

Let $K\Subset\mathbb R^n$ be fixed.  On $K$,
$\omega_{b,N}\le C_K\lambda^{a_\Sigma}$, and hence
\begin{align}
 \|d_N(x_N^-+\epsilon_N\,\cdot)\|_{C^{2,\alpha}(K)}
 &\le C_K\left(\lambda^{a_\Sigma}\varepsilon_N+A_N^2\right)
 =o(A_N).
 \label{eq:nonlinear-difference-function-small-step20}
\end{align}
Here we used
$\lambda^{a_\Sigma}/A_N\asymp\epsilon_N^{\beta_f}\to0$ and $A_N\to0$.
Thus the desired local smallness follows without any assertion that the
interface Cauchy data are $o(A_N)$.

The projected coefficients are recovered from the local equation.  On the
fixed set $K_\Psi$ the intrinsic equation is
\begin{equation}\label{eq:nonlinear-coefficient-comparison-equation-step20}
 \mathscr P_{N,-}d_N
 =H_N^{\rm diff}
  +\sum_{a=0}^n(c_a-b_{N,a})\Psi_{a,N}^{-}.
\end{equation}
By \eqref{eq:nonlinear-difference-function-small-step20},
\[
 \|H_N^{\rm diff}\|_{C^\alpha(K_\Psi)}=o(A_N).
\]
Indeed, the local part of $R_N^{\rm rem}$ and the principal quadratic term
are $O(A_N^2)$, while the local Lipschitz estimate of
Lemma~\ref{lem:step16} gives $o(A_N)$ for the final nonlinear difference.
The local coefficient estimate above therefore gives
\begin{equation}\label{eq:nonlinear-difference-small-step20}
 \|d_N(x_N^-+\epsilon_N\,\cdot)\|_{C^{2,\alpha}(K)}
 +\sum_a|c_a-b_{N,a}|=o(A_N)
 \qquad(K\Subset\mathbb R^n).
\end{equation}
In particular,
\begin{equation}\label{eq:local-correction-scale-step20}
 \|w_{N,\lambda,\xi}(x_N^-+\epsilon_N\,\cdot)
   \|_{C^{2,\alpha}(K)}
 +\sum_a|c_a|
 \le C_KA_N.
\end{equation}
Combining \eqref{eq:principal-global-limit-step20} and
\eqref{eq:nonlinear-difference-small-step20} proves
\eqref{eq:left-localization-step20} and
\eqref{eq:coefficient-localization-step20}.

The same argument applies after one normalized parameter derivative.  In
the natural bubble coordinates,
$|D\omega_{b,N}|\le C\omega_{b,N}$ and $A_N$ is independent of
$(\lambda',\xi')$.  Differentiating
\eqref{eq:difference-augmented-equation-step20}, using
\eqref{eq:quadratic-parameter-step16}, and absorbing the differentiated
nonlinear term gives
\[
 \|Dd_N\|_{\mathcal X_{b,N}}
 \le C\left(\varepsilon_N+\lambda^{-a_\Sigma}A_N^2\right).
\]
Consequently, $Dd_N=o(A_N)$ on every fixed bubble compact set, and local
coefficient extraction gives
$D(c_a-b_{N,a})=o(A_N)$.

\emph{Step 4: uniqueness of the local limit and uniformity.}
Any convergent parameter sequence has a subsequence for which $A_N^{-1}w_N$ converges locally.  Dividing the pulled-back equation by
$A_N$, using Steps~1--3, and passing to the limit gives
\[
 \mathscr L_{\lambda',\xi'}\widetilde W
 =-U^{-p_*}\mathcal R_{\xi',\lambda'}
  +\sum_{a=0}^n\widetilde\beta_a\Psi_a^{\lambda',\xi'}.
\]
The global paired constraints reduce, by $J$-invariance, to the one-ended
conditions $\ell_b^{\lambda',\xi'}(\widetilde W)=0$.  Proposition~\ref{lem:step12}
therefore identifies the limit uniquely with
$W_{\xi',\lambda'}$ and the coefficients with $\beta_a$.  A compactness
argument on the normalized parameter set upgrades subsequential convergence
to uniform convergence.

The differentiated local convergence and coefficient convergence were
proved at the end of Step~3.  Together with the differentiated principal
source convergence, they identify the parameter derivative of every local
limit with the derivative of $W_{\xi',\lambda'}$.

Finally, $w_{N,\lambda,\xi}$ is exactly $J$-invariant.  In the right chart of
\eqref{eq:metric-centered-charts-step20} its pullback is therefore the same function of $z$ as on the left.  This proves
\eqref{eq:right-localization-step20}; the physical-coordinate formulation is
the exact Kelvin image.  The proof is complete.
\end{proof}

\subsection{Absolute bubble energy and principal-tail sources}
The round Jacobi form has one negative degree-zero direction and a
degree-one kernel.  Tail estimates therefore use the positive graph norm
defined next.

Let $\Pi_k^*$ denote the degree-$k$ spherical-harmonic projection on the
round sphere $(\mathbb S^n,g_*)$, and let
$\lambda_k^*=k(k+n-1)$.  Define
\begin{equation}\label{eq:absolute-energy-norm-rigorous}
 \|u\|_{\mathcal E_*}^2
 :=\|\Pi_0^*u\|_{L^2(g_*)}^2
 +\sum_{k\ge2}(1+\lambda_k^*)
       \|\Pi_k^*u\|_{L^2(g_*)}^2
 +\sum_{b=0}^n|\ell_b^*(u)|^2,
\end{equation}
where the localized functionals $\ell_b^*$ are those of
Proposition~\ref{lem:step12}.  Let $\mathcal E_*'$ be the dual space.

\begin{lemma}[Absolute round graph norm]
\label{aux:absolute-round-graph-rigorous}
The norm \eqref{eq:absolute-energy-norm-rigorous} is equivalent to the
$H^1(g_*)$ norm.  The complete round augmented inverse satisfies
\begin{equation}\label{eq:absolute-round-inverse-rigorous}
 \|G_*F\|_{\mathcal E_*}+|\mathbf c|
 \le C\|F\|_{\mathcal E_*'},
\end{equation}
and the round Jacobi bilinear form obeys
\begin{equation}\label{eq:absolute-round-bilinear-rigorous}
 |\mathscr B_*(u,v)|
 \le C\|u\|_{\mathcal E_*}\|v\|_{\mathcal E_*}.
\end{equation}
Moreover,
\begin{equation}\label{eq:L2-to-Edual-rigorous}
 \|F\|_{\mathcal E_*'}\le C\|F\|_{L^2(g_*)}.
\end{equation}
\end{lemma}

\begin{proof}
Decompose $u$ into its degree-one component and its $L^2(g_*)$-orthogonal
complement.  The localized Gram matrix controls the degree-one coefficients,
while the spectral gap of $\mathscr L_*$ away from that eigenspace controls
the complement, including the degree-zero direction in the absolute norm.
This proves the equivalence with $H^1(g_*)$ and the inverse estimate
\eqref{eq:absolute-round-inverse-rigorous}.  The bilinear bound follows from
the same spectral decomposition, and \eqref{eq:L2-to-Edual-rigorous} is an
immediate consequence of $\|u\|_{L^2}\le C\|u\|_{\mathcal E_*}$.
\end{proof}

For a truncated round-bubble portion
$\mathcal B_T=\{t\le T\}$ define the positive local energy norm
\begin{equation}\label{eq:local-absolute-energy-rigorous}
 \|u\|_{\mathcal E_{*,T}}^2
 :=\|u\|_{H^1(\mathcal B_T,g_*)}^2
   +\sum_{b=0}^n|\ell_b^*(u)|^2,
\end{equation}
and let $\mathcal E_{*,T}'$ be its dual.  The localized functionals are
supported in a fixed compact subset, so these norms are uniform in $T$.
In particular,
\begin{equation}\label{eq:local-L2-dual-rigorous}
 \|F\|_{\mathcal E_{*,T}'}\le C\|F\|_{L^2(\mathcal B_T,g_*)},
\end{equation}
with $C$ independent of $T$.

Recall
\begin{equation}\label{eq:bubble-tail-radius-rigorous}
 R_{\sharp,N}=\frac{\sigma_N^\sharp}{\epsilon_N}
 \asymp\epsilon_N^{4/n-1}.
\end{equation}
The intrinsic normalized first principal metric source in the bubble chart is a
linear combination of
\begin{equation}\label{eq:first-source-structure-rigorous}
 U^{-8/(n-4)}\Bigl(
 D^2\widehat h_N+D\log U\,D\widehat h_N
 +(D^2\log U+|D\log U|^2)\widehat h_N\Bigr).
\end{equation}
For $r=|z|\ge1$,
\[
 U^{-8/(n-4)}\le Cr^4,
 \quad |D\log U|\le Cr^{-1},
 \quad |D^2\log U|\le Cr^{-2},
\]
and
\[
 |D^j(A_N\widehat h_N)(z)|
 \le CA_N(1+r)^{2m_f+2-j}
      \mathbf1_{\{r\le2R_{\sharp,N}\}},
 \qquad 0\le j\le2.
\]
Consequently, the actual normalized first source and one normalized parameter
derivative satisfy
\begin{equation}\label{eq:first-tail-majorant-rigorous}
 |F_N^{(1)}(z)|+|\mathfrak DF_N^{(1)}(z)|
 \le CA_N(1+r)^{2m_f+4}
      \mathbf1_{\{r\le2R_{\sharp,N}\}}.
\end{equation}
The same majorant holds for the finite-$N$ cutoff commutator, since on its
support $r\asymp R_{\sharp,N}$ and every cutoff derivative contributes a
factor $O(r^{-1})$.

The projected Green estimate used in Lemma~\ref{lem:step19} gives
\begin{equation}\label{eq:principal-response-majorant-rigorous}
 \sum_{j=0}^2(1+r)^j|D^jW_N^\sharp(z)|
 \le C(1+r)^{2m_f+4}.
\end{equation}
Set
\begin{equation}\label{eq:theta-tail-rigorous}
 \Theta_N:=A_N(R_{\sharp,N})^{2m_f+4}.
\end{equation}
Then
\begin{equation}\label{eq:theta-tail-small-rigorous}
 \Theta_N
 \le C\epsilon_N^{\beta_f-2+8(m_f+2)/n}
 =C\epsilon_N^{\kappa_f}\longrightarrow0.
\end{equation}
The local quadratic estimate therefore yields
\begin{align}
 |F_N^{(2)}(z)|+|\mathfrak DF_N^{(2)}(z)|
 &\le CA_N^2(1+r)^{4m_f+8}
        \mathbf1_{\{r\le2R_{\sharp,N}\}}\nonumber\\
 &\le C\Theta_NA_N(1+r)^{2m_f+4}
        \mathbf1_{\{r\le2R_{\sharp,N}\}}.
 \label{eq:quadratic-tail-majorant-rigorous}
\end{align}

\begin{lemma}[Principal-tail dual norm]
\label{aux:tail-dual-rigorous}
Let $F_{N,R}^{\rm loc}$ be the part in $\{r>R\}$ of the first principal metric
source, its finite-$N$ cutoff commutator, or the quadratic principal source,
and let $T_N$ be the interface parameter of the corresponding truncated
round-bubble portion.  Then
\begin{align}\label{eq:tail-dual-rigorous}
 &\|F_{N,R}^{\rm loc}\|_{\mathcal E_*'}
 +\|\mathfrak DF_{N,R}^{\rm loc}\|_{\mathcal E_*'}\nonumber\\
 &\quad+\|F_{N,R}^{\rm loc}\|_{\mathcal E_{*,T_N}'}
 +\|\mathfrak DF_{N,R}^{\rm loc}\|_{\mathcal E_{*,T_N}'}
 \le CA_N\omega_0(R),\qquad
 \omega_0(R)=CR^{-\frac12(n-4m_f-8)}.
\end{align}
In particular, $\omega_0(R)\to0$ under $n>4m_f+8$.
\end{lemma}

\begin{proof}
The round volume form satisfies
$dv_{g_*}\asymp r^{-n-1}\,dr\,d\theta$ for $r\ge1$.  By
\eqref{eq:L2-to-Edual-rigorous}, \eqref{eq:local-L2-dual-rigorous}, and
\eqref{eq:first-tail-majorant-rigorous},
\begin{align*}
 \|F_{N,R}^{\rm loc}\|_{\mathcal E_*'}^2
 &\le CA_N^2\int_R^{2R_{\sharp,N}}
       r^{4m_f+8}r^{-n-1}\,dr\\
 &\le CA_N^2\int_R^\infty r^{4m_f-n+7}\,dr
 =\frac{CA_N^2}{n-4m_f-8}
   R^{-(n-4m_f-8)}.
\end{align*}
The same calculation with \eqref{eq:local-L2-dual-rigorous} proves the two local-dual terms.  This proves the estimate for the first source, the cutoff commutator, and
their parameter derivatives.  For the quadratic source use
\eqref{eq:quadratic-tail-majorant-rigorous}; the same calculation has the
additional factor $\Theta_N\le1$ for large $N$.
\end{proof}

\begin{proposition}[Localized energy estimate for the projected inverse]
\label{aux:localized-energy-inverse-rigorous}
Initially let $F$ be a smooth $J$-paired source supported in the two exact
bubble portions, so that the actual global projected response
$(u,\mathbf d)=\mathbb G_NF$ is defined by the global inverse of
Lemma~\ref{lem:step15} in the refined spaces of
Proposition~\ref{prop:refined-transfer-step16}.  The estimates below depend
only on the indicated
energy-dual norms and therefore extend uniquely by density to all paired
data in the corresponding energy-dual completions; the extended response
is again denoted by $(u,\mathbf d)=\mathbb G_NF$.  The differentiated
conclusions extend in the same way to $C^1$ families in these fixed-domain
energy-dual spaces.
Write $F_\pm$ and $u_\pm$ for the intrinsic round-bubble pullbacks, and let
$T_{N,\pm}$ be their interface parameters.  For the fixed $T_E$ in
\eqref{eq:large-bubble-truncation-threshold-step17}, the ordered choice
ensures $T_{N,\pm}\ge T_E$, and one has
\begin{equation}\label{eq:localized-energy-inverse-rigorous}
 \sum_{\pm}\|u_\pm\|_{\mathcal E_{*,T_{N,\pm}}}+|\mathbf d|
 \le C\sum_{\pm}\|F_\pm\|_{\mathcal E_{*,T_{N,\pm}}'}.
\end{equation}
Since
$\mathscr P_{N,\pm}u_\pm=F_\pm+\sum_ad_a\Psi_{a,N}^\pm$
intrinsically, the same estimate gives the graph bound
\begin{align}\label{eq:localized-energy-graph-inverse-rigorous}
 \sum_\pm\Bigl(
  \|u_\pm\|_{\mathcal E_{*,T_{N,\pm}}}
 +\|\mathscr P_{N,\pm}u_\pm\|_{\mathcal E_{*,T_{N,\pm}}'}\Bigr)
 +|\mathbf d|
 \le C\sum_\pm\|F_\pm\|_{\mathcal E_{*,T_{N,\pm}}'}.
\end{align}
If $\mathfrak D$ is one normalized parameter derivative, taken under the
fixed-domain convention described in the proof, then
\begin{align}
 \sum_\pm\|\mathfrak Du_\pm\|_{\mathcal E_{*,T_{N,\pm}}}
 +|\mathfrak D\mathbf d|
 &\le C\sum_\pm\Bigl(
 \|\mathfrak DF_\pm\|_{\mathcal E_{*,T_{N,\pm}}'}
 +\|F_\pm\|_{\mathcal E_{*,T_{N,\pm}}'}\Bigr),
 \label{eq:localized-energy-inverse-derivative-rigorous}\\
 \sum_\pm\Bigl(
 \|\mathfrak Du_\pm\|_{\mathcal E_{*,T_{N,\pm}}}
 +\|\mathscr P_{N,\pm}\mathfrak Du_\pm
       \|_{\mathcal E_{*,T_{N,\pm}}'}\Bigr)
 +|\mathfrak D\mathbf d|
 &\le C\sum_\pm\Bigl(
 \|\mathfrak DF_\pm\|_{\mathcal E_{*,T_{N,\pm}}'}
 +\|F_\pm\|_{\mathcal E_{*,T_{N,\pm}}'}\Bigr).
 \label{eq:localized-energy-graph-derivative-rigorous}
\end{align}
The constants are independent of the selected index and of the normalized
parameters in the ordered regime.
\end{proposition}

\begin{proof}
We use the common Liouville gauge on the bubble and transition sides.  On a
truncated round bubble \(\mathcal B_T=\{t<T\}\), set
\[
 Y_u(t,\theta)=\cosh(t)^{-\frac{n-2}{2}}u(t,\theta),\qquad
 \mathcal T_Tu=Y_u(T,\cdot),\qquad
 \mathcal C_Tu=\partial_tY_u(T,\cdot).
\]
The transition gauge is normalized so that both the value and the outward
conormal agree at the common interface.  In the original variable,
\begin{align}
 \mathcal T_Tu&=\cosh(T)^{-\frac{n-2}{2}}u|_{\Gamma_T},\nonumber\\
 \mathcal C_Tu&=\cosh(T)^{-\frac{n-2}{2}}
 \left(\partial_tu-\frac{n-2}{2}\tanh(T)u\right)\Big|_{\Gamma_T}.
 \label{eq:liouville-original-cauchy-gapfix}
\end{align}
On an actual bubble block the relative operator is in divergence form, and
its weak conormal in this gauge is uniformly equivalent to
\(\mathcal C_Tu\) in \(H^{-1/2}(\mathbb S^{n-1})\).  Thus the form and
conormal estimates in Proposition~\ref{aux:preparameter-energy-package}
apply with this common trace normalization.

We first recall the two bubble-side consequences of that proposition.  The
zero-trace augmented problem on \(\mathcal B_T\), \(T\ge T_E\), has a
unique solution satisfying
\begin{equation}\label{eq:zero-boundary-energy-estimate-rigorous}
 \|u^0\|_{\mathcal E_{*,T}}+|\mathbf d^0|
 \le C\bigl(\|F\|_{\mathcal E_{*,T}'}+|\mathbf g|\bigr).
\end{equation}
Moreover, if \(H_T^{\rm bub}\varphi\) is the augmented Poisson field with
trace \(\varphi\in H^{1/2}(\mathbb S^{n-1})\) and zero localized
constraints, then
\begin{equation}\label{eq:bubble-poisson-operator-bound-rigorous}
 \|H_T^{\rm bub}\varphi\|_{\mathcal E_{*,T}}
 +|\mathbf d(\varphi)|
 \le C\|\varphi\|_{H^{1/2}}.
\end{equation}
The corresponding weak flux is bounded in \(H^{-1/2}\).  These estimates
are uniform in \(T\) and in the normalized bubble parameters.

We next pass from the ideal block problem to the finite-flat and actual
ones.  Let
\[
 \mathcal S^{\rm id}
 =\mathcal D_T^{\rm bub}+\mathcal D^{\rm ext}
 :\mathcal H_\partial\longrightarrow\mathcal H_\partial'
\]
be the ideal interface Schur complement.  Its uniform invertibility is
\eqref{eq:preparameter-ideal-schur-bound}.  The finite-flat bubble blocks
are exact round bubbles in their intrinsic coordinates, while the exterior
form differs from the ideal one by
\eqref{eq:preparameter-finite-ideal-form-bound}.  Hence
Proposition~\ref{aux:abstract-energy-dtn-stability}, together with
\eqref{eq:choose-N-flat-geometry-step17} and
\eqref{eq:energy-geometric-margin-step17}, gives uniform zero-trace and
Poisson bounds for the finite-flat exterior problem and a uniformly
invertible finite-flat Schur complement.

The actual bubble and exterior forms, sources, and constraints differ from
the finite-flat ones by the quantity in
\eqref{eq:preparameter-actual-flat-form-bound}.  The choices
\eqref{eq:choose-N0-energy-step17} and
\eqref{eq:choose-N-energy-step17} therefore allow a second application of
Proposition~\ref{aux:abstract-energy-dtn-stability}.  In particular,
\begin{equation}
 \|G_{N,{\rm bub}}^D\|+\|G_{N,{\rm ext}}^D\|\le C,
 \label{eq:actual-zero-trace-inverses-prop74}
\end{equation}
and the actual Dirichlet-to-Neumann maps differ from their finite-flat
counterparts by a sufficiently small operator on
\(\mathcal H_\partial\).  Consequently
\[
 \mathcal S_N:=\mathcal D_{N}^{\rm bub}+\mathcal D_N^{\rm ext}
\]
is uniformly invertible from \(\mathcal H_\partial\) to
\(\mathcal H_\partial'\).

We now solve the global problem.  Let
\((u_N^0,\mathbf d_N^0)\) be the paired actual zero-trace bubble solution
with source \(F\), and let \(h_{N,F}\in\mathcal H_\partial'\) be its pair
of outward weak fluxes.  By
\eqref{eq:actual-zero-trace-inverses-prop74} and the weak-flux estimate in
Proposition~\ref{aux:abstract-energy-dtn-stability},
\[
 \sum_\pm\|u_{N,\pm}^0\|_{\mathcal E_{*,T_{N,\pm}}}
 +|\mathbf d_N^0|+\|h_{N,F}\|_{\mathcal H_\partial'}
 \le C\sum_\pm\|F_\pm\|_{\mathcal E_{*,T_{N,\pm}}'}.
\]
The common interface trace is
\(\varphi=-\mathcal S_N^{-1}h_{N,F}\).  Adding the actual bubble Poisson
fields and the actual exterior homogeneous field with trace \(\varphi\)
reconstructs a global weak augmented solution.  The uniform Poisson bounds
and the estimate for \(\mathcal S_N^{-1}\) give
\eqref{eq:localized-energy-inverse-rigorous}.  For smooth data, equality of
the weak conormal traces and blockwise elliptic regularity make this a
classical solution; uniqueness in Lemma~\ref{lem:step15} identifies it with
\(\mathbb G_NF\).  General energy-dual data follow by density.  The
intrinsic equation and the uniform dual norms of the localized sources then
give \eqref{eq:localized-energy-graph-inverse-rigorous}.

It remains to differentiate the construction.  Near each moving interface
we identify neighboring domains by a cylindrical diffeomorphism which is
the identity away from a fixed terminal collar and a translation near the
boundary.  Since \(|\mathfrak DT_{N,\pm}|\le C\), the pulled-back spaces,
trace maps, and trace extensions are fixed and uniformly equivalent.
The parameter-dependent part of
Propositions~\ref{aux:preparameter-energy-package} and~\ref{aux:abstract-energy-dtn-stability} gives uniform derivatives of the
zero-trace solutions, Poisson maps, weak fluxes, and
\(\mathcal S_N\).  Differentiating
\[
 \mathcal S_N\varphi=-h_{N,F}
\]
and using the already established inverse bound yields the estimate for
\(\mathfrak D\varphi\).  Differentiating the reconstruction proves
\eqref{eq:localized-energy-inverse-derivative-rigorous}; differentiating
the intrinsic bubble equations then gives
\eqref{eq:localized-energy-graph-derivative-rigorous}.
\end{proof}

\begin{lemma}[Capacity estimate for boundary-induced bubble fields]
\label{aux:bubble-capacity-rigorous}
Let $\mathcal B_T=\{t\le T\}$ be a truncated round-bubble portion and let
$u$ solve the augmented round equation
\[
 \mathscr L_*u=f+\sum_{a=0}^nd_a\Psi_a,
 \qquad \ell_b^*(u)=0,
\]
with interface parameter $T\ge T_E$.  Write
$\phi=u|_{\Gamma_T}$ and $\psi=\partial_tu|_{\Gamma_T}$ in spherical
harmonics.  Then the stronger estimate
\begin{equation}\label{eq:bubble-capacity-liouville-rigorous}
 \|u\|_{\mathcal E_{*,T}}+|\mathbf d|
 \le C\left(
  \|f\|_{\mathcal E_{*,T}'}
  +\|\mathcal T_Tu\|_{H^{1/2}(\mathbb S^{n-1})}\right)
\end{equation}
holds, with $C$ independent of $T$.  In particular,
\begin{align}\label{eq:bubble-capacity-rigorous}
 \|u\|_{\mathcal E_{*,T}}
 \le C\bigl(&e^{-\frac{n-2}{2}T}
   (\|\phi\|_{H^{1/2}(\mathbb S^{n-1})}
    +\|\psi\|_{H^{-1/2}(\mathbb S^{n-1})})
   +\|f\|_{\mathcal E_{*,T}'}+|\mathbf d|\bigr).
\end{align}
The same conclusions hold for the actual bubble block in the ordered
regime.

Under the fixed-domain convention in the proof of
Proposition~\ref{aux:localized-energy-inverse-rigorous}, one normalized
parameter derivative satisfies
\begin{align}\label{eq:bubble-capacity-derivative-rigorous}
 \|\mathfrak Du\|_{\mathcal E_{*,T}}+|\mathfrak D\mathbf d|
 \le C\bigl(&\|\mathfrak Df\|_{\mathcal E_{*,T}'}
            +\|f\|_{\mathcal E_{*,T}'}
            +\|\mathfrak D(\mathcal T_Tu)\|_{H^{1/2}}
            +\|\mathcal T_Tu\|_{H^{1/2}}\bigr).
\end{align}

If the boundary Cauchy data are generated by a normalized source supported
outside the bubble portion, then
\begin{equation}\label{eq:remote-capacity-gain-rigorous}
 e^{-\frac{n-2}{2}T}
 (\|\phi\|_{H^{1/2}}+\|\psi\|_{H^{-1/2}})
 \le C\lambda^{a_\Sigma}
       \|F\|_{\widehat Y_{\gamma,\beta}}.
\end{equation}
The corresponding fixed-domain differentiated estimate holds with the
right-hand side replaced by
$C\lambda^{a_\Sigma}(\|\mathfrak DF\|_{\widehat Y_{\gamma,\beta}}
+\|F\|_{\widehat Y_{\gamma,\beta}})$.
\end{lemma}

\begin{proof}
We first prove the round estimate.  Put
\[
 \varphi=\mathcal T_Tu
\]
and let
$(H_T^{\rm bub}\varphi,\mathbf e_T(\varphi))$ be the augmented bubble
Poisson pair constructed in the proof of
Proposition~\ref{aux:localized-energy-inverse-rigorous}.  Thus
\begin{equation}\label{eq:capacity-poisson-properties-rigorous}
 \mathscr L_*H_T^{\rm bub}\varphi
   =\sum_{a=0}^ne_{T,a}(\varphi)\Psi_a,
 \qquad
 \mathcal T_TH_T^{\rm bub}\varphi=\varphi,
 \qquad
 \ell_b^*(H_T^{\rm bub}\varphi)=0,
\end{equation}
and, by \eqref{eq:bubble-poisson-operator-bound-rigorous},
\begin{equation}\label{eq:capacity-poisson-bound-rigorous}
 \|H_T^{\rm bub}\varphi\|_{\mathcal E_{*,T}}
 +|\mathbf e_T(\varphi)|
 \le C\|\varphi\|_{H^{1/2}}.
\end{equation}
Set
\[
 u^0=u-H_T^{\rm bub}\varphi,
 \qquad
 \mathbf d^0=\mathbf d-\mathbf e_T(\varphi).
\]
Then
\begin{equation}\label{eq:capacity-zero-trace-decomposition-rigorous}
 \mathscr L_*u^0=f+\sum_{a=0}^nd_a^0\Psi_a,
 \qquad
 \mathcal T_Tu^0=0,
 \qquad
 \ell_b^*(u^0)=0.
\end{equation}
The uniform zero-trace augmented inverse
\eqref{eq:zero-boundary-energy-estimate-rigorous} therefore gives
\begin{equation}\label{eq:capacity-zero-trace-bound-rigorous}
 \|u^0\|_{\mathcal E_{*,T}}+|\mathbf d^0|
 \le C\|f\|_{\mathcal E_{*,T}'}.
\end{equation}
Combining \eqref{eq:capacity-poisson-bound-rigorous} and
\eqref{eq:capacity-zero-trace-bound-rigorous} proves
\eqref{eq:bubble-capacity-liouville-rigorous}.

By \eqref{eq:liouville-original-cauchy-gapfix},
\begin{align}\label{eq:capacity-trace-conversion-gapfix}
 \|\mathcal T_Tu\|_{H^{1/2}}
 +\|\mathcal C_Tu\|_{H^{-1/2}}
 \le Ce^{-\frac{n-2}{2}T}
 (\|\phi\|_{H^{1/2}}+\|\psi\|_{H^{-1/2}}).
\end{align}
In particular,
\eqref{eq:bubble-capacity-liouville-rigorous} implies the weaker form
\eqref{eq:bubble-capacity-rigorous}.  The localized constraints in the
zero-trace inverse and the Poisson pair remove the dilation and translation
branches, so no compactness argument is required.

For the actual bubble operator, use the actual zero-trace augmented inverse
from \eqref{eq:actual-zero-trace-inverses-prop74} and the actual Poisson
pair given by Proposition~\ref{aux:abstract-energy-dtn-stability}.  More precisely, if
$H_{N,T}^{\rm bub}\varphi$ denotes the actual Poisson field, then
\eqref{eq:abstract-poisson-dtn-bound} gives
\[
 \|H_{N,T}^{\rm bub}\varphi\|_{\mathcal E_{*,T}}
 +|\mathbf e_{N,T}(\varphi)|
 \le C\|\varphi\|_{H^{1/2}},
\]
while \eqref{eq:actual-zero-trace-inverses-prop74} gives the analogue of
\eqref{eq:capacity-zero-trace-bound-rigorous}.  The same decomposition
$u=u^0+H_{N,T}^{\rm bub}(\mathcal T_Tu)$ therefore proves
\eqref{eq:bubble-capacity-liouville-rigorous} and
\eqref{eq:bubble-capacity-rigorous} for the actual block, with constants
uniform in the ordered regime.

We next prove the parameter estimate.  Pull the moving bubble portion to
the fixed-domain identification used in the proof of
Proposition~\ref{aux:localized-energy-inverse-rigorous}; in these
coordinates the trace map and the supports of the projection sources and
constraints are fixed.  Differentiate the decomposition
\[
 u=u^0+H_{N,T}^{\rm bub}(\mathcal T_Tu),
 \qquad
 \mathbf d=\mathbf d^0+\mathbf e_{N,T}(\mathcal T_Tu).
\]
The differentiated zero-trace inverse estimate
\eqref{eq:abstract-flux-parameter-bound} gives
\[
 \|\mathfrak Du^0\|_{\mathcal E_{*,T}}
 +|\mathfrak D\mathbf d^0|
 \le C\bigl(
   \|\mathfrak Df\|_{\mathcal E_{*,T}'}
  +\|f\|_{\mathcal E_{*,T}'}\bigr),
\]
and \eqref{eq:abstract-dtn-parameter-bound}, together with
\eqref{eq:abstract-poisson-dtn-bound}, gives
\begin{align*}
 &\|\mathfrak D[H_{N,T}^{\rm bub}(\mathcal T_Tu)]
   \|_{\mathcal E_{*,T}}
 +|\mathfrak D[\mathbf e_{N,T}(\mathcal T_Tu)]|\\
 &\hspace{4em}\le C\bigl(
   \|\mathfrak D(\mathcal T_Tu)\|_{H^{1/2}}
  +\|\mathcal T_Tu\|_{H^{1/2}}\bigr).
\end{align*}
This proves \eqref{eq:bubble-capacity-derivative-rigorous}.

It remains to prove the remote gain.  Consider first a source in the
matched transition.  The conormal estimate in Proposition~\ref{lem:step14} gives,
at the bubble-side endpoint,
\begin{equation}\label{eq:remote-physical-cauchy-gapfix}
 \|\phi\|_{H^{1/2}}+\|\psi\|_{H^{-1/2}}
 \le C_L\delta_\lambda\|F\|_{\widehat Y_{\gamma,\beta}}.
\end{equation}
The Cauchy data in this estimate are measured at the bubble-side interface,
where \eqref{eq:transition-liouville-normalization-step15} agrees exactly
with the round-bubble trace.  Hence no factor $\Theta_\lambda^{-1}$ occurs.
For a source entering from a Schwarzschild block, propagation is first
performed in the transition-normalized variable
$Y_\Sigma^{\rm tr}=\Theta_\lambda Y_\Sigma$; the same constant multiplies
its value and conormal flux and therefore leaves the homogeneous
Dirichlet-to-Neumann coefficient unchanged.  Only the resulting
bubble-side trace enters the estimate.
For a source farther out on a Schwarzschild half-neck, the kernels of
Lemma~\ref{lem:step13} and the uniformly invertible low-mode matching matrix
give the same bound: on a source block at radius $r_s\ge r_b$ the transfer
factor is
\[
 q(r_s)\left(\frac{r_b}{r_s}\right)^{a_\Sigma}
 \le Cr_b^{a_\Sigma}\asymp C\delta_\lambda.
\]
Connector and opposite-end sources first pass through a fixed
Dirichlet-to-Neumann block and obey the same estimate.  Finally,
$e^T\asymp R_{b,N}\asymp\lambda^{4/n-1}$, so
\[
 \delta_\lambda e^{-\frac{n-2}{2}T}
 \asymp\lambda^{2-8/n}
 \lambda^{(1-4/n)(n-2)/2}
 =\lambda^{a_\Sigma+(n-4)/n}
 \le C\lambda^{a_\Sigma}.
\]
Combining this with \eqref{eq:remote-physical-cauchy-gapfix} proves
\eqref{eq:remote-capacity-gain-rigorous}.  The coefficient vector is
controlled by \eqref{eq:localized-coefficient-gain-repair}.

Under the fixed-domain convention, differentiating the transition Green
operators, the Schwarzschild kernels, and the low-mode matching matrix
preserves \eqref{eq:remote-physical-cauchy-gapfix}, with right-hand side
\[
 C_L\delta_\lambda\bigl(
  \|\mathfrak DF\|_{\widehat Y_{\gamma,\beta}}
 +\|F\|_{\widehat Y_{\gamma,\beta}}\bigr).
\]
Since $|\mathfrak DT|\le C$ and
$\mathfrak D\delta_\lambda=O(\delta_\lambda)$, the preceding exponent
calculation is unchanged.  This proves the differentiated remote estimate
and completes the proof.
\end{proof}

\begin{proposition}[Localized bubble capacity and inverse pairing]
\label{prop:localized-bubble-energy-step21}
Let $\mathcal B_{N,\pm}$ be the closures of the two exact bubble portions,
cut at the bubble-side interfaces and viewed in their compactified round
metrics.  Let $(u_N,\mathbf d_N)=\mathbb G_NF_N$ be the global projected
response to a $J$-invariant source which vanishes in both exact principal
bubble portions.  Denote by $u_{N,\pm}$ the intrinsic bubble pullbacks and
by $T_{N,\pm}$ the interface parameters.  Then
\begin{align}\label{eq:remote-bubble-energy-step21}
 \sum_\pm\Bigl(
  \|u_{N,\pm}\|_{\mathcal E_{*,T_{N,\pm}}}
 +\|\mathscr P_{N,\pm}u_{N,\pm}\|_{\mathcal E_{*,T_{N,\pm}}'}\Bigr)
 +\sum_{a=0}^n|d_{N,a}|
 \le C\lambda^{a_\Sigma}
       \|F_N\|_{\widehat Y_{\gamma,\beta}}.
\end{align}
The same estimate holds after one normalized parameter derivative.  On each
fixed compact bubble set one also has the $C^{2,\alpha}$ estimate of
Proposition~\ref{prop:localized-bubble-transfer-step20}.

There is an exact Green identity for the actual relative operator.  Put
\[
 \bar g_N=(\bar v_{N,\lambda,\xi})^{8/(n-4)}g,
 \quad B_N=\bar g_N^{-1}A_{\bar g_N},
 \quad T_N=T_1(B_N),
 \quad \kappa=\frac14\binom n2,
\]
and $V_N=(n-4)\sigma_2(B_N)-n\kappa$.  Then
\begin{equation}\label{eq:actual-relative-divergence-form-step21}
 \mathscr L_Nz
 =(\bar v_{N,\lambda,\xi})^{p_*}
 \bigl(-T_N^{ij}\bar\nabla_i\bar\nabla_jz+V_Nz\bigr),
\end{equation}
and, for every smooth domain $\Omega$,
\begin{align}\label{eq:actual-green-identity-step21}
 \int_\Omega(u\mathscr L_Nz-z\mathscr L_Nu)\,dv_g
 =-\int_{\partial\Omega}T_N^{ij}\nu_i
 (u\bar\nabla_jz-z\bar\nabla_ju)\,d\sigma_{\bar g_N}.
\end{align}
The associated symmetric interior form is
\begin{equation}\label{eq:actual-energy-bilinear-step21}
 \mathscr H_{N,\Omega}(u,z)
 :=4a\int_\Omega
 (T_N^{ij}\bar\nabla_i u\bar\nabla_jz+V_Nuz)\,dv_{\bar g_N}.
\end{equation}
Let $\mathscr P_{N,\pm}$ denote the intrinsic normalized pullback of
$\mathscr L_N$ to the corresponding round-bubble chart.  For a solution on
one bubble portion define
\[
 \|u\|_{\mathfrak E_{N,\pm}}
 :=\|u\|_{\mathcal E_{*,T_{N,\pm}}}
   +\|\mathscr P_{N,\pm}u\|_{\mathcal E_{*,T_{N,\pm}}'}
   +\sum_a|d_a|.
\]
Then
\begin{equation}\label{eq:bubble-bilinear-bound-step21}
 |\mathscr H_{N,\mathcal B_{N,\pm}}(u,z)|
 +|\operatorname{Flux}_{N,\pm}(u,z)|
 \le C\|u\|_{\mathfrak E_{N,\pm}}
       \|z\|_{\mathfrak E_{N,\pm}},
\end{equation}
where the flux is the boundary term in
\eqref{eq:actual-green-identity-step21}.  Consequently a remote response has
self-pairing $O(\lambda^{n-4}\|F_N\|^2)$, while its pairing with a principal
response of energy size $O(A_N)$ is
$O(A_N\lambda^{a_\Sigma}\|F_N\|)$.

Let $F_{N,R}^{\rm tail}$ be the paired part, lying in $|z|>R$, of the first
principal metric source, its finite-$N$ cutoff error, or the quadratic source
generated by the finite principal response.  Put
$u_{N,R}^{\rm tail}=\mathbb G_NF_{N,R}^{\rm tail}$.  Then
\begin{align}\label{eq:principal-tail-pairing-step21}
 &\sum_\pm\Bigl(
 |\mathscr H_{N,\mathcal B_{N,\pm}}(u_{N,R}^{\rm tail},u_{N,R}^{\rm tail})|
 +|\mathscr H_{N,\mathcal B_{N,\pm}}(w_N^\sharp,u_{N,R}^{\rm tail})|\Bigr)
 \nonumber\\
 &\quad+4a\sum_\pm\left|
 \int_{\mathcal B_{N,\pm}}P_{N,\pm}^\sharp
 u_{N,R}^{\rm tail}\,dv_g\right|
 \le A_N^2\omega(R),\qquad \omega(R)\to0,
\end{align}
uniformly in $N$ and with one normalized parameter derivative.
\end{proposition}

\begin{proof}
For a remote source the intrinsic equation in an exact bubble portion has no
interior source except the localized projection terms.  Proposition~\ref{prop:localized-bubble-transfer-step20} gives
$|\mathbf d_N|\le C\lambda^{a_\Sigma}\|F_N\|$.  The transition and neck
estimates give Cauchy data of size $C\delta_\lambda\|F_N\|$ at the bubble
interface.  Lemma~\ref{aux:bubble-capacity-rigorous} therefore gives the
energy and coefficient parts of \eqref{eq:remote-bubble-energy-step21}.  In
the bubble interior the source vanishes and
$\mathscr P_{N,\pm}u_{N,\pm}=\sum_ad_{N,a}\Psi_{a,N}^\pm$; the uniform dual
norms of the localized sources give the displayed operator-dual term.  Exact
Kelvin conjugation gives the second end, and the parameter statement follows
from the differentiated capacity estimate.

Model-relative Newton ellipticity and boundedness of $V_N$ imply
\[
 |\mathscr H_{N,\mathcal B_{N,\pm}}(u,z)|
 \le C\|u\|_{\mathcal E_{*,T_{N,\pm}}}
       \|z\|_{\mathcal E_{*,T_{N,\pm}}}.
\]
After pulling to the intrinsic bubble chart, the Green identity expresses the flux as
$\int(u\mathscr P_{N,\pm}z-z\mathscr P_{N,\pm}u)$.  Duality between
$\mathcal E_{*,T}$ and $\mathcal E_{*,T}'$ proves
\eqref{eq:bubble-bilinear-bound-step21}.  For a remote response the
interior projected source has the same $\lambda^{a_\Sigma}$ bound as its
coefficient vector, so its $\mathfrak E$ size is
$O(\lambda^{a_\Sigma}\|F_N\|)$.  The finite principal response and its
interior source have $\mathfrak E$ size $O(A_N)$ by
Lemma~\ref{lem:step19}.  This proves the stated pairing consequences, using
$2a_\Sigma=n-4$.

For the tail estimate, Lemma~\ref{aux:tail-dual-rigorous} gives
\[
 \sum_\pm\|F_{N,R,\pm}^{\rm tail}\|_{\mathcal E_{*,T_{N,\pm}}'}
 +\sum_\pm\|\mathfrak DF_{N,R,\pm}^{\rm tail}\|_{\mathcal E_{*,T_{N,\pm}}'}
 \le CA_N\omega_0(R).
\]
Proposition~\ref{aux:localized-energy-inverse-rigorous} yields
\[
 \sum_\pm\|u_{N,R,\pm}^{\rm tail}\|_{\mathcal E_{*,T_{N,\pm}}}
 +|\mathbf d_{N,R}|\le CA_N\omega_0(R).
\]
The finite principal response and source have energy and dual-energy norms
bounded by $CA_N$.  Applying \eqref{eq:bubble-bilinear-bound-step21} and
duality gives \eqref{eq:principal-tail-pairing-step21}, after replacing
$\omega_0$ by a fixed multiple.  The differentiated estimate follows from
the differentiated auxiliary results.

Finally, \eqref{eq:actual-relative-divergence-form-step21} follows from the
conformal Schouten formula and
$dv_{\bar g_N}=(\bar v_{N,\lambda,\xi})^{p_*}dv_g$.  The contracted
Bianchi identity gives $\bar\nabla_iT_N^{ij}=0$, and integration by parts
proves \eqref{eq:actual-green-identity-step21}.  This completes the proof.
\end{proof}

\begin{lemma}[Energy localization of the principal response]
\label{aux:principal-response-energy-localization-gapfix}
Let $(w_N^\sharp,\mathbf b_N)$ be the global projected response from
Lemma~\ref{lem:step20} to the paired principal source
\[
 -(P_{N,-}^\sharp+P_{N,+}^\sharp).
\]
Set
\[
 W_N^\sharp:=W_{N,\xi',\lambda'}^\sharp,
\]
and let $\boldsymbol\beta_N$ be its finite local coefficient vector from
\eqref{eq:finite-principal-response-problem-step19}.  Denote the intrinsic
copies of $W_N^\sharp$ on the two exact bubble portions by
$W_{N,\pm}^\sharp$.  Then
\begin{align}\label{eq:principal-response-energy-localization-gapfix}
 &\sum_\pm\Bigl(
 \|w_{N,\pm}^\sharp-A_NW_{N,\pm}^\sharp\|_{\mathcal E_{*,T_{N,\pm}}}
 \nonumber\\
 &\hspace{22mm}+
 \|\mathscr P_{N,\pm}(w_{N,\pm}^\sharp-A_NW_{N,\pm}^\sharp)
   \|_{\mathcal E_{*,T_{N,\pm}}'}\Bigr)
 +|\mathbf b_N-A_N\boldsymbol\beta_N|=o(A_N),
\end{align}
uniformly on the normalized parameter set and with one normalized parameter
derivative.
\end{lemma}

\begin{proof}
Fix \(R>1\) and split the paired principal source into its part supported in
\(|z|\le2R\) and its tail.  For the fixed-core part, the corresponding
complete-round response is uniformly smooth at the second stereographic
pole, together with one normalized parameter derivative.  Cut it off in a
fixed terminal collar of each truncated bubble and transplant the paired
field to the global block.  Its cap trace and the resulting commutator are
\(o_R(A_N)\) in the bubble energy and dual-energy norms.  On the fixed core,
the intrinsic actual operator, source normalization, and localized
constraints converge to their round counterparts; outside that core, but
inside the exact principal bubble, support separation leaves only an
\(o(1)\) model-relative coefficient error.  The transplanted pair is
therefore an approximate actual augmented solution with total graph-norm
defect \(o_R(A_N)\).  Applying
\eqref{eq:localized-energy-graph-inverse-rigorous} gives
\begin{align*}
 &\sum_\pm\Bigl(
 \|w_{N,R,\pm}^\sharp-A_NW_{N,R,\pm}^\sharp\|_{\mathcal E_{*,T_{N,\pm}}}
 \nonumber\\
 &\hspace{18mm}+\|\mathscr P_{N,\pm}(w_{N,R,\pm}^\sharp
       -A_NW_{N,R,\pm}^\sharp)\|_{\mathcal E_{*,T_{N,\pm}}'}\Bigr)
 +|\mathbf b_{N,R}-A_N\boldsymbol\beta_{N,R}|=o_R(A_N).
\end{align*}

For the tail, Lemma~\ref{aux:tail-dual-rigorous} gives the bound
\(CA_N\omega_0(R)\), with \(\omega_0(R)\to0\), in both the complete-round
and truncated-bubble dual norms.  The localized global graph inverse and
the absolute round inverse give the same bound for the corresponding
responses and coefficient vectors.  Letting first \(N\to\infty\) for fixed
\(R\), and then \(R\to\infty\), proves
\eqref{eq:principal-response-energy-localization-gapfix}.  The argument is
unchanged after one normalized parameter derivative, because the cap
cutoff is fixed in the pulled-back terminal collar and all operator and
source convergences above are \(C^1\).
\end{proof}

\begin{lemma}[Energy localization of the nonlinear difference]
\label{aux:nonlinear-difference-energy-step21}
Let
\[
 d_N=w_{N,\lambda,\xi}-w_N^\sharp,
 \qquad
 \mathbf e_N=\mathbf c_N-\mathbf b_N,
\]
and let $d_{N,\pm}$ be the intrinsic pullbacks to the two exact bubble
portions.  Then, uniformly on the normalized parameter set and with one
normalized parameter derivative,
\begin{align}\label{eq:nonlinear-difference-energy-small-step21}
 \sum_\pm\Bigl(
  \|d_{N,\pm}\|_{\mathcal E_{*,T_{N,\pm}}}
 +\|\mathscr P_{N,\pm}d_{N,\pm}\|_{\mathcal E_{*,T_{N,\pm}}'}\Bigr)
 +|\mathbf e_N|=o(A_N).
\end{align}
\end{lemma}

\begin{proof}
Use the augmented difference equation
\eqref{eq:difference-augmented-equation-step20}.  For fixed \(R>1\), its
source is the sum of a fixed-core second-order remainder, a paired principal
tail in \(|z|>R\), terms supported outside the two exact bubble portions,
and the nonlinear difference.  The first three classes have, respectively,
dual-energy or propagated bubble-energy bounds
\[
 CA_N^2,\qquad CA_N\omega_0(R),\qquad
 C\lambda^{a_\Sigma}\mathcal E_N,
 \quad
 \mathcal E_N=\eta_{L,N_0}+C_L\lambda^{2-8/n}
              +C\epsilon_N^{\kappa_f},
\]
by Lemma~\ref{aux:tail-dual-rigorous} and
Proposition~\ref{prop:localized-bubble-energy-step21}.

We next record the localization estimate needed for the nonlinear term.
On a truncated intrinsic bubble \(\mathcal B_T=\{t<T\}\), let
\(\mathfrak Q_T(w_1,w_2)\) be the pullback of
\(Q_N(w_1)-Q_N(w_2)\), put \(d=w_1-w_2\), and set
\(R_{12}=\|w_1\|_{\widehat X_{\gamma,\beta}}
       +\|w_2\|_{\widehat X_{\gamma,\beta}}\).  Then
\begin{align}\label{eq:nonlinear-energy-lipschitz-cap-step21}
 \|\mathfrak Q_T(w_1,w_2)\|_{\mathcal E_{*,T}'}
 \le{}& C_{Q,E}R_{12}\|d\|_{\mathcal E_{*,T}}
 +Ce^{-nT/2}R_{12}
       \|d\|_{C^{2,\alpha}(\mathcal B_T,g_*)}.
\end{align}
Indeed, choose a cutoff equal to one on \(t\le T-2\) and zero on
\(t\ge T-1\).  Multiplication by this cutoff, followed by extension by
zero, is uniformly bounded from \(\mathcal E_{*,T}\) to
\(H^1(\mathbb S^n,g_*)\); this follows from the weighted Hardy inequality
at the second stereographic pole.  Testing the integral Hessian identity
for the energy against the cut-off function and using
\eqref{eq:quadratic-energy-form-step16} gives the first term.  On the
terminal collar the pointwise quadratic estimate and
\(\operatorname{Vol}_{g_*}\{T-2<t<T\}=O(e^{-nT})\) give the second.
After fixed-domain identification, differentiation yields the analogous
estimate with one derivative on \(d\), \(w_1\), or \(w_2\).

Apply this estimate with \(w_1=w_N\) and \(w_2=w_N^\sharp\).  Since
\(e^{-nT_{N,\pm}/2}\asymp_L\lambda^{a_\Sigma}\), the collar term is part
of the remote capacity error.  If
\[
 D_N^E=\sum_\pm\|d_{N,\pm}\|_{\mathcal E_{*,T_{N,\pm}}}
        +|\mathbf e_N|,
\]
Propositions~\ref{aux:localized-energy-inverse-rigorous} and~\ref{prop:localized-bubble-energy-step21} give
\[
 D_N^E\le C\bigl(A_N^2+A_N\omega_0(R)
       +\lambda^{a_\Sigma}\mathcal E_N\bigr)
 +C_{Q,E}(\|w_N\|_{\widehat X}+\|w_N^\sharp\|_{\widehat X})D_N^E
 +o(A_N).
\]
The last linear term is absorbed by
\eqref{eq:final-energy-radius-step16} and
\eqref{eq:choose-N-principal-radius-step17}.  Dividing by \(A_N\), then
letting \(N\to\infty\) and \(R\to\infty\), using
\(\lambda^{a_\Sigma}/A_N\asymp\epsilon_N^{\beta_f}\to0\), gives
\(D_N^E=o(A_N)\).

Finally, the intrinsic difference equation and
\eqref{eq:nonlinear-energy-lipschitz-cap-step21} show that the
operator-dual terms are also \(o(A_N)\), proving
\eqref{eq:nonlinear-difference-energy-small-step21}.  Differentiating the
same equation on the fixed domains gives the asserted \(C^1\) estimate;
the term containing \(\mathfrak Dd_N\) is absorbed by the same fixed
energy radius, while all remaining differentiated terms are already
\(o(A_N)\).
\end{proof}

\begin{proposition}[Global-to-local energy comparison]
\label{prop:global-local-energy-step21}
Let $v_{N,\lambda,\xi}$ be the projected solution of
Lemma~\ref{lem:step17}, and let
$\mathscr E_{N,\pm}^{\rm loc}$ be the two principal local energies from
Lemma~\ref{lem:step19}.  There is a fixed integer $C_0\ge0$, independent of
$N$, such that, uniformly with one normalized parameter derivative,
\begin{align}
 &\mathscr E_g(v_{N,\lambda,\xi})
  -\mathscr E_{N,-}^{\rm loc}
  -\mathscr E_{N,+}^{\rm loc}
 \nonumber\\
 &\qquad=O_{C^1}\!\left(
      \lambda^{n-4}(1+|\log\lambda|)^{C_0}\right)
      +o_{C^1}(s_N).
 \label{eq:global-local-energy-comparison-step21}
\end{align}
The $O(\lambda^{n-4}(1+|\log\lambda|)^{C_0})$ term contains the flat
matched transitions, the Schwarzschild region, the two assembly
collars, the remaining metric supports, and their correction pairings.  The
$o(s_N)$ term contains the tail of the principal metric source and all
interactions between a principal core and a source supported away
from that core.
\end{proposition}

\begin{proof}
All estimates are made in the natural coordinates of the corresponding
blocks.  We separate the principal bubbles, the tails and direct exterior
energy, and the correction on the remaining blocks.

On the two exact bubble portions, let
\[
 \widetilde d_{N,\pm}
 =w_{N,\lambda,\xi,\pm}-A_NW_{N,\pm}^\sharp,
 \qquad
 \widetilde{\mathbf e}_N
 =\mathbf c_N-A_N\boldsymbol\beta_N.
\]
Lemmas~\ref{aux:principal-response-energy-localization-gapfix} and~\ref{aux:nonlinear-difference-energy-step21} give
\[
 \sum_\pm\Bigl(
 \|\widetilde d_{N,\pm}\|_{\mathcal E_{*,T_{N,\pm}}}
 +\|\mathscr P_{N,\pm}\widetilde d_{N,\pm}\|_{
       \mathcal E_{*,T_{N,\pm}}'}\Bigr)
 +|\widetilde{\mathbf e}_N|=o_{C^1}(A_N).
\]
The first variation at the finite local principal configuration equals the
localized projected source plus an \(o(A_N)\) dual-energy remainder.  After
summing the two ends, the projected part pairs to zero because
\(\mathfrak L_a(w_N)=0\) and
\(\ell_a^*(W_N^\sharp)=0\).  The remaining first-variation term and the
interface flux are \(o_{C^1}(A_N^2)\) by the graph estimate above and
\eqref{eq:bubble-bilinear-bound-step21}.  The integral Taylor formula,
with the uniformly bounded Hessian form in the Newton cone, then shows that
the global energy on the paired principal bubble portions differs from the
two truncated local energies by \(o_{C^1}(s_N)\).

The omitted round caps contribute \(O(\lambda^{n-4})\), while the metric and
response tails of the local models are \(o_{C^1}(s_N)\) by
Lemma~\ref{lem:step19} and Lemma~\ref{aux:tail-dual-rigorous}.  For the
principal metric tail itself, the Schwarzschild Newton scale and
\[
 |D^2h_N|+r^{-1}|Dh_N|+r^{-2}|h_N|
 \le C\mu_Nr^{2m_f}
\]
give
\[
 C\mu_N\lambda^{n-4}
 \int_{\sigma_N^\sharp}^{\rho_N}r^{2m_f+3-n/2}\,dr
 \le C\mu_N\lambda^{n-4}
       (\sigma_N^\sharp)^{2m_f+4-n/2}=o(s_N).
\]
Here \(\sigma_N^\sharp\asymp\epsilon_N^{4/n}\) and
\(n>4m_f+8\); the same estimate holds after one normalized parameter
derivative.

The two matched transitions contribute \(O_L(\lambda^{n-4})\) by
Lemma~\ref{lem:step6}; the Schwarzschild pieces have zero curvature density
and target-volume contribution \(O_L(\lambda^{n-4})\); and the assembly
collars are controlled by Proposition~\ref{prop:paired-assembly-collars}.
On each remaining metric support, Propositions~\ref{lem:step2} and~\ref{lem:step3} give
\(C\lambda^{n-4}\mu_M(1+M)^{P_1}\) for a fixed \(P_1\), and the sum over
\(M\) converges.  Allowing for the number of cylinder blocks and moving
interfaces, these direct terms are
\[
 O_{C^1}\!\left(\lambda^{n-4}
 (1+|\log\lambda|)^{C_0}\right)+o_{C^1}(s_N).
\]

It remains to estimate the correction energy away from the principal
cores.  On every Newton-elliptic unit block, Taylor expansion and the
Newton-strong norm give
\[
 |\Delta\mathfrak e_Q|
 \le C\int_Q\mathcal V^4q^2
 \left(|f_Q|\,\mathfrak X_Q(w)+\mathfrak X_Q(w)^2
       +\mathfrak X_Q(w)^3\right)\,dt\,d\theta.
\]
The low modes are summed with the \(L_t^1\) component of
\(\widehat Y_{\gamma,\beta}\); the radial quadratic terms use
\eqref{eq:exact-relative-radial-residual-step16}, and all terms containing
a nonradial factor use the \(L_t^\infty\times L_t^1\) estimate of
Lemma~\ref{lem:step16}.  Since
\[
 \sum_{Q\subset\mathcal S\cup\mathcal T}
 \int_Q\mathcal V^4q^2\,dt\,d\theta
 \le C_L\lambda^{n-4}(1+|\log\lambda|),
\]
these contributions have the same bound as the direct terms.  Projected
sources pair to zero with the constraints, and principal--remote pairings
are \(o_{C^1}(s_N)\) by Proposition~\ref{prop:localized-bubble-energy-step21}
and
\[
 \lambda^{a_\Sigma}/A_N\asymp\epsilon_N^{\beta_f}\to0,
 \qquad
 \lambda^{n-4}/s_N\asymp\epsilon_N^{2\beta_f}\to0.
\]
Summing over a fixed-overlap partition proves
\eqref{eq:global-local-energy-comparison-step21}.  The same estimates in
the fixed natural coordinates give the normalized parameter derivative.
\end{proof}

\begin{lemma}[Two-ended reduced-energy expansion]\label{lem:step21}
Uniformly on the normalized parameter set, all nonprincipal blocks and all
interactions between distinct source regions contribute $o_{C^1}(s_N)$.
More quantitatively, for the fixed integer $C_0$ in
Proposition~\ref{prop:global-local-energy-step21},
\begin{align}
 \mathscr J_N(\lambda',\xi')
 ={}&2\mathscr E_2(\mathbb S^n)
  +2s_N\mathcal F(\xi',\lambda')
 \nonumber\\
 &+O_{C^1}\!\left(
   \lambda^{n-4}(1+|\log\lambda|)^{C_0}\right)
  +o_{C^1}(s_N).
 \label{eq:quantitative-two-ended-energy-step21}
\end{align}
In particular,
\begin{equation}\label{eq:two-ended-energy-step21}
 \mathscr J_N(\lambda',\xi')
 =2\mathscr E_2(\mathbb S^n)
  +2s_N\mathcal F(\xi',\lambda')+o_{C^1}(s_N).
\end{equation}
The two assembly collars separately have energy
$O_{C^1}(\lambda^{n-4})=o_{C^1}(s_N)$.
\end{lemma}

\begin{proof}
Proposition~\ref{prop:global-local-energy-step21} and the two local
expansions of Lemma~\ref{lem:step19} give
\eqref{eq:quantitative-two-ended-energy-step21}.  Finally,
\[
 \frac{\lambda^{n-4}(1+|\log\lambda|)^{C_0}}{s_N}
 \asymp
 \epsilon_N^{\,n-4-(n/2+2m_f)}
 (1+|\log\epsilon_N|)^{C_0}
 =\epsilon_N^{2\beta_f}
 (1+|\log\epsilon_N|)^{C_0}\longrightarrow0,
\]
because $\beta_f>0$ and $\lambda/\epsilon_N$ stays in a fixed compact
subset of $(0,\infty)$.  This proves
\eqref{eq:two-ended-energy-step21}.  The collar statement follows from the assembly-collar estimate in the proof
of Proposition~\ref{prop:global-local-energy-step21}.  The proof is
complete.
\end{proof}

\section{Finite-dimensional step and completion}

All preceding analytic statements were proved for an arbitrary nonzero
algebraic Weyl tensor and for an arbitrary polynomial degree satisfying
$n>4m_f+8$.  We now specialize to $m_f=4$ and choose a four-dimensional
self-dual Weyl block, extended by zero in the remaining directions.  Put
\begin{equation}\label{eq:W-center-matrix-step22}
 |W|^2:=\sum_{i,j,k,l}(W_{ikjl}+W_{iljk})^2,
 \qquad
 \mathscr W_{ij}:=\sum_{k,p,q}
 (W_{ikpq}+W_{kqip})(W_{jkpq}+W_{kqjp}),
\end{equation}
and normalize the block so that
\begin{equation}\label{eq:W-normalization-step22}
 |W|^{-2}\mathscr W
 =\operatorname{diag}\!\left(\frac14,\frac14,\frac14,\frac14,
 0,\ldots,0\right).
\end{equation}
All preceding estimates remain valid under this specialization.

The one-ended functional $\F$ of Lemma~\ref{lem:step19} is exactly the
functional appearing in the unit-bubble coefficient calculation.  Indeed,
the additive principal response used in that calculation is
$aU_{\lambda',\xi'}W_{\xi',\lambda'}$, so its pairing
$2\int \mathcal R U^{-1}(aUW)$ is precisely the term
$2a\int\mathcal R W$ in
\eqref{eq:local-functional-simplified-step19}.  Hence there are scalar
functions $P_f,Q_f,R_f$ such that, for $t>0$,
\begin{align}
 \F(0,\sqrt t)&=|W|^2P_f(t),\label{eq:F-scale-P-step22}\\
 \partial_{\xi_i'}\F(0,\sqrt t)&=0,\label{eq:F-first-center-step22}\\
 \partial_{\xi_i'\xi_j'}^2\F(0,\sqrt t)
 &=|W|^2Q_f(t)\delta_{ij}+R_f(t)\mathscr W_{ij}.
 \label{eq:finite-dimensional-coefficients-step22}
\end{align}
For clarity, the principal-response term in
\eqref{eq:local-functional-simplified-step19} does not change the scale-axis
value or the center Hessian in these identities.  Indeed, in the
metric-centered variable the first source has the exact form
\[
 \mathcal R_{\xi',\lambda'}(y)
 =n(n-1)\frac{U_{\lambda',\xi'}(y)^{4(n-2)/(n-4)}}
 {\bigl((\lambda')^2+|y-\xi'|^2\bigr)^2}
 \widehat h_{ij}(y)(y_i-\xi_i')(y_j-\xi_j').
\]
Since $y_i\widehat h_{ij}(y)=0$, the last contraction equals
$\widehat h_{ij}(y)\xi_i'\xi_j'$.  Thus
$\mathcal R_{\xi',\lambda'}=O(|\xi'|^2)$ in the projected source norm and
$W_{\xi',\lambda'}=O(|\xi'|^2)$.  The response contribution to $\F$ is
therefore $O(|\xi'|^4)$.  The remaining pure-metric angular contractions and
radial beta integrals give \eqref{eq:F-scale-P-step22}--
\eqref{eq:finite-dimensional-coefficients-step22}.

For a parameter $a$, set
\begin{align}\label{def-unified-quartic-family}
        f_a(s):=a-5960s+1181s^2-73s^3+s^4.
\end{align}
The constant term is chosen as a function of $n$.

Remove the common positive factor
\begin{align}\label{def-quartic-common-factor}
        \Theta_n
        :=2^n\omega_{n-1}
        \frac{\Gamma(n/2)\Gamma(n/2-12)}{\Gamma(n-4)},
        \qquad n\ge27,
\end{align}
and write
\[
        \widehat P_f:=\Theta_n^{-1}P_f,
        \qquad
        \widehat Q_f:=\Theta_n^{-1}Q_f,
        \qquad
        \widehat R_f:=\Theta_n^{-1}R_f.
\]
The exact beta-integral evaluation of the coefficient identities
\eqref{eq:finite-dimensional-coefficients-step22}, with the final coefficient
$24t^6\mathcal I_{n-3,n+9}^{23}$ in the $R_f^{[-3]}$ contribution, gives
\begin{align}\label{quartic-scale-equation-all-n}
        \widehat P_{f_a}'(1)
        =\frac{A_na^2+B_na+C_n}
        {65536(n-3)(n-2)^2},
\end{align}
where
\begin{align}\label{quartic-ABC-all-n}
A_n={}&-(n-24)(n-22)(n-20)(n-18)(n-16)(n-14)(n-12)(n-10)(n-2)^2,
\notag\\
B_n={}&(n-24)(n-22)(n-20)(n-18)(n-2)(n+4)
\notag\\
&\quad\times
(13515n^4-684494n^3+12499560n^2-94689040n+232611456),
\notag\\
C_n={}&-24(n+4)
\bigl(1751211n^9-220009697n^8+11540807010n^7
\notag\\
&\quad-321660897252n^6+4907687372744n^5
-34315210114528n^4
\notag\\
&\quad-38624210366720n^3+2074854736330752n^2
-9720170316011520n
\notag\\
&\quad+9055510528204800\bigr).
\end{align}
Its discriminant is
\begin{align}\label{quartic-discriminant-all-n}
        \Delta_n
        =(n-24)(n-22)(n-20)(n-18)(n-2)^2(n+4)d_{13}(n).
\end{align}
Since $A_n<0$ for $n\ge27$, define the single algebraic branch
\begin{align}\label{def-a0-all-n-quartic}
        a_0(n):=\frac{-B_n-\sqrt{\Delta_n}}{2A_n},
        \qquad
        f_n:=f_{a_0(n)}.
\end{align}
The coefficient computation below shows that $\Delta_n>0$ and that all
stability inequalities have the required sign on this branch.

\begin{lemma}[Stable negative quartic profile]\label{lem:step22}
For every integer $n\ge27$, the polynomial $f_n$ in
\eqref{def-a0-all-n-quartic} has degree four, satisfies
$f_n(0)=a_0(n)>0$, and, at the common scale $t_0=1$,
\begin{align}\label{quartic-all-n-signs}
        P_{f_n}'(1)=0,
        \qquad P_{f_n}(1)<0,
        \qquad P_{f_n}''(1)>0,
\end{align}
and
\begin{align}\label{quartic-all-n-center-signs}
        Q_{f_n}(1)>0,
        \qquad
        Q_{f_n}(1)+\frac14R_{f_n}(1)>0.
\end{align}
Consequently, $f_n$ is a stable negative profile for every $n\ge27$.
\end{lemma}

\begin{proof}
Equation \eqref{quartic-scale-equation-all-n} and the definition of $a_0(n)$
give $P_{f_n}'(1)=0$.  Moreover, $A_n<0$, $B_n>0$, and $C_n<0$ for
$n\ge27$.  For the last two signs, after writing $n=z+27$, the remaining
quartic factor in $B_n$ and the degree-nine factor in $-C_n$ have strictly
positive integer coefficients.  Hence the two roots of
$A_na^2+B_na+C_n=0$ have positive product $C_n/A_n$ and positive sum
$-B_n/A_n$.  In particular, the branch $a_0(n)$ is positive, proving
$f_n(0)\ne0$.  We prove the remaining signs by exact polynomial identities.  Reducing the numerators of
\[
        -\widehat P_{f_a}(1),\qquad
        \widehat P_{f_a}''(1),\qquad
        \widehat Q_{f_a}(1),\qquad
        \widehat Q_{f_a}(1)+\frac14\widehat R_{f_a}(1)
\]
modulo $A_na^2+B_na+C_n=0$, and then inserting
\eqref{def-a0-all-n-quartic}, gives
\begin{align}\label{quartic-four-sign-forms}
-\widehat P_{f_n}(1)
&=\frac{U_-(n)+V_-(n)\sqrt{\Delta_n}}
{(-2A_n)\,131072(n-3)(n-2)^2},\notag\\
\widehat P_{f_n}''(1)
&=\frac{U_2(n)+V_2(n)\sqrt{\Delta_n}}
{(-2A_n)\,65536(n-3)(n-2)^2},\notag\\
\widehat Q_{f_n}(1)
&=\frac{U_Q(n)+V_Q(n)\sqrt{\Delta_n}}
{(-2A_n)\,32768(n-3)(n-2)^2},\notag\\
\widehat Q_{f_n}(1)+\frac14\widehat R_{f_n}(1)
&=\frac{U_H(n)+V_H(n)\sqrt{\Delta_n}}
{(-2A_n)\,65536(n-3)(n-2)^2}.
\end{align}
The factors in these numerators are
\begin{align}\label{quartic-UV-factors}
U_-={}&(n-24)(n-22)(n-20)(n-18)(n-2)^2(n+4)u_-(n),\notag\\
V_-={}&(n-24)(n-22)(n-20)(n-18)(n-2)(n+4)v_-(n),\notag\\
U_2={}&(n-24)(n-22)(n-20)(n-18)(n-2)^2(n+4)u_2(n),\notag\\
V_2={}&(n-24)(n-22)(n-20)(n-18)(n-2)(n+4)v_2(n),\notag\\
U_Q={}&(n-24)^2(n-22)(n-20)(n-18)(n-16)(n-2)^2u_Q(n),\notag\\
V_Q={}&(n-24)(n-22)(n-20)(n-18)(n-16)(n-2)v_Q(n),\notag\\
U_H={}&2(n-24)^2(n-22)(n-20)(n-18)(n-16)(n-2)^2u_H(n),\notag\\
V_H={}&4(n-24)(n-22)(n-20)(n-18)(n-16)(n-2)v_Q(n).
\end{align}
The remaining integer polynomials are
\begin{align}\label{quartic-positive-polynomials}
d_{13}(n)&=14538969 n^{13} - 3251314596 n^{12}\notag\\
&\quad{} +331657384528 n^{11} - 20370494780432 n^{10}\notag\\
&\quad{} +837191962557936 n^{9} - 24187543505430144 n^{8}\notag\\
&\quad{} +501949979036960256 n^{7} - 7503665337837616128 n^{6}\notag\\
&\quad{} +79647600043952946176 n^{5} - 578502301857857167360 n^{4}\notag\\
&\quad{} +2656029018637785317376 n^{3} - 6277431388462957854720 n^{2}\notag\\
&\quad{} +1361345086315664769024 n +17771898772260409835520,\\[0.3em]
u_-(n)&=14538969 n^{13} - 2816495328 n^{12}\notag\\
&\quad{} +249332537856 n^{11} - 13337844686736 n^{10}\notag\\
&\quad{} +479897741537392 n^{9} - 12214457597957824 n^{8}\notag\\
&\quad{} +224743909765978624 n^{7} - 2993580873617317888 n^{6}\notag\\
&\quad{} +28324137894637139968 n^{5} - 181310351032624672768 n^{4}\notag\\
&\quad{} +698802194459424276480 n^{3} - 1048758450300727197696 n^{2}\notag\\
&\quad{} - 2308895635803601895424 n +8674110855699108986880,\\[0.3em]
v_-(n)&=3813 n^{4} - 205546 n^{3}\notag\\
&\quad{} +4003032 n^{2} - 32348240 n\notag\\
&\quad{} +84493824,\\[0.3em]
u_2(n)&=7157001 n^{13} - 1630974500 n^{12}\notag\\
&\quad{} +175300241768 n^{11} - 11631669601712 n^{10}\notag\\
&\quad{} +524621815120944 n^{9} - 16761301113579264 n^{8}\notag\\
&\quad{} +385240397312255616 n^{7} - 6362404296206497792 n^{6}\notag\\
&\quad{} +74270532769771466752 n^{5} - 590490539032576638976 n^{4}\notag\\
&\quad{} +2962955825001135046656 n^{3} - 7768082060472667078656 n^{2}\notag\\
&\quad{} +3278826490704930275328 n +20831217932333576355840,\\[0.3em]
v_2(n)&=9503 n^{4} - 550326 n^{3}\notag\\
&\quad{} +11495584 n^{2} - 99298560 n\notag\\
&\quad{} +275277312,\\[0.3em]
u_Q(n)&=14538969 n^{12} - 2249250570 n^{11}\notag\\
&\quad{} +156894317496 n^{10} - 6523125457640 n^{9}\notag\\
&\quad{} +179933260949408 n^{8} - 3463150236309536 n^{7}\notag\\
&\quad{} +47483610437351552 n^{6} - 462787506994390400 n^{5}\notag\\
&\quad{} +3118260427749099776 n^{4} - 13574822156040754176 n^{3}\notag\\
&\quad{} +32529103269544120320 n^{2} - 22452479645214523392 n\notag\\
&\quad{} -38809785197822607360,\\[0.3em]
v_Q(n)&=3813 n^{4} - 171170 n^{3}\notag\\
&\quad{} +2730300 n^{2} - 17330680 n\notag\\
&\quad{} +30447552,\\[0.3em]
u_H(n)&=14538969 n^{12} - 2248129216 n^{11}\notag\\
&\quad{} +158542976888 n^{10} - 6782600382248 n^{9}\notag\\
&\quad{} +197202121747696 n^{8} - 4117952286987328 n^{7}\notag\\
&\quad{} +63113260476922624 n^{6} - 705868833869469312 n^{5}\notag\\
&\quad{} +5570698047508768256 n^{4} - 28915991673618913280 n^{3}\notag\\
&\quad{} +85172838429775601664 n^{2} - 87175632491723096064 n\notag\\
&\quad{} -97192216767831736320.
\end{align}
After the shift $n=z+27$, direct expansion gives
\begin{align}\label{quartic-positive-coefficient-certificate}
&d_{13}(z+27),\ u_-(z+27),\ v_-(z+27),\notag\\
&\hspace{12mm}u_2(z+27),\ v_2(z+27),\notag\\
&\hspace{12mm}u_Q(z+27),\ v_Q(z+27),\ u_H(z+27)
 \in\mathbb Z_{>0}[z].
\end{align}
Hence all eight polynomials in \eqref{quartic-positive-polynomials} are positive
for $n\ge27$.  It follows first from \eqref{quartic-discriminant-all-n} that
$\Delta_n>0$, and then from \eqref{quartic-UV-factors} that every $U$ and $V$ in
\eqref{quartic-four-sign-forms} is positive.  The denominators in
\eqref{quartic-four-sign-forms} are positive because $A_n<0$.  This proves all
four strict inequalities.

At $n=27$, the same formula specializes to
\[
        a_0(27)
        =\frac{467551703+\sqrt{456713347098449}}{55250},
\]
which is the value of the algebraic branch in dimension $27$.
\end{proof}

\begin{lemma}[Reduced critical point]\label{lem:step23}
For all sufficiently large $N$, the reduced energy has an interior critical
point satisfying
\[
 \lambda_N=\epsilon_N(1+o(1)),
 \qquad \xi_N=o(\epsilon_N),
\]
and the corresponding normalized energy is strictly below the two-bubble
level.
\end{lemma}

\begin{proof}
Define
\begin{equation}\label{eq:normalized-reduced-functional-step23}
 \widehat{\mathscr J}_N(\lambda',\xi')
 :=\frac{\mathscr J_N(\lambda',\xi')
       -2\mathscr E_2(\mathbb S^n)}{2s_N}.
\end{equation}
By Lemma~\ref{lem:step21},
\(\widehat{\mathscr J}_N=\mathcal F+o_{C^1}(1)\) uniformly on the
normalized parameter set.  The formulas in Section~8 and
Lemma~\ref{lem:step22} give
\[
 \partial_{\lambda'}\mathcal F(0,1)=0,
 \qquad
 \partial_{\lambda'\lambda'}^2\mathcal F(0,1)
 =4|W|^2P_{f_n}''(1)>0,
\]
all mixed scale--center derivatives vanish, and
\[
 D_{\xi'\xi'}^2\mathcal F(0,1)
 =|W|^2Q_{f_n}(1)I+R_{f_n}(1)\mathscr W.
\]
Under \eqref{eq:W-normalization-step22}, the latter matrix has eigenvalues
\(|W|^2Q_{f_n}(1)\) and
\(|W|^2(Q_{f_n}(1)+\frac14R_{f_n}(1))\), both positive.  Thus \((0,1)\)
is a nondegenerate strict local minimum of \(\mathcal F\), and
\(\mathcal F(0,1)=|W|^2P_{f_n}(1)<0\).

Choose a small closed ball \(\overline{B_\rho}\) centered at \((0,1)\),
contained in the interior of the normalized parameter set, on whose
boundary \(\mathcal F\ge\mathcal F(0,1)+3\eta\) for some \(\eta>0\).
Uniform \(C^1\) convergence implies that, for large \(N\),
\(\widehat{\mathscr J}_N\) attains its minimum on this ball at an interior
point \((\xi_N',\lambda_N')\), and the nondegeneracy of the limiting
minimum gives \((\xi_N',\lambda_N')\to(0,1)\).  Hence
\[
 \lambda_N=\epsilon_N\lambda_N'=\epsilon_N(1+o(1)),
 \qquad
 \xi_N=\epsilon_N\xi_N'=o(\epsilon_N).
\]
Finally,
\[
 \mathscr J_N(\lambda_N',\xi_N')
 =2\mathscr E_2(\mathbb S^n)
  +2s_N\bigl(\mathcal F(0,1)+o(1)\bigr)
 <2\mathscr E_2(\mathbb S^n).
\]
Lemma~\ref{lem:step18} shows that the projected coefficients vanish at this
interior critical point.  The corresponding projected solution is therefore
a positive \(\Gamma_2^+\)-admissible solution of the full equation.
\end{proof}

\begin{lemma}[Fixed conformal class and noncompactness]\label{lem:step24}
For the single smooth metric $g$ defined by \eqref{eq:paired-metric}, activating successively the paired perturbations indexed
by $N\to\infty$ produces a sequence of smooth positive $J$-invariant
solutions in the one conformal class $[g]$, with unbounded supremum and
quotient tending from below to $2^{4/n}\mathcal Y_2(\mathbb S^n)$.
\end{lemma}

\begin{proof}
Choose \(L\) and \(N_0\) as in
Proposition~\ref{prop:ordered-parameters-step17}.  The tensor \(h\), and
hence
\[
 g=g_{\rm rd}(e^h\,\cdot,\cdot),
\]
were fixed by \eqref{eq:hplus-hsum} and \eqref{eq:paired-metric} and are
independent of the selected index \(N\).  Since \(f_n(0)\ne0\) by
Lemma~\ref{lem:step22}, the standard curvature calculation at the centers
\(x_N^-\) shows that \(g\) is smooth, \(J\)-invariant, and not locally
conformally flat.

For each sufficiently large \(N\), let
\((\lambda_N',\xi_N')\) be the critical point from
Lemma~\ref{lem:step23} and set
\[
 \lambda_N=\epsilon_N\lambda_N',\qquad
 \xi_N=\epsilon_N\xi_N',\qquad
 v_N=v_{N,\lambda_N,\xi_N}.
\]
By Lemmas~\ref{lem:step18} and \ref{lem:step23}, all projected coefficients
vanish, so
\[
 \mathcal N_g(v_N)=0,
 \qquad
 \sigma_2(g_{v_N}^{-1}A_{g_{v_N}})=\frac14\binom n2,
 \qquad
 g_{v_N}=v_N^{8/(n-4)}g
\]
and \(g_{v_N}\) is \(\Gamma_2^+\)-admissible.  Uniqueness of the projected
fixed point in the \(J\)-invariant space gives \(v_N\circ J=v_N\).
Elliptic regularity yields \(v_N\in C^\infty(\mathbb S^n)\); hence all these
metrics lie in the fixed conformal class \([g]\).

Let \(q_N^-=x_N^-+\xi_N\).  At this point the approximation is the exact
left bubble, and Lemma~\ref{lem:step20} gives
\(w_{N,\lambda_N,\xi_N}(q_N^-)=O(A_N)=o(1)\).  Since
\(q_N^-\to p_-\), \(U_{1,0}(q_N^-)\to2^a\), and
\(\lambda_N=\epsilon_N(1+o(1))\),
\[
 v_N(q_N^-)=\lambda_N^{-a}(1+o(1))
 =\epsilon_N^{-a}(1+o(1))\longrightarrow\infty,
 \qquad a=\frac{n-4}{4}.
\]
Thus \(\sup_{\mathbb S^n}v_N\to\infty\), and the antipodal peak has the
same height.

For a normalized solution, the energy and quotient satisfy
\[
 \mathscr E_g(v)=\frac1n\binom n2\operatorname{Vol}(g_v),
 \qquad
 \mathcal Q_g(v)=\frac14\binom n2\operatorname{Vol}(g_v)^{4/n}.
\]
Lemma~\ref{lem:step21} and \(\mathcal F(0,1)<0\) give
\[
 \mathscr E_g(v_N)
 =2\mathscr E_2(\mathbb S^n)
  +2s_N\bigl(\mathcal F(0,1)+o(1)\bigr),
\]
so \(\operatorname{Vol}(g_{v_N})<2|\mathbb S^n|\) for large \(N\) and
\(\operatorname{Vol}(g_{v_N})\to2|\mathbb S^n|\).  Consequently
\[
 \mathcal Q_g(v_N)<2^{4/n}\mathcal Y_2(\mathbb S^n),
 \qquad
 \mathcal Q_g(v_N)\longrightarrow
 2^{4/n}\mathcal Y_2(\mathbb S^n).
\]
This proves noncompactness in the fixed conformal class and completes the
proof of Theorem~\ref{thm:target}.
\end{proof}

\section*{Acknowledgments}

The first author would like to thank Professor Xinan Ma for his warm
hospitality and helpful discussions during his visit to the University of
Science and Technology of China, where part of this work was carried out.
The research of J. Wei is partially supported by the GRF grant entitled
``On critical and supercritical Fujita equation''.

\subsection*{Declaration of AI-assisted technologies}

The central ideas of this work were conceived by the authors. OpenAI's ChatGPT was used to assist with language,
exposition, and preliminary consistency checks. All mathematical
arguments were independently verified by the authors, who take full
responsibility for the content of the paper.

\appendix
\section{Coefficient formulas for the reduced energy}
\label{app:coefficient-bookkeeping}

This appendix gives the coefficient functions used in
Section~8.  The formulas are one-ended: they compute the functional
$\mathcal F$ associated with the left bubble.  The antipodal end gives the
same contribution exactly.  Consequently the unnormalized two-ended energy
has coefficients $2P_f,2Q_f,2R_f$, whereas the normalized reduced
functional in \eqref{eq:normalized-reduced-functional-step23} has precisely
the coefficients $P_f,Q_f,R_f$ written below.

All tensors in this appendix are computed in Euclidean coordinates, and
indices are raised and lowered by $\delta$.  Let $g_t=e^{th}$, with $h$
symmetric and trace-free, and fix a positive function $u$.  Write
\begin{equation}\label{eq:app-B-expansion}
 B_u(t):=(u^{8/(n-4)}e^{th})^{-1}
 A_{u^{8/(n-4)}e^{th}}
 =B_u^0+t\dot B_u[h]+\frac{t^2}{2}\ddot B_u[h]+O(t^3).
\end{equation}
If $T_u^0$ is the first Newton tensor at $B_u^0$, then the first metric
source is
\begin{equation}\label{eq:app-first-source}
 S_u[h]=u^{p_*}(T_u^0)_{ji}\dot B_u[h]_{ij}.
\end{equation}
For the second variation define the local energy density
\begin{equation}\label{eq:app-local-energy-density}
 \mathfrak e_{e^h}(u)
 :=u^{p_*}\left[
 \sigma_2\bigl((u^{8/(n-4)}e^h)^{-1}
 A_{u^{8/(n-4)}e^h}\bigr)
 -\frac{n-4}{4n}\binom n2\right]
\end{equation}
and put
\begin{equation}\label{eq:app-G-def}
 G_u[h]:=\frac12\left.\frac{d^2}{dt^2}\right|_{t=0}
 \mathfrak e_{e^{th}}(u).
\end{equation}
Then
\begin{equation}\label{eq:app-G-general}
 G_u[h]=\frac12u^{p_*}\left\{
 (T_u^0)_{ji}\ddot B_u[h]_{ij}
 +\bigl(\tr\dot B_u[h]\bigr)^2
 -\tr\bigl((\dot B_u[h])^2\bigr)\right\}.
\end{equation}

For the coefficient computation one must distinguish the scale axis from
the translation directions.  Put
\[
 U=U_{1,0}=\left(\frac2{1+|y|^2}\right)^{(n-4)/4},
 \qquad s=1+|y|^2,
 \qquad w=\log\frac2{1+|y|^2},
\]
and let $E=y\cdot\nabla$ be the Euler vector field.  For a trace-free and
divergence-free symmetric tensor $h$, define
\begin{equation}\label{eq:app-VY}
 V_j[h]=y_a h_{aj},\qquad Y[h]=y_ay_bh_{ab}.
\end{equation}
The full first variation of the conformal Schouten endomorphism at the unit
bubble is
\begin{equation}\label{eq:app-full-Bdot}
 \dot B_U[h]_{ij}
 =\frac{s}{4}\bigl(Eh_{ij}-\partial_iV_j[h]-\partial_jV_i[h]\bigr)
 -\frac{s^2}{8(n-2)}\Delta h_{ij}
 +\frac12Y[h]\delta_{ij}.
\end{equation}
Only on the scale axis do $V[h]$ and $Y[h]$ vanish.  In particular,
\eqref{eq:app-full-Bdot} cannot be replaced by its tangential specialization
before differentiating with respect to the center.

Let $B_U(h,k)$ be the symmetric bilinear form obtained by polarizing
$G_U[h]$.  In terms of \eqref{eq:app-full-Bdot},
\begin{align}
 B_U(h,k)={}&\frac12U^{p_*}\left[
 \tr\dot B_U[h]\,\tr\dot B_U[k]
 -\tr\bigl(\dot B_U[h]\dot B_U[k]\bigr)\right]
 \nonumber\\
 &+\frac18U^{p_*}e^{-2w}\left[
 -\frac12\langle Dh,Dk\rangle
 -2(n-1)\mathscr D(h,k)
 -(n-1)(n-2)\mathscr W_2(h,k)
 \right],
 \label{eq:app-B-bilinear}
\end{align}
where
\begin{align}
 \mathscr D(h,k)
 &:=\partial_i\left[
 \frac12(h_{ia}k_{aj}+k_{ia}h_{aj})w_j\right],
 \label{eq:app-D-bilinear}\\
 \mathscr W_2(h,k)
 &:=\frac12(h_{ia}k_{aj}+k_{ia}h_{aj})w_iw_j,
 \qquad w_i=-\frac{2y_i}{1+|y|^2}.
 \label{eq:app-W2-bilinear}
\end{align}
Formula \eqref{eq:app-B-bilinear} is the full local quadratic density used
for the translation coefficients.

For the scale-axis coefficient, if
$h_{ij}=F(|y|^2)H_{ij}(y)$ with
$H_{ij}(y)=W_{ipjq}y_py_q$, then $V[h]=Y[h]=0$ and
\[
 \dot B_U[h]=\frac{1+|y|^2}{4}Eh
 -\frac{(1+|y|^2)^2}{8(n-2)}\Delta h,
 \qquad \tr\dot B_U[h]=0,
\]
while
\[
 \tr\ddot B_U[h]
 =-\frac{U^{-8/(n-4)}}{4(n-1)}
 \sum_{i,j,k}(\partial_kh_{ij})^2.
\]
Thus
\begin{align}
 G_U[h]={}&G_U^{\partial h}[h]+G_U^{E,E}[h]
 +G_U^{\Delta,E}[h]+G_U^{\Delta,\Delta}[h],
 \label{eq:app-scale-density-split}\\
 G_U^{\partial h}[h]
 &:=-\frac1{16}U^{p_*-8/(n-4)}
 \sum_{i,j,k}(\partial_kh_{ij})^2,\nonumber\\
 G_U^{E,E}[h]
 &:=-\frac1{32}U^{p_*}(1+|y|^2)^2
 \sum_{i,j}(Eh_{ij})^2,\nonumber\\
 G_U^{\Delta,E}[h]
 &:=\frac1{32(n-2)}U^{p_*}(1+|y|^2)^3
 \sum_{i,j}(\Delta h_{ij})(Eh_{ij}),\nonumber\\
 G_U^{\Delta,\Delta}[h]
 &:=-\frac1{128(n-2)^2}U^{p_*}(1+|y|^2)^4
 \sum_{i,j}(\Delta h_{ij})^2.\nonumber
\end{align}
This tangential formula is used only for $P_f$, not for the translation
Hessian.

For $a,b\ge0$, set
\begin{equation}\label{eq:app-I-def}
 F_a(r,t):=f^{(a)}(tr^2),\qquad
 \mathcal I_{\ell,m}^{ab}(t)
 :=\int_0^\infty
 \frac{r^m f^{(a)}(tr^2)f^{(b)}(tr^2)}{(1+r^2)^\ell}\,dr,
\end{equation}
and let $\omega_{n-1}=|\mathbb S^{n-1}|$.  The required angular identities
are collected next.

\begin{lemma}[Angular contractions]\label{lem:app-angular}
With $H_{ij}(y)=W_{ipjq}y_py_q$, on $\partial B_1(0)$ one has
\begin{align*}
 \int\sum_{i,j}H_{ij}^2
 &=\frac{\omega_{n-1}}{2n(n+2)}|W|^2,
 &
 \int\sum_{i,j,k}(\partial_kH_{ij})^2
 &=\frac{\omega_{n-1}}n|W|^2,\\
 \int\sum_sH_{is}H_{js}
 &=\frac{\omega_{n-1}}{2n(n+2)}\mathscr W_{ij},
 &
 \int\sum_{a,b}(\partial_iH_{ab})(\partial_jH_{ab})
 &=\frac{\omega_{n-1}}n\mathscr W_{ij},\\
 \int\sum_{a,b}H_{ab}(\partial_jH_{ab})y_i
 &=\frac{\omega_{n-1}}{n(n+2)}\mathscr W_{ij},&&
 \end{align*}
and
\begin{align*}
 \int\sum_{a,b}H_{ab}^2y_iy_j
 &=\frac{2\omega_{n-1}}{n(n+2)(n+4)}\mathscr W_{ij}
 +\frac{\omega_{n-1}}{2n(n+2)(n+4)}|W|^2\delta_{ij},\\
 \int\sum_{a,b,k}(\partial_kH_{ab})^2y_iy_j
 &=\frac{2\omega_{n-1}}{n(n+2)}\mathscr W_{ij}
 +\frac{\omega_{n-1}}{n(n+2)}|W|^2\delta_{ij}.
\end{align*}
Moreover,
\begin{align*}
 \partial_s(f(|y|^2)H_{ij})
 &=2f'(|y|^2)y_sH_{ij}+f(|y|^2)\partial_sH_{ij},\\
 \partial_{st}(f(|y|^2)H_{ij})
 &=2\delta_{st}f'H_{ij}+4f''y_sy_tH_{ij}
 +2f'y_t\partial_sH_{ij}+2f'y_s\partial_tH_{ij}
 +f\partial_{st}H_{ij}.
\end{align*}
\end{lemma}

For the rescaled tensor on the scale axis one has
\[
 \widehat h_{0,\sqrt t}=tF_0H,
 \qquad E\widehat h_{0,\sqrt t}=2t(F_0+tr^2F_1)H,
\]
\[
 \Delta\widehat h_{0,\sqrt t}
 =t^2\bigl(2(n+4)F_1+4tr^2F_2\bigr)H,
\]
and
\[
 |D\widehat h_{0,\sqrt t}|^2
 =t^2F_0^2|DH|^2+8t^3F_0F_1|H|^2
 +4t^4r^2F_1^2|H|^2.
\]
Substitution into \eqref{eq:app-scale-density-split} and radial integration
give the scale coefficient
\begin{align}
 P_f(t)={}&-\frac{2^{n-6}\omega_{n-1}}{n(n+2)}
 \Bigl[(n+2)t^2\mathcal I_{n-2,n+1}^{00}
 +4t^2\mathcal I_{n-2,n+3}^{00}
 +4t^3\mathcal I_{n-2,n+3}^{01}\nonumber\\
 &\qquad+2t^4\mathcal I_{n-2,n+5}^{11}
 +8t^3\mathcal I_{n-2,n+5}^{01}
 +4t^4\mathcal I_{n-2,n+7}^{11}\Bigr]\nonumber\\
 &+\frac{2^{n-4}\omega_{n-1}}{n(n+2)(n-2)}
 \Bigl[(n+4)t^3\mathcal I_{n-3,n+3}^{01}
 +(n+4)t^4\mathcal I_{n-3,n+5}^{11}\nonumber\\
 &\qquad+2t^4\mathcal I_{n-3,n+5}^{02}
 +2t^5\mathcal I_{n-3,n+7}^{12}\Bigr]\nonumber\\
 &-\frac{2^{n-6}\omega_{n-1}}{n(n+2)(n-2)^2}
 \Bigl[(n+4)^2t^4\mathcal I_{n-4,n+3}^{11}
 +4(n+4)t^5\mathcal I_{n-4,n+5}^{12}
 +4t^6\mathcal I_{n-4,n+7}^{22}\Bigr].
 \label{eq:app-P-formula}
\end{align}
Here and below all $\mathcal I$'s are evaluated at $t$.

For the center directions put
\begin{align*}
 h^0_{ab}&=tF_0H_{ab},\\
 h^p_{ab}&=\sqrt t\bigl(2ty_pF_1H_{ab}+F_0\partial_pH_{ab}\bigr),\\
 h^{pq}_{ab}&=4t^2y_py_qF_2H_{ab}+2t\delta_{pq}F_1H_{ab}
 +2tF_1(y_p\partial_qH_{ab}+y_q\partial_pH_{ab})
 +tF_0\partial_{pq}H_{ab}.
\end{align*}
The center Hessian is
\begin{equation}\label{eq:app-center-Hessian}
 \mathscr H_{pq}(t)
 :=\int_{\mathbb R^n}
 \left[2B_U(h^p,h^q)+B_U(h^0,h^{pq})+B_U(h^{pq},h^0)\right]dy
 =|W|^2Q_f(t)\delta_{pq}+R_f(t)\mathscr W_{pq}.
\end{equation}
The angular contractions of Lemma~\ref{lem:app-angular} yield
\begin{equation}\label{eq:app-Q-split}
 Q_f(t)=\omega_{n-1}\bigl(Q_f^{[-4]}(t)
 +Q_f^{[-3]}(t)+Q_f^{[-2]}(t)\bigr),
\end{equation}
where
\begingroup\small
\begin{align}
 Q_f^{[-4]}={}&-\frac{2^{n-4}}{n(n-2)^2(n+2)(n+4)}
 \Bigl[(n+4)^2(n+6)t^4\mathcal I_{n-4,n+3}^{12}
 +4(n+4)(n+6)t^5\mathcal I_{n-4,n+5}^{13}\nonumber\\
 &\quad+4(n+5)(n+6)t^5\mathcal I_{n-4,n+5}^{22}
 +4(n+4)t^6\mathcal I_{n-4,n+7}^{14}
 +16(n+6)t^6\mathcal I_{n-4,n+7}^{23}\nonumber\\
 &\quad+8t^7\mathcal I_{n-4,n+9}^{24}
 +8t^7\mathcal I_{n-4,n+9}^{33}\Bigr],
 \label{eq:app-Q-minus4}\\
 Q_f^{[-3]}={}&\frac{2^{n-3}}{n(n-2)(n+2)(n+4)}
 \Bigl[(n+4)(n+6)t^3\mathcal I_{n-3,n+3}^{02}
 +(n+4)^2t^3\mathcal I_{n-3,n+3}^{11}\nonumber\\
 &\quad+4(n+6)t^4\mathcal I_{n-3,n+5}^{03}
 +2(n+5)(n+10)t^4\mathcal I_{n-3,n+5}^{12}
 +4t^5\mathcal I_{n-3,n+7}^{04}\nonumber\\
 &\quad+2(3n+22)t^5\mathcal I_{n-3,n+7}^{13}
 +2(3n+20)t^5\mathcal I_{n-3,n+7}^{22}
 +4t^6\mathcal I_{n-3,n+9}^{14}
 +12t^6\mathcal I_{n-3,n+9}^{23}\Bigr],
 \label{eq:app-Q-minus3}\\
 Q_f^{[-2]}={}&-\frac{2^{n-4}}{n(n+2)(n+4)}
 \Bigl[(n+2)(n+4)t^2\mathcal I_{n-2,n+1}^{01}
 +4(n+4)t^2\mathcal I_{n-2,n+3}^{01}\nonumber\\
 &\quad+4(n+4)t^3\mathcal I_{n-2,n+3}^{02}
 +5(n+4)t^3\mathcal I_{n-2,n+3}^{11}
 +4(n+8)t^3\mathcal I_{n-2,n+5}^{02}\nonumber\\
 &\quad+4t^4\mathcal I_{n-2,n+5}^{03}
 +2(2n+17)t^3\mathcal I_{n-2,n+5}^{11}
 +2(n+14)t^4\mathcal I_{n-2,n+5}^{12}\nonumber\\
 &\quad+8t^4\mathcal I_{n-2,n+7}^{03}
 +4(n+14)t^4\mathcal I_{n-2,n+7}^{12}
 +4t^5\mathcal I_{n-2,n+7}^{13}
 +4t^5\mathcal I_{n-2,n+7}^{22}\nonumber\\
 &\quad+8t^5\mathcal I_{n-2,n+9}^{13}\nonumber\\
 &\quad+8t^5\mathcal I_{n-2,n+9}^{22}\Bigr].
 \label{eq:app-Q-minus2}
\end{align}
\endgroup

Similarly,
\begin{equation}\label{eq:app-R-split}
 R_f(t)=\omega_{n-1}\bigl(R_f^{[-4]}(t)
 +R_f^{[-3]}(t)+R_f^{[-2]}(t)+R_f^{\rm out}(t)\bigr),
\end{equation}
where
\begingroup\small
\begin{align}
 R_f^{[-4]}={}&-\frac{2^{n-4}}{n(n-2)^2(n+2)(n+4)}
 \Bigl[(n+2)(n+4)^3t^3\mathcal I_{n-4,n+1}^{11}
 +4(n+4)^2(3n+14)t^4\mathcal I_{n-4,n+3}^{12}\nonumber\\
 &\quad+8(n+4)(3n+16)t^5\mathcal I_{n-4,n+5}^{13}
 +4(7n^2+70n+176)t^5\mathcal I_{n-4,n+5}^{22}\nonumber\\
 &\quad+16(n+4)t^6\mathcal I_{n-4,n+7}^{14}
 +16(5n+28)t^6\mathcal I_{n-4,n+7}^{23}
 +32t^7\mathcal I_{n-4,n+9}^{24}
 +32t^7\mathcal I_{n-4,n+9}^{33}\Bigr],
 \label{eq:app-R-minus4}\\
 R_f^{[-3]}={}&\frac{2^{n-2}}{n(n-2)(n+2)(n+4)}
 \Bigl[2(n+4)(n+6)t^3\mathcal I_{n-3,n+3}^{02}
 +(n+4)^2(n+6)t^3\mathcal I_{n-3,n+3}^{11}\nonumber\\
 &\quad+8(n+6)t^4\mathcal I_{n-3,n+5}^{03}
 +2(n+6)(5n+26)t^4\mathcal I_{n-3,n+5}^{12}
 +8t^5\mathcal I_{n-3,n+7}^{04}\nonumber\\
 &\quad+8(2n+13)t^5\mathcal I_{n-3,n+7}^{13}
 +16(n+6)t^5\mathcal I_{n-3,n+7}^{22}
 +8t^6\mathcal I_{n-3,n+9}^{14}
 +24t^6\mathcal I_{n-3,n+9}^{23}\Bigr],
 \label{eq:app-R-minus3}\\
 R_f^{[-2]}={}&-\frac{2^{n-5}}{n(n+2)(n+4)}
 \Bigl[n(n+2)(n+4)t\mathcal I_{n-2,n-1}^{00}
 +16(n+2)(n+4)t^2\mathcal I_{n-2,n+1}^{01}\nonumber\\
 &\quad+32(n+4)t^2\mathcal I_{n-2,n+3}^{01}
 +40(n+4)t^3\mathcal I_{n-2,n+3}^{02}
 +4(n+16)(n+4)t^3\mathcal I_{n-2,n+3}^{11}\nonumber\\
 &\quad+32(n+8)t^3\mathcal I_{n-2,n+5}^{02}
 +32t^4\mathcal I_{n-2,n+5}^{03}
 +8(n^2+15n+62)t^3\mathcal I_{n-2,n+5}^{11}\nonumber\\
 &\quad+32(n+9)t^4\mathcal I_{n-2,n+5}^{12}
 +64t^4\mathcal I_{n-2,n+7}^{03}
 +64(n+9)t^4\mathcal I_{n-2,n+7}^{12}\nonumber\\
 &\quad+32t^5\mathcal I_{n-2,n+7}^{13}
 +32t^5\mathcal I_{n-2,n+7}^{22}\nonumber\\
 &\quad+64t^5\mathcal I_{n-2,n+9}^{13}
 +64t^5\mathcal I_{n-2,n+9}^{22}\Bigr],
 \label{eq:app-R-minus2}\\
 R_f^{\rm out}={}&\frac{2^{n-3}(n-1)}{n(n+2)}
 \Bigl[(n+2)t\mathcal I_{n-1,n+1}^{00}
 +4t^2\mathcal I_{n-1,n+3}^{01}
 -nt\mathcal I_{n,n+3}^{00}\Bigr].
 \label{eq:app-R-out}
\end{align}
\endgroup
The coefficient $24$ in the last term of
\eqref{eq:app-R-minus3} is essential.

The limiting projected correction makes no contribution
to the displayed scale value or center Hessian.  Indeed, its source is
\begin{equation}\label{eq:app-correction-source}
 \mathcal R_{\xi',\sqrt t}(y)
 =n(n-1)\frac{U_{\sqrt t,\xi'}(y)^{4(n-2)/(n-4)}}
 {(t+|y-\xi'|^2)^2}
 \widehat h_{ij}(y)\xi_i'\xi_j'.
\end{equation}
Thus $\mathcal R_{0,\sqrt t}=0$ and
$\partial_{\xi_p'}\mathcal R_{\xi',\sqrt t}|_{\xi'=0}=0$; the response
term is $O(|\xi'|^4)$ and hence
\begin{equation}\label{eq:app-correction-zero}
 Q_f^{\rm corr}(t)=R_f^{\rm corr}(t)=0.
\end{equation}

Finally, if $f(s)=\sum_{q=0}^{m_f}a_qs^q$, every integral in
\eqref{eq:app-P-formula}--\eqref{eq:app-R-out} is an explicit beta
integral:
\begin{align}
 \mathcal I_{\ell,m}^{\alpha\gamma}(t)
 ={}&\frac12\sum_{q=\alpha}^{m_f}
 \sum_{r=\gamma}^{m_f}a_qa_r
 \frac{q!}{(q-\alpha)!}\frac{r!}{(r-\gamma)!}
 t^{q+r-\alpha-\gamma}\nonumber\\
 &\quad\times
 B\left(\frac{m+1}{2}+q+r-\alpha-\gamma,
 \ell-\frac{m+1}{2}-q-r+\alpha+\gamma\right),
 \label{eq:app-beta-rule}
\end{align}
whenever the second beta argument is positive; terms whose derivative order
exceeds the degree of $f$ are absent.  Substituting
\eqref{eq:app-beta-rule} into
\eqref{eq:app-P-formula}--\eqref{eq:app-R-out} gives the exact finite
expansions used in Lemma~\ref{lem:step22}.


\begin{thebibliography}{99}

\bibitem{AmmannHumbertMorel}
B. Ammann, E. Humbert, and B. Morel,
\emph{Mass endomorphism and spinorial Yamabe type problems on conformally
flat manifolds}, Comm. Anal. Geom. \textbf{14} (2006), 163--182.

\bibitem{Brendle}
S. Brendle, \emph{Blow-up phenomena for the Yamabe equation},
J. Amer. Math. Soc. \textbf{21} (2008), 951--979.

\bibitem{BrendleMarques}
S. Brendle and F. C. Marques,
\emph{Blow-up phenomena for the Yamabe equation II},
J. Differential Geom. \textbf{81} (2009), 225--250.

\bibitem{CatinoMazzieri}
G. Catino and L. Mazzieri,
\emph{Connected sum construction for $\sigma_k$-Yamabe metrics},
J. Geom. Anal. \textbf{23} (2013), no.~2, 812--854.

\bibitem{CGY}
S.-Y. A. Chang, M. J. Gursky, and P. C. Yang,
\emph{An equation of Monge--Amp\`ere type in conformal geometry, and
four-manifolds of positive Ricci curvature},
Ann. of Math. \textbf{155} (2002), 709--787.

\bibitem{DruetMulti}
O. Druet, \emph{From one bubble to several bubbles: the low-dimensional
case}, J. Differential Geom. \textbf{63} (2003), 399--473.

\bibitem{DruetCompact}
O. Druet, \emph{Compactness for Yamabe metrics in low dimensions},
Int. Math. Res. Not. (2004), 1143--1191.

\bibitem{GeWang}
Y. Ge and G. Wang, \emph{On a fully nonlinear Yamabe problem},
Ann. Sci. \'{E}c. Norm. Sup\'er. \textbf{39} (2006), 569--598.

\bibitem{GuanWang}
P. Guan and G. Wang,
\emph{A fully nonlinear conformal flow on locally conformally flat
manifolds}, J. Reine Angew. Math. \textbf{557} (2003), 219--238.


\bibitem{GurskyViaclovskyVolume}
M. J. Gursky and J. A. Viaclovsky,
\emph{Volume comparison and the $\sigma_k$-Yamabe problem},
Adv. Math. \textbf{187} (2004), no.~2, 447--487.


\bibitem{GurskyViaclovsky}
M. J. Gursky and J. A. Viaclovsky,
\emph{Prescribing symmetric functions of the eigenvalues of the Ricci
tensor}, Ann. of Math. \textbf{166} (2007), 475--531.

\bibitem{JerisonLee}
D. Jerison and J. M. Lee, \emph{The Yamabe problem on CR manifolds},
J. Differential Geom. \textbf{25} (1987), 167--197.

\bibitem{KMS}
M. A. Khuri, F. C. Marques, and R. M. Schoen,
\emph{A compactness theorem for the Yamabe problem},
J. Differential Geom. \textbf{81} (2009), 143--196.

\bibitem{KMW}
S. Kim, M. Musso, and J. Wei,
\emph{A non-compactness result on the fractional Yamabe problem in large
dimensions}, J. Funct. Anal. \textbf{273} (2017), 3759--3830.

\bibitem{LiLi1}
A. Li and Y. Y. Li,
\emph{On some conformally invariant fully nonlinear equations},
Comm. Pure Appl. Math. \textbf{56} (2003), 1416--1464.

\bibitem{LiLi2}
A. Li and Y. Y. Li,
\emph{On some conformally invariant fully nonlinear equations, II:
Liouville, Harnack and Yamabe}, Acta Math. \textbf{195} (2005), 117--154.


\bibitem{LiNguyenCompactness}
Y. Y. Li and L. Nguyen,
\emph{A compactness theorem for a fully nonlinear Yamabe problem under a lower Ricci curvature bound},
J. Funct. Anal. \textbf{266} (2014), no.~6, 3741--3771.

\bibitem{LiNguyenWangNirenberg}
Y. Y. Li, L. Nguyen, and B. Wang,
\emph{On the $\sigma_k$-Nirenberg problem},
Amer. J. Math. \textbf{146} (2024), no.~1, 241--276.


\bibitem{LiXiong}
Y. Y. Li and J. Xiong,
\emph{Compactness of conformal metrics with constant $Q$-curvature. I},
Adv. Math. \textbf{345} (2019), 116--160.

\bibitem{LiZhang1}
Y. Y. Li and L. Zhang,
\emph{A Harnack type inequality for the Yamabe equation in low dimensions},
Calc. Var. Partial Differential Equations \textbf{20} (2004), 133--151.

\bibitem{LiZhang2}
Y. Y. Li and L. Zhang,
\emph{Compactness of solutions to the Yamabe problem. II},
Calc. Var. Partial Differential Equations \textbf{24} (2005), 185--237.

\bibitem{LiZhang3}
Y. Y. Li and L. Zhang,
\emph{Compactness of solutions to the Yamabe problem. III},
J. Funct. Anal. \textbf{245} (2007), 438--474.

\bibitem{LiZhu}
Y. Y. Li and M. Zhu,
\emph{Yamabe type equations on three-dimensional Riemannian manifolds},
Comm. Contemp. Math. \textbf{1} (1999), 1--50.

\bibitem{Marques}
F. C. Marques,
\emph{A priori estimates for the Yamabe problem in the non-locally
conformally flat case}, J. Differential Geom. \textbf{71} (2005), 315--346.

\bibitem{STW}
W. Sheng, N. S. Trudinger, and X.-J. Wang,
\emph{The Yamabe problem for higher order curvatures},
J. Differential Geom. \textbf{77} (2007), 515--553.

\bibitem{Viaclovsky}
J. A. Viaclovsky,
\emph{Conformal geometry, contact geometry, and the calculus of variations},
Duke Math. J. \textbf{101} (2000), 283--316.

\bibitem{WeiZhao}
J. Wei and C. Zhao,
\emph{Non-compactness of the prescribed $Q$-curvature problem in large
dimensions}, Calc. Var. Partial Differential Equations \textbf{46} (2013),
123--164.

\end{thebibliography}
\end{document}